\documentclass[10pt, oneside]{article}
\usepackage{mathrsfs}
\usepackage{amsfonts}
\usepackage{amssymb} 
\usepackage{amsmath,amsfonts,amssymb,amsthm}
\usepackage{caption}
\usepackage{subfigure}
\usepackage{cite}
\usepackage{color}
\usepackage{graphicx}
\usepackage{subfigure}
\usepackage{float}
\usepackage{paralist}
\usepackage{indentfirst}
\usepackage{authblk}

\usepackage[colorlinks,
linkcolor=blue,
citecolor=blue
]{hyperref}

\usepackage[T1]{fontenc}
\def\div{ \hbox{\rm div}\,  }

\newtheorem{theorem}{Theorem}[section]
\newtheorem{lemma}{Lemma}[section]

\newtheorem{cor}{Corollary}[section]
\newtheorem{remark}{Remark}[section]

\def\bma#1\ema{{\allowdisplaybreaks\begin{aligned}#1\end{aligned}}}

\def\var{\varepsilon}

\date{}
\numberwithin{equation}{section}

\begin{document}
\title{{\LARGE \textbf{The hyperbolic-parabolic chemotaxis system for vasculogenesis: global dynamics and relaxation limit toward a Keller-Segel model}}}

\author[a]{Timothée Crin-Barat \thanks{
E-mail: timotheecrinbarat@gmail.com (T. Crin-Barat).}}

\author[b,c]{Qingyou He  \thanks{
E-mail:		qyhe.cnu.math@qq.com (Q. He).}}

\author[b,c]{Ling-Yun Shou \thanks{E-mail: shoulingyun11@gmail.com (L.-Y Shou).
\\The authors: Timothée Crin-Barat, Qingyou He and Ling-Yun Shou claim that the work is not currently submitted to other journals and that it will not be submitted to other journals during the reviewing process.}}

\affil[a]{Chair of Computational Mathematics, Fundación Deusto, Avenida de las Universidades 24, 48007 Bilbao, Basque Country, Spain}
    \affil[b]{School of Mathematical Sciences,
	Capital Normal University, Beijing 100048, P.R. China}
\affil[c]{Academy for Multidisciplinary Studies, Capital Normal University, Beijing 100048, P.R. China}

\renewcommand*{\Affilfont}{\small\it}
\maketitle
\begin{abstract}
 An Euler-type hyperbolic-parabolic system of chemotactic aggregation describing the vascular network formation is investigated in the critical regularity setting. For initial data near a constant equilibrium state, the global well-posedness of the classical solution to the Cauchy problem with general pressure laws is proved in critical hybrid Besov spaces, and qualitative regularity estimates uniform with respect to the relaxation parameter are established. Then, the optimal time-decay rates of the global solution are analyzed under an additional regularity assumption on the initial data.
Furthermore, the relaxation limit of the hyperbolic-parabolic system toward a parabolic–elliptic Keller-Segel model is justified rigorously. It is shown that as the relaxation parameter tends to zero, the solutions of the hyperbolic-parabolic chemotaxis system converges to the solutions of the Keller-Segel model at an explicit rate of convergence. Our approach relies on the introduction of new effective unknowns in low frequencies and the construction of a Lyapunov functional in the spirit of that of Beauchard and Zuazua's in \cite{bea1} to treat the high frequencies. 
\end{abstract}

\section{Introduction}

 Chemotaxis is a phenomenon describing the influence of environmental chemical substances on the motion of  various cells. It plays an important role in many biological process, such as cells aggregations (amoebae, bacteria, etc), embryonic development, cancer growth, vascular networks and so on, cf. \cite{murray1,murray2,li01,hel1}. Vasculogenesis, the process of new blood vessels formation led by chemotaxis, has complex and varied mechanisms as shown recently in numerous experimental investigations. Understanding the formation of blood vessels in organisms is a long-standing problem in the life sciences, cf. \cite{amb1,amb2,ch0,gamba1,hel1}.

In this paper, we consider the following $d$-dimensional  ($d\geq1$) hyperbolic-parabolic chemotaxis (HPC) system modelling vasculogenesis:
\begin{equation}\label{m1}\tag{HPC}
\left\{
\begin{aligned}
& \partial_{t}\rho+\div (\rho u)=0,\\
& \partial_{t}(\rho u)+\div (\rho u\otimes u)+\nabla P(\rho)+\frac{1}{\var} \rho u-\mu\rho \nabla \phi=0,\\
&\partial_{t}\phi-\Delta \phi-a\rho+b\phi=0,\quad\quad\quad \quad x\in\mathbb{R}^{d},\quad t>0,\\
\end{aligned}
\right.
\end{equation}
where $\rho=\rho(x,t)\geq0$ and $u=u(x,t)\in\mathbb{R}^d$ are, respectively, the cell density and the cell velocity, and $\phi=\phi(x,t)\geq0$ stands for the
concentration of the chemoattractant secreted by cells. Furthermore, the constant $\mu>0$ measures the intensity of cell response to the chemoattractant, $\displaystyle\frac{1}{\var}\rho u$ is a damping (friction) term due to the interaction between cells and the underlying substratum with the relaxation parameter $\var>0$, $a>0$ and $b>0$ are two constants denoting the growth and death rates of the chemoattractant, respectively, and the pressure $P(\rho)$ is assumed to be a smooth function depending only on the density.

System \eqref{m1} was recently developed by Gamba et.al. in\cite{gamba1} and by Ambrosi, Bussolino and Preziosi in \cite{amb1} to describe blood vessel networks from in vitro observations and has also been formally derived by Chavanis and Sire in \cite{ch1} from a nonlinear mean-field Fokker-Planck equation.

Without the effect of chemotaxis, System \eqref{m1} becomes the isentropic damped compressible Euler equations
\begin{equation}\label{euler}\tag{DCE}
\left\{
\begin{aligned}
& \partial_{t}\rho+\div (\rho u)=0,\\
& \partial_{t}(\rho u)+\div (\rho u\otimes u)+\nabla P(\rho)+\frac{1}{\var} \rho u=0,\quad\quad x\in\mathbb{R}^{d},\quad t>0.
\end{aligned}
\right.
\end{equation}
Recently, many advances have been made in the 
analysis of global solutions to System \eqref{euler}. For small initial data in the Sobolev spaces $H^{s}$ ($s>\frac{d}{2}+1$), the global well-posedness and time-asymptotical behaviors of classical solutions for \eqref{euler} have been studied in many significant works \cite{daf1, nas1,si1,wang1}. In particular, it was shown in \cite{si1,wang1} that the damping term $\rho u$ can prevent the formation of shock waves that would occurs in finite time without this term. Then, using the theory of partially dissipative hyperbolic system, Xu and Kawashima  in \cite{xu1} showed that System~(\ref{euler}) is globally well-posed for small initial erturbations in the inhomogeneous Besov space $B^{\frac{d}{2}+1}_{2,1}$. Very recently, inspired by the works \cite{hoff1,haspot1} of Hoff and Haspot for the compressible Navier-Stokes equations, the first author and Danchin in \cite{c1,c2} considered a new effective velocity (also called damped mode) in low frequencies to recover new and useful properties on the solution. Then, motivated by Beauchard and Zuazua's \textit{hyperbolic hypocoercivity} approach \cite{bea1}, they constructed a Lyapunov functional including lower-order terms to handle the high frequencies. These two elements allowed them to establish the global well-posedness and optimal time-decay rates of global solutions to (\ref{euler}) for small initial perturbations in the critical homogeneous Besov space $\dot{B}^{\frac{d}{2}}_{2,1}\cap\dot{B}^{\frac{d}{2}+1}_{2,1}$. Moreover, in  \cite{c3}, they established the global well-posedness in $L^p$ type hybrid space $\dot{B}^{\frac{d}{p},\frac{d}{2}+1}_{p,2}$ for some $2 \leq p \leq 4$. In addition, many important results  have been shown concerning the global dynamics of more general hyperbolic systems or hyperbolic-parabolic systems, refer to \cite{c2,c3,BH,liu1,kaw1,shi1,xu1,xu2,yong1} and references therein. To have a better understanding of this sequence of improvements, one must keep in mind the following embeddings:
$$
H^{s} (s>\frac{d}{2}+1)\hookrightarrow B^{\frac{d}{2}+1}_{2,1}\hookrightarrow \dot{B}^{\frac{d}{2}}_{2,1}\cap \dot{B}^{\frac{d}{2}+1}_{2,1}\hookrightarrow \dot{B}^{\frac{d}{p},\frac{d}{2}+1}_{p,2} (p>2)\hookrightarrow \mathcal{C}^1_b.
$$

It is in light of these recent advances that we study System \eqref{m1} here. First, we show the global well-posedness and optimal decay rates of solutions to System (\ref{m1}) close to constant equilibrium states in $\dot{B}^{\frac{d}{2}}_{2,1}\cap \dot{B}^{\frac{d}{2}+1}_{2,1}$. Our study is in line with the work \cite{d0} of Danchin concerning the study of the compressible Navier-Stokes system in critical spaces. It is motivated by the fact that this space appears to be the largest space in which one is able to justify the existence of unique global-in-time solution for systems with Euler-like structures. Indeed, it is well-known that a Lipschitz ($\dot{W}^{1,\infty}$)-bound on the solution, which can be obtained  via the end-point embedding $ \dot{B}^{\frac{d}{2}+1}_{2,1}\hookrightarrow \dot{W}^{1,\infty}$, allows to avoid the finite time blow-up for the compressible Euler equations and other hyperbolic models (cf. \cite{daf1,wang1,xu1,BCBT}). We also mention the work \cite{li010} where ill-posedness results for the Burgers equation are obtained for initial data belonging to $H^{s}$ with $s<d/2+1$. 
\medbreak
Next, we focus on the relaxation limit of the System (\ref{m1}). Let us recall some important works on the relaxation limit or large-time asymptotics of profiles for System~\eqref{euler}. Considering the diffusive rescaling $\displaystyle (\rho^{\var},u^{\var})(x,\tau):=(\rho,\frac{u}{\var})(x,t)$ for $\tau=\var t$, System \eqref{euler} becomes
\begin{equation}\label{DCEvar}\tag{${\rm{DCE}}_{\var}$}
\left\{
\begin{aligned}
&\partial_{t}\rho^{\var}+\div (\rho^{\var} u^{\var})=0,\\
&\var^2 \partial_{t}(\rho^{\var}u^{\var})+\var^2\div (\rho^{\var}u^{\var}\otimes u^{\var})+\nabla P(\rho^{\var})+\rho^{\var}u^{\var}=0.
\end{aligned}
\right.
\end{equation}
As $\var\rightarrow0$, the dynamics of System (\ref{DCEvar}) is formally governed by the porous media model
 \begin{equation}\label{pm}\tag{PM}
 \left\{
 \begin{aligned}
 &\partial_t\rho^{*}-\Delta P(\rho^{*})=0,\\
 & \rho^{*}u^{*}=-\nabla P(\rho^{*})\quad\quad ({\rm Darcy's~law}).
 \end{aligned}
 \right.
 \end{equation}
The rigorous derivation of the relaxation limit of (\ref{euler}) to (\ref{pm}) was studied in \cite{mar0,cou1,mar1,junca1,xu00}, and the time-asymptotical convergence of global solutions with vacuum for the system~(\ref{euler}) to the Barenblatt profile of (\ref{pm}) was made in \cite{huang1,luo1}. In \cite{c3}, the authors justified rigorously the relaxation limit of (\ref{euler}) and also derived an explicit convergence rate of the process in the multi-dimensional setting. The key point of their study is a spectral analysis of the system which reveals that the suitable threshold between the low and high frequencies to study the system should be $J_\varepsilon=\left\lfloor\rm -log_2\,\varepsilon\right\rfloor+k$ with $k$ a small enough constant. This threshold corresponds to the place where the 0-order terms and the 1-order terms have the same weight (parameter included) and allow them to recover suitable uniform bounds for such relaxation limit. It is this approach that we are going to adapt to our system here.

   A natural and important issue is to investigate the influences of the chemoattractant concentration $\phi$ secreted by cells on the well-posedness and the qualitative
behavior of solutions to System \eqref{m1}.  From the mathematical viewpoint, the main difficulties about this system are that it does not satisfy the usual symmetric conditions of the classical hyperbolic-parabolic systems nor the Shizuta-Kawashima (SK) condition, see \cite{liu1,kaw1,shi1,yong1}. Therefore, even though it seems to have "enough dissipation", the classical methods for partially dissipative systems can not be applied to System (\ref{m1}). These difficulties arise mainly from the term $\mu\rho\nabla\phi$ which forces us to rethink how to treat the low and high-frequency regimes. 
\medbreak

There is an extensive literature on the study of System \eqref{m1} \cite{fil1,nat1,nat2,russo1,russo2,ber1,hong1,kow1,liuq1,liuq2}.
Among them, the numerical simulations of System \eqref{m1} can be found in \cite{fil1,nat1,nat2,russo1}. For the one-dimensional case, the existence and nonlinear time-stability of steady states of (\ref{m1}) were established in \cite{ber1,hong1}, and the convergence to nonlinear diffusion waves of (\ref{m1}) in $\mathbb{R}$ was proved in \cite{liuq1}. As for the
multi-dimensional case, the linear stability of solutions of (\ref{m1}) was obtained  in \cite{kow1} under the assumption $P'(\bar{\rho})>\dfrac{a\mu}{b}\bar{\rho}$ with an additional viscous term $\Delta u$ added in ${\rm(\ref{m1})_{2}}$. The existence and time-decay estimates of global solutions of (\ref{m1}) around a constant equilibrium state $(\bar{\rho},0,\dfrac{a}{b}\bar{\rho})$ in the Sobolev spaces $H^{s}$ $(s>\frac{d}{2}+1)$ have been proved by Russo and Sepe in \cite{russo1,russo2} under a smallness assumption on the background density $\bar{\rho}>0$. Very recently, the convergence to linear diffusion waves \eqref{m1} with Sobolev regularity in three dimensions was derived by Liu, Peng and Wang in \cite{liuq2}.



\medbreak
Concerning the justification of the relaxation limit (or large friction limit) of System (\ref{m1}) as $\var\rightarrow0$, again we introduce the rescaled time variable $\tau=\var t$ and define
\begin{align}
(\rho^{\var},u^{\var},\phi^{\var})(x,\tau):=(\rho,\frac{u}{\var},\phi)(x,t)\label{scaling1}.
\end{align}
Then $(\rho^{\var},u^{\var},\phi^{\var})$ satisfies
\begin{equation}\label{HPCvar}\tag{${\rm{HPC}}_{\var}$}
\left\{
\begin{aligned}
&\partial_{t}\rho^{\var}+\div (\rho^{\var} u^{\var})=0,\\
&\var^2 \partial_{t}(\rho^{\var}u^{\var})+\var^2\div (\rho^{\var}u^{\var}\otimes u^{\var})+\nabla P(\rho^{\var})-\mu\rho^{\var}\nabla \phi^{\var}+\rho^{\var}u^{\var}=0,\\
&\var \partial_{t}\phi^{\var}-\Delta \phi^{\var}-a\rho^{\var}+b\phi^{\var}=0.
\end{aligned}
\right.
\end{equation}
 At the formal level, as $\var\rightarrow0$, the solution $(\rho^{\var},u^{\var},\phi^{\var})$ of System (\ref{HPCvar}) converges to the limit $(\rho^{*},u^{*},\phi^{*})$ solving the parabolic-elliptic Keller-Segel type model
\begin{equation}\label{KS}\tag{KS}
\left\{
\begin{aligned}
&\partial_{t}\rho^{*}-\div (\nabla P(\rho^*)-\mu\rho^*\nabla\phi^*)=0,\\
&\rho^{*}u^{*}=-\nabla P(\rho^{*})+\mu\rho^{*} \nabla \phi^{*},\\
&-\Delta \phi^{*}-a\rho^{*}+b\phi^{*}=0.
\end{aligned}
\right.
\end{equation}
Keller-Segel type models are widely used in biology to describe the collective motion of cells or the evolution of a density of bacteria in models involving chemotaxis, see \cite{aru1,keller1,perthame1,murray1} and references therein.
 
The momentum formula $({\rm\ref{KS}})_{2}$ is analogous to the  Darcy's law $({\rm\ref{pm}})_{2}$. To our knowledge, the only result about the relaxation limit from (\ref{m1}) to (\ref{KS}) is due to Francesco and Donatelli in \cite{di1} where they justified the local-in-time convergence of solutions for (\ref{m1}) to (\ref{KS}) in a two-dimensional periodic domain. Let us also mention the recent work by Lattanzio and Tzavaras \cite{la1} in which they justified the relaxation limit from the damped Euler-Poisson equations (i.e. without $\partial_{t}\phi$ in $\eqref{m1}_3$) to (\ref{KS}).

\medbreak
 The purpose of the present paper is to study the Cauchy problem for System~(\ref{m1}) in a critical regularity framework. First, under the natural stability assumption $P'(\bar{\rho})>\dfrac{a\mu}{b}\bar{\rho}$, we prove the global existence and uniqueness of classical solutions around the constant equilibrium state $(\bar{\rho},0,\dfrac{a}{b}\bar{\rho})$ in $L^{2}$-type hybrid Besov spaces with different regularity exponents
in low and high frequencies. Then, we derive the optimal time-decay rates of global solutions under an additional regularity assumption (weaker than $L^1$) on the initial perturbation.
And afterward, we provide a rigorous justification of the strong relaxation limit from (\ref{HPCvar}) to (\ref{KS}) and exhibit an explicit convergence rate of the process.

\vspace{2ex}

The rest of the paper is organized as follows. In Section \ref{sectionmain}, we state the main results and explain the strategies of this paper. In Section \ref{section3}, we establish a-priori estimates and prove
Theorem \ref{theorem11} concerning the global well-posedness for System \eqref{m1}. In Section \ref{section5}, we prove Theorem \ref{theorem14} concerning the optimal time-decay estimates. Section \ref{section4} is devoted to the proof of Theorems \ref{theorem12} and \ref{theorem13} about the global well-posedness for the Keller-Segel model and the relaxation limit process. Section \ref{section6} is an appendix regarding hybrid Besov spaces and tools that are used throughout the paper.

\section{Main results and strategy of proof}\label{sectionmain}

\subsection{Main results}

In this paper, the solution $(\rho,u,\phi)$ to System (\ref{m1}) is obtained as a small perturbation of the equilibrium state
\begin{equation}
\begin{aligned}
(\bar{\rho},0,\bar{\phi}),\label{equ}
\end{aligned}
\end{equation}
with the constant background density $\bar{\rho}>0$ and the constant background concentration $\bar{\phi}:=\dfrac{a}{b}\bar{\rho}>0$. We assume that the pressure law $P(\rho)$ satisfies
\begin{equation}\nonumber P'(\rho)>0\ \text{ for }\ \rho\ \text{ close to }\  \bar{\rho}.\end{equation}
Then, introducing the variables 
\begin{equation}\nonumber
\begin{aligned}
n:=\int_{\bar{\rho}}^{\rho} \frac{P'(s)}{s}ds,\quad\quad  \psi:=\phi-\bar{\phi},
\end{aligned}
\end{equation}
 System (\ref{m1}) can be reformulated as
\begin{equation}\label{m1n}
\left\{
\begin{aligned}
&\partial_{t}n+u\cdot\nabla n +c_{0}\div u+G(n)\div u=0,\\
&\partial_{t}u+u\cdot \nabla u+\frac{1}{\var} u+\nabla n-\mu\nabla \psi=0,\\
&\partial_{t}\psi-\Delta \psi+b\psi-c_{1}n- H(n)=0,\quad \quad x\in\mathbb{R}^{d},\quad t>0,
\end{aligned}
\right.
\end{equation}
where the constants $c_{i}$ $(i=1,2),G(n)$ and $H(n)$ are defined by
\begin{align}
&c_{0}:=P'(\bar{\rho}),\quad\quad c_{1}:=a\partial_{n}\rho|_{n=0}=\frac{a\rho}{P'(\rho)}\Big{|}_{n=0}=\frac{a\bar{\rho}}{P'(\bar{\rho})},\nonumber
\end{align}
and 
\begin{align}
G(n):=P'(\rho)-P'(\bar{\rho}),\quad\quad H(n):=a\big{(} \rho-\bar{\rho}- \frac{\bar{\rho}}{P'(\bar{\rho})}n \big{ )} .\nonumber
\end{align}

First, we state our result concerning the uniform-in-$\varepsilon$ global well-posedness to the Cauchy problem for System (\ref{m1n}) in hybrid Besov spaces. For this purpose, we denote the total energy
\begin{equation}\label{Xdef}
\begin{aligned}
&\mathcal{X}(t):=\mathcal{X}_{L}(t)+\mathcal{X}_{H}(t),
\end{aligned}
\end{equation}
with the low-frequency energy 
\begin{align*}
&\mathcal{X}_{L}(t):=\|(n,u,\psi)\|_{\widetilde{L}^{\infty}_{t}(\dot{B}^{\frac{d}{2}}_{2,1})}^{\ell}+\var\|(n,\psi)\|_{L^1_{t}(\dot{B}^{\frac{d}{2}+2}_{2,1})}^{\ell}\\
&\quad\quad\quad~+\|u\|_{L^1_{t}(\dot{B}^{\frac{d}{2}+1}_{2,1})}^{\ell}+\var^{-\frac{1}{2}}\|u\|_{\widetilde{L}^{2}_{t}(\dot{B}^{\frac{d}{2}}_{2,1})}^{\ell}+\|\partial_{t}\psi\|_{L^1_{t}(\dot{B}^{\frac{d}{2}}_{2,1})}^{\ell}\nonumber\\
&\quad\quad\quad~+\|\frac{1}{\var} u+\nabla n-\mu \nabla\psi\|_{L^1_{t}(\dot{B}^{\frac{d}{2}}_{2,1})}^{\ell}+\|\psi -(b-\Delta)^{-1}(c_{1}n+H(n) )\|_{L^1_{t}(\dot{B}^{\frac{d}{2}}_{2,1}\cap\dot{B}^{\frac{d}{2}+2}_{2,1})}^{\ell},
\end{align*}
and the high-frequency energy
\begin{align*}
&\mathcal{X}_{H}(t):=\var\|(n,u,\nabla\psi)\|_{\widetilde{L}^{\infty}_{t}(\dot{B}^{\frac{d}{2}+1}_{2,1})}^{h}+ \|(n,u)\|_{L^{1}_{t}(\dot{B}^{\frac{d}{2}+1}_{2,1})}^{h}+\|\psi\|_{L^{1}_{t}(\dot{B}^{\frac{d}{2}+3}_{2,1})}^{h}\\
&\quad\quad\quad~+\var^{-\frac{1}{2}}\|u\|_{\widetilde{L}^{2}_{t}(\dot{B}^{\frac{d}{2}}_{2,1})}^{h}+\|\partial_{t}\psi\|_{L^1_{t}(\dot{B}^{\frac{d}{2}+1}_{2,1})}^{h}
\end{align*}
where the notations related to the Besov functional framework are defined in Subsection \ref{sec:func-spaces}.
\begin{theorem} \label{theorem11} {\rm(}Global well-posedness for the HPC system{\rm)}
Let $d\geq1$, $0<\var<1$,
\begin{equation}\label{ABmu}
\begin{aligned}
P'(\bar{\rho})>\frac{a\mu}{b}\bar{\rho}>0
\end{aligned}
\end{equation}
\vspace{-0.3cm}
and \vspace{-0.3cm}
\begin{equation} \label{J0}
    \begin{aligned}
J_{\var}=[-\log_2{\var}]+k
\end{aligned}
\end{equation}
for a sufficiently small integer $k$ independent of $\var$. There exists a positive constant $\alpha$ independent of $\var$ such that if $(n_{0},u_{0})\in\dot{B}^{\frac{d}{2}}_{2,1}\cap \dot{B}^{\frac{d}{2}+1}_{2,1}$, $\psi_{0}\in \dot{B}^{\frac{d}{2}}_{2,1}\cap \dot{B}^{\frac{d}{2}+2}_{2,1}$ and
\begin{equation}\label{a1}
\begin{aligned}
\mathcal{X}_{0}:=\|(n_{0},u_{0},\psi_{0})\|_{\dot{B}^{\frac{d}{2}}_{2,1}}^{\ell}+\var\|(n_{0},u_{0},\nabla\psi_{0})\|_{\dot{B}^{\frac{d}{2}+1}_{2,1}}^{h} \leq \alpha,
\end{aligned}
\end{equation}
then System {\rm(\ref{m1n})} supplemented with the initial data $(n_{0},u_{0},\psi_{0})$ admits a unique global-in-time solution $(n,u,\psi)$ satisfying
\begin{equation}\label{r1}
\left\{
\begin{aligned}
&n^{\ell}\in \mathcal{C}_{b}(\mathbb{R}^{+};\dot{B}^{\frac{d}{2}}_{2,1})\cap L^1(\mathbb{R}^{+};\dot{B}^{\frac{d}{2}+2}_{2,1}),\quad\quad n^{h}\in \mathcal{C}_{b}(\mathbb{R}^{+};\dot{B}^{\frac{d}{2}+1}_{2,1})\cap L^1(\mathbb{R}^{+};\dot{B}^{\frac{d}{2}+1}_{2,1}),\\
&u^{\ell}\in \mathcal{C}_{b}(\mathbb{R}^{+};\dot{B}^{\frac{d}{2}}_{2,1})\cap L^1(\mathbb{R}^{+};\dot{B}^{\frac{d}{2}+1}_{2,1}),\quad\quad u^{h}\in \mathcal{C}_{b}(\mathbb{R}^{+};\dot{B}^{\frac{d}{2}+1}_{2,1})\cap L^1(\mathbb{R}^{+};\dot{B}^{\frac{d}{2}+1}_{2,1}),\\
&\psi^{\ell}\in \mathcal{C}_{b}(\mathbb{R}^{+};\dot{B}^{\frac{d}{2}}_{2,1})\cap L^1(\mathbb{R}^{+};\dot{B}^{\frac{d}{2}+2}_{2,1}),\quad \quad\psi^{h}\in \mathcal{C}_{b}(\mathbb{R}^{+};\dot{B}^{\frac{d}{2}+2}_{2,1})\cap L^1(\mathbb{R}^{+};\dot{B}^{\frac{d}{2}+3}_{2,1}),\\
\end{aligned}
\right.
\end{equation}
and
\begin{equation}\label{XX0}
\begin{aligned}
\mathcal{X}(t)\leq C\mathcal{X}_{0},\quad\text{for }  t>0,
\end{aligned}
\end{equation}
where the constant $C>0$  is independent of the time $t$ and $\var$.

\end{theorem}



\begin{remark}
In Theorem \ref{theorem11}, we exhibit qualitative and uniform regularity properties of the solution $(n,u,\psi)$ of System {\rm(\ref{m1n})} in the low-frequency region $|\xi|\lesssim \frac{1}{\var}$ and high-frequency region $|\xi|\gtrsim \frac{1}{\var}$ separately. These properties are consistent with the spectral analysis of \eqref{m1n} done in Subsection \ref{subsections}. Moreover, the spectral analysis also shows that $(\ref{ABmu})$ is a natural hypothesis to ensure the stability of the associated linear system.
\end{remark}

\begin{remark}
Since the map $\displaystyle\rho\rightarrow \int^{\rho}_{\bar{\rho}}\frac{P'(s)}{s}ds$ is 
 diffeomorphism from $\mathbb{R}_{+}$ to a small neighborhood of $0$, by a direct composition estimate one can show that
\begin{equation}\nonumber
\begin{aligned}
\rho-\bar{\rho}, u\in \mathcal{C}(\mathbb{R}_{+}; \dot{B}^{\frac{d}{2}}_{2,1}\cap\dot{B}^{\frac{d}{2}+1}_{2,1})\hookrightarrow \mathcal{C}(\mathbb{R}_{+};\mathcal{C}^1),\quad \phi -\bar{\phi} \in \mathcal{C}(\mathbb{R}_{+}; \dot{B}^{\frac{d}{2}}_{2,1}\cap\dot{B}^{\frac{d}{2}+2}_{2,1})\hookrightarrow \mathcal{C}(\mathbb{R}_{+};\mathcal{C}^2).
\end{aligned}
\end{equation}
Thus, Theorem \ref{theorem11} also provides the unique global-in-time classical solution $(\rho ,u,\phi)$ to System {\rm{(\ref{m1})}}.
\end{remark}

\begin{remark}
The constant $k$ in {\rm{(\ref{J0})}} can be computed explicitly in terms of $d$, $\mu$, $a$, $b$ and $P$. It is chosen small enough in the computations so as to absorb bad linear terms in the low-frequency regime.
\end{remark}

\begin{remark}
It is shown in {\rm(\ref{XX0})}  that the effective velocity $u+\var\nabla n-\var\mu \nabla\psi$ and the effective concentration $\psi -(b-\Delta)^{-1}(c_{1}n+H(n))$ have better regularity properties than the whole solution. This comes from the fact that they are purely damped mode, as it can be osberved in $(\ref{Ga})_{2}$-$(\ref{Ga})_{3}$. These properties are crucial to derive uniform-in-$\var$ estimates and the convergence rate of the relaxation process in Theorem \ref{theorem13}.
\end{remark}

Next, we study the optimal time-decay rates of global solutions to System (\ref{m1n}).
\begin{theorem}\label{theorem14} {\rm(}Optimal time-decay rates for the HPC system{\rm)}
Assume $d\geq1$,  $0<\var<1$ and that {\rm{(\ref{ABmu})}}-{\rm{(\ref{a1})}} hold. Let $(n,u,\psi)$ be the corresponding global solution to System {\rm(\ref{m1n})} supplemented with the initial data $(n_{0},u_{0},\psi_{0})$ given by Theorem \ref{theorem11}. If it further holds
\begin{equation}\label{a3}
\begin{aligned}
(n_{0},  u_{0}, \psi_{0})^{\ell}\in \dot{B}^{\sigma_{0}}_{2,\infty}\quad\text{for}~\sigma_{0}\in[-\frac{d}{2},\frac{d}{2})\quad \text{uniformly in}~\var,
\end{aligned}
\end{equation}
 then for some constant $C>0$ independent of $\var$ and time, $(n,u,\psi)$ satisfies
 \begin{equation}\nonumber
\begin{aligned}
&\|(n,u,\psi)\|_{\widetilde{L}^{\infty}_{t}(\dot{B}^{\sigma_{0}}_{2,\infty})}^{\ell}\leq C(\mathcal{X}_{0}+\|(n_{0},  u_{0}, \psi_{0})\|_{\dot{B}^{\sigma_{0}}_{2,\infty}}^{\ell}),\quad\quad t>0,
\end{aligned}
\end{equation}
 and
\begin{equation}\label{decay1}
\left\{
\begin{aligned}
&\|(n,u, \psi)(t)\|_{\dot{B}^{\sigma}_{2,1}}\leq C(1+ \var t)^{-\frac{1}{2}(\sigma-\sigma_{0})},\quad\quad \sigma\in(\sigma_{0},\frac{d}{2}],\\
&\|(n,u, \nabla \psi)(t)\|_{\dot{B}^{\frac{d}{2}+1}_{2,1}}^{h} \leq \frac{C}{\var} (1+ \var t)^{-\frac{1}{2}(\frac{d}{2}-\sigma_{0})}.
\end{aligned}
\right.
\end{equation}
Furthermore, if $d\geq2$ and $\sigma_{0}\in[-\frac{d}{2},\frac{d}{2}-1)$, then the following time-decay estimates hold$:$
\begin{equation}\label{decay3}
\left\{
\begin{aligned}
&\|u(t)\|_{\dot{B}^{\sigma_{0}}_{2,\infty}}+\|(b\psi-c_{1}n-H(n))(t)\|_{\dot{B}^{\sigma_{0}}_{2,\infty}}\leq \frac{C}{\var}(1+\var t)^{-\frac{1}{2}},\\
&\|u(t)\|_{\dot{B}^{\sigma}_{2,1}}+\|(b\psi-c_{1}n-H(n))(t)\|_{\dot{B}^{\sigma}_{2,1}}\leq \frac{C}{\var}(1+\var t)^{-\frac{1}{2}(1+\sigma-\sigma_{0})},\quad\quad \sigma\in(\sigma_{0},\frac{d}{2}-1].
\end{aligned}
\right.
\end{equation}
\end{theorem}

\begin{remark}
Due to the embedding $L^1\hookrightarrow \dot{B}^{-\frac{d}{2}}_{2,\infty}$, the above assumption {\rm(\ref{a3})} with $\sigma_{0}=-\frac{d}{2}$ covers the classical $L^1$ condition presented first by Matsumura and Nishida in {\rm\cite{mats1}}. Moreover, it should be noted that we do not require additional $\dot{B}^{\sigma_{0}}_{2,\infty}$-smallness assumption on the initial data. This is an improvement compared to previous researches in similar settings e.g. {\rm\cite{DanchinXu}} where such smallness assumption is used to derive decay rates. Our approach is based on new time-weighted energy
estimates and interpolation arguments, which is an adaptation of the previous approaches in {\rm\cite{guo1,xin1}} and enables us to deal with mixed $L^1$-in-time and $L^2$--in-time dissipation structures with different variables.
\end{remark}

\begin{remark}
Applying standard interpolation inequalities to $(\ref{decay1})$-$(\ref{decay3})$, for $r\geq 2$, we have the following classical $L^{r}$-type time-decay rates of the global solution $(\rho,u,\phi)$ to System {\rm(\ref{m1})}$:$
\begin{equation}\label{re1}
\begin{aligned}
&\|\Lambda^{\sigma} (\rho-\bar{\rho},u,\phi-\bar{\phi})(t)\|_{L^{r}}\leq C (1+\var t)^{-\frac{1}{2}(\sigma+\frac{d}{2}-\frac{d}{r}-\sigma_{0})},\quad\quad\sigma+\frac{d}{2}-\frac{d}{r}\in(\sigma_{0},\frac{d}{2}].
\end{aligned}
\end{equation}
Moreover, as $b\phi-a\rho=b\psi-c_{1}n-H(n)$, it holds for $d\geq2$ and $\sigma_{0}\in[-\frac{d}{2},\frac{d}{2}-1)$ that
\begin{equation}\label{re2}
\begin{aligned}
&\|\Lambda^{\sigma}u(t)\|_{L^{r}}+\|\Lambda^{\sigma}(b\phi-a\rho)(t)\|_{L^{r}}\leq \frac{C}{\var} (1+\var t)^{-\frac{1}{2}(1+\sigma+\frac{d}{2}-\frac{d}{r}-\sigma_{0})},\quad~\sigma+\frac{d}{2}-\frac{d}{r}\in(\sigma_{0},\frac{d}{2}-1].
\end{aligned}
\end{equation}
 Therefore, the time-decay estimates $(\ref{re1})$-$(\ref{re2})$ generalize the previous result of decay rates of global solutions to System {\rm(\ref{m1})} in {\rm\cite{liuq2,russo2}} subject to $L^1$-assumption on the initial perturbation.
\end{remark}

\begin{remark}
Theorem \ref{theorem14} implies that as $t\rightarrow\infty$, the density $\rho$ and the velocity $u$ converge to $\bar{\rho}$ and $0$, respectively, at the same decay rates as the global solutions of compressible Euler equations with damping, refer to {\rm{\cite{c1,c2,si1,xu2}}} and references therein. In addition, it is proved that the coupling term $b\phi-a\rho$ decays faster than the chemoattractant production $a\rho$ and the chemoattractant consumption $-b\phi$, which, as far as we know, is a new phenomenon observed in this paper.
\end{remark}

\vspace{1ex}

Next, we justify rigorously the relaxation limit from (\ref{HPCvar}) to (\ref{KS}) with an explicit convergence rate. To do this, we need the following global well-posedness result for System (\ref{KS}).
\begin{theorem}\label{theorem12} {\rm(}Global well-posedness for the Keller-Segel model{\rm)} Let $d\geq1$ and $p\in[1, \infty]$. Suppose that $(\ref{ABmu})$ holds. There exists a positive constant $\alpha^{*}$ such that if
\begin{equation}\label{a2}
\begin{aligned}
&\|\rho^{*}_{0}-\bar{\rho}\|_{\dot{B}^{\frac{d}{p}}_{p,1}}\leq \alpha^{*},
\end{aligned}
\end{equation}
then a unique solution $\rho^{*}$ to System {\rm(\ref{KS})} supplemented with the initial data $\rho^{*}_{0}$ exists globally in time and satisfies 
\begin{equation}\label{r2222}
\begin{aligned}
&\rho^{*}-\bar{\rho}\in \mathcal{C}_{b}(\mathbb{R}_{+};\dot{B}^{\frac{d}{p}}_{p,1})\cap L^1(\mathbb{R}_{+};\dot{B}^{\frac{d}{p}+2}_{p,1}),
\end{aligned}
\end{equation}
and
\begin{equation}\label{r2}
\begin{aligned}
&\|\rho^{*}-\bar{\rho}\|_{\widetilde{L}^{\infty}_{t}(\dot{B}^{\frac{d}{p}}_{p,1})}+\|\rho^{*}-\bar{\rho}\|_{L^1_{t}(\dot{B}^{\frac{d}{p}+2}_{p,1})}\leq C\|\rho^{*}_{0}-\bar{\rho}\|_{\dot{B}^{\frac{d}{p}}_{p,1}},\quad\quad t>0,
\end{aligned}
\end{equation}
where $C>0$ is a constant independent of time.
\end{theorem}

\begin{remark}
In {\rm\cite{di1}}, Di Francesco and Donatelli proved the existence of local-in-time solutions to \eqref{KS} in Sobolev Spaces via a relaxation weak-limit argument. We also mention the work {\rm{\cite{mu1}}} about the local well-posedness issue to \eqref{KS} in inhomogeneous Besov spaces $B^{s}_{p,r}$ with $1\leq p,r\leq \infty$ and $s>1+\frac{d}{p}$. In the context of initial data with finite mass and vacuum, we refer to the important results {\rm\cite{su1,su2,do1,perthame1,kim1,aru1}} and reference therein concerning the analysis of the Keller-Segel equations in Sobolev spaces.
\end{remark}

\begin{remark}
The low-frequency regularity properties in \eqref{r1} of the cell density $\rho$ to System \eqref{m1} obtained in Theorem \ref{theorem11} are consistent with the regularity properties in \eqref{r2222} of the solution $\rho^{*}$ to System \eqref{KS} for $p=2$. This is compatible with the relaxation limit result stated below.
\end{remark}

Employing \eqref{XX0}, one can establish the uniform-in-$\var$ regularity estimates of the solution to System \eqref{HPCvar} under the diffusive scaling \eqref{scaling1}, which leads to the justification of the weak relaxation limit of System (\ref{HPCvar}) to System (\ref{KS}) via standard compactness arguments. However, such method may not provide a convergence rate. In the next theorem, we prove a strong relaxation result and exhibit an explicit convergence rate of the relaxation process.

\begin{theorem}\label{theorem13} {\rm(}Relaxation limit{\rm)}
 Assume $d\geq1$, $0<\var<1$, $(\ref{ABmu})$-$(\ref{a1})$ and that $(\ref{a2})$ is satisfied with $p=2$. Let $(\rho,u,\phi)$ be the global solution to System \eqref{m1} associated to the initial data $(\rho_{0},u_{0},\phi_{0})$ obtained with Theorem \ref{theorem11}, and $(\rho^{*},u^*,\phi^*)$ be the global solution to System {\rm(\ref{KS})} supplemented with the initial data $\rho^{*}_{0}$ given by Theorem \ref{theorem12}. Then, for any $t>0$, there exists a constant $C>0$ independent of $\var$ and time such that
 \begin{equation}\label{uniform1}
 \begin{aligned}
 &\|\nabla P(\rho^{\var})-\mu\rho^{\var}\nabla\phi^{\var}+\rho^{\var}u^{\var}\|_{L^1_{t}(\dot{B}^{\frac{d}{2}}_{2,1})}+\|-\Delta\phi^{\var}-a\rho^{\var}+b\phi^{\var}\|_{L^1_{t}(\dot{B}^{\frac{d}{2}}_{2,1})}\leq C\var,
 \end{aligned}
 \end{equation}
 where $(\rho^\var,u^\var,\phi^\var)$ are defined by \eqref{scaling1}.
If it further holds
\begin{equation}\label{ar}
\begin{aligned}
&\|\rho_{0}-\rho^{*}_{0}\|_{\dot{B}^{\frac{d}{2}-1}_{2,1}}=\mathcal{O}(\var),
\end{aligned}
\end{equation}
then for any $t>0$,  we have the quantitative error estimates
\begin{equation}\label{error}
\begin{aligned}
&\|\rho^{\var}-\rho^{*}\|_{\widetilde{L}^{\infty}_{t}(\dot{B}^{\frac{d}{2}-1}_{2,1})}+\|\rho^{\var}-\rho^{*}\|_{L^{1}_{t}(\dot{B}^{\frac{d}{2}+1}_{2,1})}\\
&\quad\quad+\|u^{\var}-u^{*}\|_{L^{1}_{t}(\dot{B}^{\frac{d}{2}}_{2,1})}+\|\phi^{\var}-\phi^{*}\|_{L^{1}_{t}(\dot{B}^{\frac{d}{2}+1}_{2,1}\cap \dot{B}^{\frac{d}{2}+2}_{2,1})}\leq C\var.
\end{aligned}
\end{equation}
Consequently, as $\var\rightarrow0$, $(\rho^{\var},u^{\var},\phi^{\var})$ converges to $(\rho^{*},u^{*},\phi^{*})$ in the following sense$:$
\begin{equation}\label{conver}
\left\{
\begin{aligned}
&\rho^{\var}\rightarrow \rho^{*}\quad\quad\text{strongly in}\quad \widetilde{L}^{\infty}(\mathbb{R}_{+};\dot{B}^{\frac{d}{2}-1}_{2,1})\cap L^{1}(\mathbb{R}_{+};\dot{B}^{\frac{d}{2}+1}_{2,1}),\\
&u^{\var}\rightarrow u^{*}  \quad\quad \text{strongly in}\quad L^{1}(\mathbb{R}_{+}; \dot{B}^{\frac{d}{2}}_{2,1}),\\
&\phi^{\var}\rightarrow \phi^{*}\quad\quad\text{strongly in}\quad L^{1}(\mathbb{R}_{+};\dot{B}^{\frac{d}{2}+1}_{2,1}\cap\dot{B}^{\frac{d}{2}+2}_{2,1}).
\end{aligned}
\right.
\end{equation}

\end{theorem}

\begin{remark}To the best of our knowledge, Theorem \ref{theorem13} is the first result concerning the rigorous justification of the uniform-in-time relaxation process from the HPC system to the Keller-Segel model.
\end{remark}


\subsection{Strategy of proofs}\label{subsections}

Before explaining the strategies and ingredients to prove the above theorems, let us illustrate the main difficulties. First, the analysis of System \eqref{m1} may not be included in the classical theory for partial dissipative hyperbolic systems \cite{c2,xu1,xu2,yong1} or hyperbolic-parabolic systems \cite{kaw1, shi1}. In fact, to avoid the smallness assumption on the constant state $\bar{\rho}$, we should not study System (\ref{m1n}) as a damped compressible Euler equations and a damped heat equation with additional source terms (i.e. by treating $\mu\nabla\psi$ and $c_1n$ as source terms). And due to the difference between the hyperbolic scaling in $\eqref{m1n}_{1}$-$\eqref{m1n}_{2}$ and the parabolic scaling in $\eqref{m1n}_{3}$, we are not able to perform a rescaling as in \cite{c1,c2,c3} to reduce the proofs to the case $\var=1$ which allowed the authors to easily recover the exact dependency with respect to the parameter $\var$ thanks to the homogeneity of the norms. In addition, the nonlinear term $H(n)$ is of order $0$ (which is unusual in the analysis of compressible flows) in $\eqref{m1n}_{3}$, and the classical composition estimates may not be applied here. To overcome these diffucilties, we need delicate energy estimates with mixed $L^1$-time and $L^2$-time regularities of dissipation terms in both high and low frequencies.

In order to understand the behavior of the solution to \eqref{m1n} with respect to $\var$, we perform a spectral analysis of the associated linearized system of \eqref{m1n} as follows. By the Hodge decomposition, we denote the compressible part $m=\Lambda^{-1}\div u$ and the incompressible part $\omega=\Lambda^{-1}\nabla\times u$ and rewrite the linearized system of \eqref{m1n} as
 \begin{equation}\nonumber
\begin{aligned}
& \partial_{t}
\left(\begin{matrix}
   n  \\
   u  \\
   \psi \\
  \end{matrix}\right)
  =\mathbb{A}\left(\begin{matrix}
  n \\
  u\\
  \psi
     \end{matrix}\right),\quad  \mathbb{A}:=\left(\begin{matrix}
0                                 &    -\Lambda                                   &     0\\
\Lambda                           &    -\frac{1}{\var}                            &      -\mu\Lambda \\          
c_1                               &0                                              & -b-{\Lambda}^2\\
  \end{matrix}\right),
 \quad\quad \partial_{t} \omega +\frac{1}{\var}\omega=0.
   \end{aligned}
\end{equation}
The eigenvalues of the matrix $\widehat{\mathbb{A}}(\xi)$ satisfy
\begin{equation}\nonumber
\begin{aligned}
&\lambda^3+(\frac{1}{\var}+b+|\xi|^2  )\lambda^2+\big( \frac{b}{\var}+ \frac{|\xi|^2}{\var}+ |\xi|^2\big) \lambda+ |\xi|^2(|\xi|^2+b-c_{1}\mu)=0.
\end{aligned}
\end{equation}
By Taylor's expansion, for $\var|\xi|<<1$, the eigenvalues $\lambda_{1}$, $\lambda_{2}$ and $\lambda_{3}$ are real and asymptotically equal to
\begin{equation}\nonumber
\left\{
\begin{aligned}
&\lambda_{1}=-\Big(P'(\bar{\rho})-\frac{a\mu\bar{\rho}}{b}\Big)\var |\xi|^2+\frac{1}{\var}O(|\var \xi |^{4}),\\
&\lambda_{2}=-\frac{1 }{\var}+\frac{1}{\var}O(|\var \xi |^{2}),\\
&\lambda_{3}=-b+\Big(-\frac{1}{\var}+O(1)\Big)|\var \xi |^{2} + \frac{1}{\var}O(|\var \xi |^{2}).
\end{aligned}
\right.
\end{equation}
Similarly, for $\var|\xi|>>1$, there exists two conjugate complex eigenvalues $\lambda_{1}, \lambda_{2}$ and a real eigenvalue $\lambda_{3}$ which satisfy
\begin{equation}\nonumber
\left\{
\begin{aligned}
&\lambda_{1}=-\frac{1}{2\var}+i | \xi|+ \frac{1}{\var} O(\frac{1}{|\var \xi|} ),\\
&\lambda_{2}=-\frac{1}{2\var}-i | \xi|+ \frac{1}{\var} O(\frac{1}{|\var \xi|} ),\\
&\lambda_{3}=-b-|\xi|^2+\frac{1}{\var} O(\frac{1}{|\var \xi|} ).
\end{aligned}
\right.
\end{equation}
One can  observe that the behaviors of spectrum in low and high-frequency regimes are fundamentally different. For instance, $\lambda_1$ is associated to a parabolic behavior in low frequencies and to an exponentially damped behavior in high frequencies. Therefore, in order to capture the optimal regularity properties, we separate the frequency space into the above two regimes $|\xi|\lesssim \frac{1}{\var}$ and $|\xi|\gtrsim \frac{1}{\var}$ with  a threshold $J_{\var}$ satisfying $J_{\var}\sim \log_2(\frac{1}{\var})$, where the binary logarithm comes from the dyadic decomposition induced by the Littlewood-Paley theory.
Another key observation from these spectral structures is that the eigenvalues in low frequencies are purely real, and thus we expect to be able to partially diagonalize the system in this regime and study it in a decoupled manner.
This argument is central to overcome the so-called overdamping phenomenon (cf. \cite{zuazua1}) stating that the behavior of the whole solution is proportional to $\min\{1/\varepsilon,\varepsilon\}$ and therefore it is hard difficult to recover uniform-in-$\varepsilon$ bounds. Here, by splitting the frequencies and studying the low-frequency regime in a decoupled manner, we are able to derive the uniformity-in-$\varepsilon$ and describe the overdamping phenomenon for the HPC system.

Moreover, the choice of this particular threshold implies that $J_\varepsilon\to\infty$ as $\varepsilon\to0$ and thus that the low-frequency region recovers the whole frequency-space and the high frequencies disappears in the large-friction limit. This is coherent with the associated relaxation limit since the behavior of the limit system \eqref{KS} is similar to the behavior of System \eqref{m1} in low frequencies.

\medbreak
\textbf{Global well-posedness for the HPC system} 

The first step is to establish global a-priori estimates. In the low frequencies region, we introduce two new effective variables
\begin{equation}\nonumber
\begin{aligned}
&\varphi:=\psi -(b-\Delta)^{-1}(c_{1}n+H(n) )\quad\text{and}\quad v:= u+\varepsilon\nabla n-\varepsilon\mu \nabla\psi,
\end{aligned}
\end{equation}
so that $( n, \varphi,v)$ satisfies the equations
\begin{equation}\nonumber
\left\{
\begin{aligned}
&\partial_{t}n-\tilde{\Delta}_1n= L_{1}+R_{1},\\
&\partial_{t}\varphi-\tilde{\Delta}_2\varphi= L_{2}+R_{2},\\
&\partial_{t} v+\var^{-1} v= L_{3}+R_{3},
\end{aligned}
\right.
\end{equation}
where the operators $\tilde{\Delta}_{i}$ $(i=1,2)$ satisfy the dissipative properties (\ref{dis1})-(\ref{dis2}), the high-order linear terms $L_{i}$ $(i=1,2,3)$ will be absorbed if the threshold $J_{\var}$ takes the form (\ref{J0}) for a $k$ chosen small enough, and the nonlinear terms $R_{i}$ $(i=1,2,3)$ will be estimated by some non so-classical composition lemmas and product estimates, see Subsection \ref{subsection31}. Both these new unknowns can be considered as low-frequency damped modes for which we can recover optimal dissipative properties. The second unknown is reminiscent of the effective velocity introduced by Hoff and Haspot in \cite{hoff1,haspot1} to treat the high frequencies for the compressible Navier-Stokes system and by Crin-Barat and Danchin in \cite{c1} to analyze the low frequencies for the compressible Euler equations with damping.

In the high frequencies region, we are able to derive the Lyapunov type inequality
\begin{equation}\nonumber
\begin{aligned}
&\frac{d}{dt}\mathcal{L}_{j}(t)+\mathcal{H}_{j}(t)\lesssim (\text{nonlinear terms} )\sqrt{\mathcal{L}_{j}(t)},\quad\quad j\geq J_{\var}-1,
\end{aligned}
\end{equation}
where $\mathcal{L}_{j}(t)$ is the nonlinear energy functional
\begin{equation}\label{Lyapunov} 
\begin{aligned}
\mathcal{L}_{j}(t):&=\var\int[ \frac{1}{2}|n_{j}|^2+\frac{2^{-2j}}{2\eta_{0}} |H(n)_{j}|^2+\frac{1}{2}w_{j}|u_{j}|^2+\frac{\mu b}{2c_{1}} |\psi_{j}|^2+\frac{\mu}{2c_{1}} |\nabla \psi_{j}|^2\\
&\quad-\mu n_{j} \psi_{j}-H(n)_{j} \psi_{j}]dx +\eta_{0}2^{-2j}\int (\frac{\mu}{2c_{1}}|\nabla \psi_{j}|^2+u_{j} \cdot\nabla n_{j} )dx, 
\end{aligned}
\end{equation}
and $\mathcal{H}_{j}(t)$ is the corresponding dissipation
\begin{equation}\nonumber
\begin{aligned}
&\mathcal{H}_{j}(t):= \var\int  (\frac{1}{\var} w_{j}|u_{j}|^2+|\partial_{t}\psi_{j}|^2)dx+\eta_{0} 2^{-2j}\int \big{(}|\nabla n_{j}|^2+\frac{\mu b}{c_{1}}|\nabla \psi_{j}|^2+\frac{\mu}{c_{1}}|\Delta\psi_{j}|^2\\
&\quad\quad\quad\quad-2\mu \nabla n_{j}\cdot \nabla \psi_{j}-w_{j}|\div u_{j}|^2+\frac{1}{\var} u_{j} \cdot\nabla n_{j} \big{)}dx,
\end{aligned}
\end{equation}
with $\eta_{0}$ a well-chosen constant and the weight function $w_{j}:=c_{0}+\dot{S}_{j-1}G(n)$. The construction of the Lyapunov functional with lower-order terms is motivated by the previous work by Beauchard and Zuazua in \cite{bea1} and the basic energy equality of the HPC system. The weight function satisfying $w_{j}\sim1$ by the smallness of energy is used to get rid of the nonlinear term $G(n)\div u$ which breaks the symmetry in $\eqref{m1n}_1$. It should be noted that to compensate $\mu\nabla\psi$ in $\eqref{m1n}_2$, one has to take the $L^2$ inner of the equation for $\psi_{j}$ with $\partial_{t}\psi_{j}$. Consequently, we need to control
\begin{equation}\nonumber
\begin{aligned}
&\int H(n)_{j} \partial_t\psi_{j}dx=\frac{d}{dt}\int 2^{-j}H(n)_{j} 2^{j}\psi_{j}dx-\int 2^{-j}\partial_{t}H(n)_{j} 2^{j} \psi_{j}dx,
\end{aligned}
\end{equation} 
and thus $2^{-2j}\|H(n)_{j}\|_{L^2}^2$ is added into the energy functional $\mathcal{L}_{j}(t)$ to avoid the difficulty caused by additional nonlinearities. It is also noted that the dissipation term $\|\partial_{t}\psi_{j}\|_{L^2}^2$ in $\mathcal{H}_{j}(t)$ plays a key role in deriving the dissipation term $2^{-2j}\|H(n)_{j}\|_{L^2}^2$. In addition, we show new composition estimates in Lemma \ref{compositionlp} to deal with the quadratic nonlinear term $H(n)$. With the help of these observations, we are able to establish the expected estimates in the high frequencies region. The details will be presented in Subsection \ref{subsection32}.

\textbf{Optimal time-decay rates}

To show Theorem \ref{theorem14}, we adapt the works by Guo-Wang \cite{guo1} and Xin-Xu \cite{xin1} about optimal time-decay rates for the compressible Navier-Stokes equations without additional smallness assumptions. However, the strategy here is slightly different due to the specific dissipative structures in low frequencies and in high frequencies of our system.

 First, we assume the additional regularity on the initial data \eqref{a3} which is a weaker assumption than the usual $L^1$ assumption, and we show that this regularity is propagated to the solution constructed in the existence theorem (see (\ref{lowd})). Then, we introduce a new time-weighted functional $\mathcal{D}^{\theta}(t)$, defined in (\ref{Xdef}) with the time weight $t^{\theta}$ for $\theta>1+\frac{1}{2}(\frac{d}{2}-\sigma_{0})$. Since it involves the higher decay information $t^{-\theta}$, one needs to bound $\mathcal{D}^{\theta}(t)$ by $t^{\theta-\frac{1}{2}(\frac{d}{2}-\sigma_{0})}$ in which $\frac{1}{2}(\frac{d}{2}-\sigma_{0})$ matches the index of the time-decay rate for $\sigma=\frac{d}{2}$ in $(\ref{decay1})_{1}$. To derive the time-weighted estimates in low frequencies, we multiply the equations of $(n,v,\varphi)$ by $t^{\theta}$ to obtain
\begin{equation}\nonumber
\left\{
\begin{aligned}
&\partial_{t}(t^{\theta}n)-\tilde{\Delta}_1(t^{\theta}n)=\theta t^{\theta-1}n+ t^{\theta} L_{1}+t^{\theta}R_{1},\\
&\partial_{t}( t^{\theta}\varphi)-\tilde{\Delta}_2(t^{\theta}\varphi)=\theta t^{\theta-1}\varphi+ t^{\theta} L_{2}+t^{\theta} R_{2},\\
&\partial_{t} (t^{\theta}v)+\var^{-1} ( t^{\theta}v)=\theta t^{\theta-1}v+t^{\theta} L_{3}+t^{\theta} R_{3}.
\end{aligned}
\right.
\end{equation}
 Here the right-hand terms of the above equations can be estimated by similar arguments as in the low-frequency analysis to prove the global existence. The key point here is to make use of an interpolation inequality in both time and space, e.g., 
\begin{equation}\nonumber
\begin{aligned}
&\int_{0}^{t}\tau^{\theta-1}\|n^{\ell}\|_{\dot{B}^{\frac{d}{2}}_{2,1}}d\tau\lesssim \big{(} \|n^{\ell}\|_{\widetilde{L}^{\infty}_{t}(\dot{B}^{\sigma_{0}}_{2,\infty})} t^{\theta-\frac{1}{2}(\frac{d}{2}-\sigma_{0})} \big{)}^{\frac{2}{\frac{d}{2}+2-\sigma_{0}}}\big{(} \var\|n^{\ell}\|_{L^1_{t}(\dot{B}^{\frac{d}{2}+2}_{2,1})}\big{)}^{\frac{\frac{d}{2}-\sigma_{0}}{\frac{d}{2}+2-\sigma_{0}}}.
\end{aligned}
\end{equation}
For the high frequencies, multiplying the Lyapunov type inequality by $t^{\theta}$ yields
\begin{equation}\nonumber
\begin{aligned}
&\frac{d}{dt}\big{(} t^{\theta} \mathcal{L}_{j}(t) \big{)}+t^{\theta}\mathcal{H}_{j}(t)\lesssim t^{\theta-1} \mathcal{L}_{j}(t)+t^{\theta}(\text{nonlinear terms} )\sqrt{\mathcal{L}_{j}(t)},\quad\quad j\geq J_{\var}-1.
\end{aligned}
\end{equation}
Arguing similarly as in the high-frequency analysis of global existence to control the nonlinear terms and taking advantage of the inequality
\begin{equation}\nonumber
\begin{aligned}
&\int_{0}^{t} \tau^{\theta-1}  \|n\|_{\dot{B}^{\frac{d}{2}+1}_{2,1}}^{h}d\tau\lesssim \big{(} \|n\|_{\widetilde{L}^{\infty}_{t}(\dot{B}^{\frac{d}{2}+1}_{2,1})}^{h} t^{\theta-\frac{1}{2}(\frac{d}{2}-\sigma_{0})}\big{)}^{\frac{2}{\frac{d}{2}+2-\sigma_{0}}} \big{(}  \|n\|_{L^{1}_{t}(\dot{B}^{\frac{d}{2}+1}_{2,1})}^{h} \big{)}^{\frac{\frac{d}{2}-\sigma_{0}}{\frac{d}{2}+2-\sigma_{0}}},
\end{aligned}
\end{equation}
we can prove the high-frequency time-weighted estimates. Adding these time-weighted estimates in both low and high frequencies, we are able to get (\ref{decay1}). Then, the faster time-decay rates in (\ref{decay1})-(\ref{decay3}) are shown by isolating and capturing the damping effects in the equations $\eqref{m1n}_{2}$ and $\eqref{m1n}_{3}$.

\textbf{Relaxation limit}

To justify the relaxation limit, we first establish the global well-posedness for System (\ref{KS}). This can be done by using the reformulation
\begin{equation}\label{reKS2}
\begin{aligned}
&\partial_{t} \rho^\var-\tilde{\Delta}_{*} \rho^\var=\text{nonlinear terms}.
\end{aligned}
\end{equation}
The key observation is that the differential operator $\tilde{\Delta}_{*}:=(P'(\bar{\rho})-\mu a\bar{\rho} (b-\Delta)^{-1} )\Delta$ behaves like the Laplacian $\Delta$ under the condition $P'(\bar{\rho})>\dfrac{\mu a}{b}\bar{\rho}$. Thus, treating (\ref{reKS2}) as the heat diffusion, we can obtain the expected a-priori estimates, see Subsection \ref{subsection35}.

Finally, to justify the strong convergence and derive a convergence rate of the relaxation process, we estimate the error equations
 \begin{equation}\nonumber
 \left\{
 \begin{aligned}
 &\partial_{t}\delta \rho^{\var}-\tilde{\Delta}_{*}\delta \rho^{\var}=\mathcal{G}^{\var}+\text{nonlinear terms},\\
 &\delta\phi^{\var}=(b-\Delta)^{-1}( -\var \partial_{t}\phi^{\var}+a\delta \rho^\var ),
 \end{aligned}
 \right.
 \end{equation}
 with $\delta \rho^\var=\rho^\var-\rho^*$ and $\delta \phi^\var=\phi^\var-\phi^*$ for $(\rho^\var,\phi^\var)$ defined in \eqref{scaling1}.
Thanks to the uniform-in-$\var$ estimates \eqref{XX0} on the effective velocity $\dfrac{1}{\var}(u^\var+\dfrac{1}{\rho^{\var}}\nabla P(\rho^{\var})-\mu\nabla\psi^\var)$ and the time derivative $\partial_{t}\phi^{\var}$, we are able to show $\|\mathcal{G}^{\var}\|_{L^1_t(B^{\frac d2-1}_{2,1})}=\mathcal{O}(\var)$, which plays a key role in deriving the estimate for $\delta \rho^{\var}$ and justify the relaxation limit of global classical solutions from (\ref{HPCvar}) to (\ref{KS}) with the convergence rate $\var$.

\section{Global well-posedness for the HPC system}\label{section3}

In this section, we prove Theorem \ref{theorem11} pertaining to the global well-posedness of classical solutions to the Cauchy problem for the HPC system (\ref{m1n}). The key points are to derive the a-priori estimates for any classical solution $(n,u,\psi)$ uniformly with respect to time and then apply a bootstrap argument. To this end, we assume that we are given a $(n,u,\psi)$ is a smooth solution of \eqref{m1n} and we do the following a-priori assumption:  
\begin{equation}\label{aLinfty}
\begin{aligned}
\mathcal{X}(t)\leq  C_{0}\mathcal{X}_{0}\leq C_{0}\alpha,
\end{aligned}
\end{equation}
where $C_{0}>0$ and $\alpha>0$ are two constants which will be determined later. This assumption will enable us to use freely the composition estimates.
 
\subsection{Low-frequency analysis}\label{subsection31}

In the low-frequency region $ |\xi|\leq 2^{J_{\var}}$, we introduce two new effective unknowns
\begin{equation}\label{vw}
\begin{aligned}
&\varphi:=\psi -(b-\Delta)^{-1}(c_{1}n+H(n) )\quad\text{and}\quad v:= u+\varepsilon\nabla n-\varepsilon\mu \nabla\psi,
\end{aligned}
\end{equation}
so as to obtain the fine structure
\begin{equation}\label{good}
\left\{
    \begin{aligned}
    &\partial_{t}u+u\cdot\nabla u+\frac{1}{\var}v=0,\\
    &\partial_{t}\psi+b\varphi-\Delta \varphi=0.
    \end{aligned}
    \right.
\end{equation}
Making use of (\ref{good}) and substituting
\begin{equation}\label{vw1}
\left\{
    \begin{aligned}
    &\psi=\varphi+(b-\Delta)^{-1}(c_{1}n+H(n) ),\\
    &u=v -\var \nabla (1-\mu c_{1}(b-\Delta)^{-1})n+\var\mu\nabla\varphi+ \var \mu \nabla(b-\Delta)^{-1} H(n)
    \end{aligned}
    \right.
\end{equation}
into (\ref{m1n}), we can verify that $(n,\varphi,v)$ satisfies the equations
\begin{equation}\label{Ga}
\left\{
\begin{aligned}
&\partial_{t}n-\tilde{\Delta}_1n= L_{1}+R_{1},\\
&\partial_{t}\varphi-\tilde{\Delta}_2\varphi= L_{2}+R_{2},\\
&\partial_{t} v+\var^{-1}v= L_{3}+R_{3},\\
\end{aligned}
\right.
\end{equation}
where $\tilde{\Delta}_{i}$ $(i=1,2)$ are the differential operators
\begin{equation}\label{delta1delta2}
\left\{
    \begin{aligned}
        & \tilde{\Delta}_1:=\var c_{0}(1- \mu c_{1} \Delta (b-\Delta)^{-1})\Delta,\\
        &\tilde{\Delta}_2:=-b+(1+\var \mu  c_{0}c_{1} (b-\Delta)^{-1}) \Delta,
    \end{aligned}
    \right.
\end{equation}
$L_{i}$ ($i=1,2,3$) are the linear terms
\begin{equation}\label{L123}
\left\{
    \begin{aligned}
    &L_{1}:=-\var \mu c_{0} \Delta\psi- c_{0}\div v,\\
        &L_{2}:=-\var c_{0}c_{1}(b-\Delta)^{-1}( 1-c_{1}(b-\Delta)^{-1}) \Delta n + c_{0}c_{1} (b-\Delta)^{-1} \div v,\\
        &L_{3}:=\var^2 c_{0}\nabla (1-c_{1}\Delta (b-\Delta)^{-1})\Delta n-\var^2 \mu c_{0}\nabla\Delta \varphi+\var\mu\nabla(b-\Delta)\varphi - \var c_{0}\nabla\div v,
    \end{aligned}
    \right.
\end{equation}
and $R_{i}$ ($i=1,2,3$) are the nonlinear terms
\begin{equation}\label{R123}
\left\{
    \begin{aligned}
&R_{1}:=-u\cdot\nabla n-G(n)\div u-\var \mu c_{0} \Delta(b-\Delta)^{-1} H(n),\\
&R_{2}:= -(b-\Delta)^{-1} (c_{1} R_{1}+\partial_{t} H(n) ),\\
&R_{3}:= \var \nabla R_{1}-u\cdot \nabla u.
    \end{aligned}
    \right.
\end{equation}

Before estimating (\ref{Ga}), we consider the following linear problems 
\begin{equation}\label{linear1}
\left\{
\begin{aligned}
&\partial_{t}f-\tilde{\Delta}_1f=g,\quad\quad x\in\mathbb{R}^{d},\quad t>0,\\
&f(x,0)=f_{0}(x),\quad \quad x\in\mathbb{R}^{d},
\end{aligned}
\right.
\end{equation}
and
\begin{equation}\label{linear2}
\left\{
\begin{aligned}
&\partial_{t}f-\tilde{\Delta}_2f=g,\quad\quad x\in\mathbb{R}^{d},\quad t>0\\
&f(x,0)=f_{0}(x),\quad \quad x\in\mathbb{R}^{d}.
\end{aligned}
\right.
\end{equation}
The following lemma will be very useful to obtain our a-priori estimates.
\begin{lemma}\label{lemma21} Let $r\in[1, \infty]$, $s\in\mathbb{R}$ and $T>0$. Assume $u_{0}\in\dot{B}^{s}_{2,r}$ and $g\in L^{1}(0,T;\dot{B}^{s}_{2,r})$.

(1) If $P'(\bar{\rho})>\dfrac{\mu a}{b}\bar{\rho}$ holds and $f$ is a solution to the problem $(\ref{linear1})$, then there exists a universal constant $C>0$ such that
\begin{equation}\label{diss1}
\begin{aligned}
&\|f\|_{\widetilde{L}^{\infty}_{t}(\dot{B}^{s}_{2,r})}+\var\|f\|_{\widetilde{L}^1_{t}(\dot{B}^{s+2}_{2,r})}\leq C(\|f_{0}\|_{\dot{B}^{s}_{2,r}}+\|g\|_{\widetilde{L}^1_{t}(\dot{B}^{s}_{2,r})}),\quad~ 0<t<T.
\end{aligned}
\end{equation}

(2) If $f$ is a solution to the problem $(\ref{linear2})$, then it holds that
\begin{equation}\label{diss2}
\begin{aligned}
&\|f\|_{\widetilde{L}^{\infty}_{t}(\dot{B}^{s}_{2,r})}+\|f\|_{\widetilde{L}^1_{t}(\dot{B}^{s}_{2,r}\cap\dot{B}^{s+2}_{2,r})}\leq C(\|f_{0}\|_{\dot{B}^{s}_{2,r}}+\|g\|_{\widetilde{L}^1_{t}(\dot{B}^{s}_{2,r})}),\quad 0<t<T.
\end{aligned}
\end{equation}
\end{lemma}
\begin{proof}
Under the assumption $P'(\bar{\rho})>\dfrac{\mu a}{b}\bar{\rho}$, it can be verified with the Plancherel formula that $\tilde{\Delta}_1$ satisfies the dissipative property
\begin{equation}\label{dis1}
\begin{aligned}
&-\int \tilde{\Delta}_1f\,fdx=\var c_{0}\int \big(1-\mu c_{1}b^{-1}+\frac{\mu c_{1}}{b+|\xi|^2}\big)|\xi|^2|\mathcal{F}(f)|^2d\xi \geq \var c_{0}( 1-\mu c_{1}b^{-1})\|\nabla f\|_{L^2}^2.
\end{aligned}
\end{equation}
 In addition, one has
\begin{equation}\label{dis2}
\begin{aligned}
&-\int \tilde{\Delta}_2f\,fdx=\int \big( b+(1+\frac{\var \mu c_{0}c_{1}}{b+|\xi|^2})|\xi|^2\big)|\mathcal{F}(f)|^2dx \geq b\|f\|_{L^2}^2+\|\nabla f\|_{L^2}^2.
\end{aligned}
\end{equation}
Thus, performing localized-in-frequency energy estimates, making use of \eqref{dis1}-\eqref{dis2} and summing these over $j\in\mathbb{Z}$ with suitable weights, we can show (\ref{diss1})-(\ref{diss2}) by using \eqref{ABmu} and the definition of $c_1$. For brevity, we omit the details.
\end{proof}

The key point in the low frequencies region is the following a-priori estimates of $(n,\varphi,v)$.
\begin{lemma}\label{lemma22} 
Let $T>0$ be any given time,  $J_{\var}=[-\log_2{\var}]+k$ for a sufficiently small constant $k$, and $(n,u,\psi)$ for $t\in(0,T)$ be the solution of System \eqref{m1n} with the initial data $(n_{0},u_{0},\phi_{0})$. Then, it holds under the assumptions {\rm(\ref{ABmu})} and {\rm(\ref{aLinfty})} that
\begin{equation}\label{XLvw}
\begin{aligned}
&\|(n,\varphi,v)\|_{\widetilde{L}^{\infty}_{t}(\dot{B}^{\frac{d}{2}}_{2,1})}^{\ell}+\var\|n\|_{L^{1}_{t}(\dot{B}^{\frac{d}{2}+2}_{2,1})}^{\ell}+\|\varphi\|_{L^{1}_{t}(\dot{B}^{\frac{d}{2}}_{2,1}\cap\dot{B}^{\frac{d}{2}+2}_{2,1})}^{\ell}+\frac{1}{\var}\|v\|_{L^{1}_{t}(\dot{B}^{\frac{d}{2}}_{2,1})}^{\ell}\\
&\quad\leq C\big(\mathcal{X}_{0}+\mathcal{X}^2(t) \big),\quad 0<t<T,
\end{aligned}
\end{equation}
where $\mathcal{X}(t)$, $\mathcal{X}_{0}$ and $(\varphi,v)$ are defined by \eqref{Xdef}, \eqref{a1} and \eqref{vw1}, respectively.
\end{lemma}

\begin{proof}
First, by $(\ref{Ga})_{1}$ and (\ref{diss1}), we have 
\begin{equation}\label{221}
\begin{aligned}
&\|n\|_{\widetilde{L}^{\infty}_{t}(\dot{B}^{\frac{d}{2}}_{2,1})}^{\ell}+\var\|n\|_{L^{1}_{t}(\dot{B}^{\frac{d}{2}+2}_{2,1})}^{\ell}\\
&\quad\lesssim \|n_{0}\|_{\dot{B}^{\frac{d}{2}}_{2,1}}^{\ell}+\|L_{1}\|_{L^{1}_{t}(\dot{B}^{\frac{d}{2}}_{2,1})}^{\ell}+\|R_{1}\|_{L^1_{t}(\dot{B}^{\frac{d}{2}}_{2,1})}^{\ell}\\
&\quad\lesssim \|n_{0}\|_{\dot{B}^{\frac{d}{2}}_{2,1}}^{\ell}+\var\|\varphi\|_{L^{1}_{t}(\dot{B}^{\frac{d}{2}+2}_{2,1})}^{\ell}+\|v\|_{L^{1}_{t}(\dot{B}^{\frac{d}{2}+1}_{2,1})}^{\ell}+\|R_{1}\|_{L^1_{t}(\dot{B}^{\frac{d}{2}}_{2,1})}^{\ell}\\
&\quad\lesssim \|n_{0}\|_{\dot{B}^{\frac{d}{2}}_{2,1}}^{\ell}+\var\|\varphi\|_{L^{1}_{t}(\dot{B}^{\frac{d}{2}+1}_{2,1}\cap\dot{B}^{\frac{d}{2}+3}_{2,1})}^{\ell}+\|v\|_{L^{1}_{t}(\dot{B}^{\frac{d}{2}+1}_{2,1})}^{\ell}+\|R_{1}\|_{L^1_{t}(\dot{B}^{\frac{d}{2}}_{2,1})}^{\ell},
\end{aligned}
\end{equation}
where in the last inequality one has used the embedding $\dot{B}^{\frac{d}{2}+1}_{2,1}\cap\dot{B}^{\frac{d}{2}+3}_{2,1}\hookrightarrow \dot{B}^{\frac{d}{2}+2}_{2,1}$. It is clear that the Bessel potential $(b-\Delta)^{-1}$ satisfies
\begin{equation}\label{bessel}
\begin{aligned}
&\|(b-\Delta)^{-1}f\|_{\dot{B}^{s}_{p,r}\cap \dot{B}^{s+2}_{p,r}}\lesssim \|f\|_{\dot{B}^{s}_{p,r}},\quad \forall f\in \dot{B}^{s}_{p,r}, \quad s\in\mathbb{R},\quad p,r\in[1,\infty].
\end{aligned}
\end{equation}
This combined with $(\ref{Ga})_{2}$, (\ref{diss1}) and (\ref{221}) implies
\begin{equation}\label{222}
\begin{aligned}
&\|\varphi\|_{\widetilde{L}^{\infty}_{t}(\dot{B}^{\frac{d}{2}}_{2,1})}^{\ell}+\|\varphi\|_{L^{1}_{t}(\dot{B}^{\frac{d}{2}}_{2,1}\cap \dot{B}^{\frac{d}{2}+2}_{2,1})}^{\ell}\\
&\quad\lesssim \|\varphi|_{t=0}\|_{\dot{B}^{\frac{d}{2}}_{2,1}}^{\ell}+\|L_{2}\|_{L^1_{t}(\dot{B}^{\frac{d}{2}}_{2,1})}^{\ell}+\|R_{2}\|_{L^1_{t}(\dot{B}^{\frac{d}{2}}_{2,1})}^{\ell}\\
&\quad\lesssim \|\varphi|_{t=0}\|_{\dot{B}^{\frac{d}{2}}_{2,1}}^{\ell}+\var\|n\|_{L^1_{t}(\dot{B}^{\frac{d}{2}+2}_{2,1})}^{\ell}+\|v\|_{L^1_{t}(\dot{B}^{\frac{d}{2}+1}_{2,1})}^{\ell}+\|R_{2}\|_{L^1_{t}(\dot{B}^{\frac{d}{2}}_{2,1})}^{\ell}\\
&\quad\lesssim \|(n,\varphi)|_{t=0}\|_{\dot{B}^{\frac{d}{2}}_{2,1}}^{\ell}+\var\|\varphi\|_{L^{1}_{t}(\dot{B}^{\frac{d}{2}+1}_{2,1}\cap \dot{B}^{\frac{d}{2}+3}_{2,1})}^{\ell}+\|v\|_{L^{1}_{t}(\dot{B}^{\frac{d}{2}+1}_{2,1})}^{\ell}+\|(R_{1},R_{2})\|_{L^1_{t}(\dot{B}^{\frac{d}{2}}_{2,1})}^{\ell}.
\end{aligned}
\end{equation}
 Similarly, it also holds by $(\ref{Ga})_{3}$ and (\ref{222}) that
\begin{equation}\label{223}
\begin{aligned}
&\|v\|_{\widetilde{L}^{\infty}_{t}(\dot{B}^{\frac{d}{2}}_{2,1})}^{\ell}+\frac{1}{\var}\|v\|_{L^{1}_{t}(\dot{B}^{\frac{d}{2}}_{2,1})}^{\ell}\\
&\quad\lesssim \|v|_{t=0}\|_{\dot{B}^{\frac{d}{2}}_{2,1}}^{\ell}+\|L_{3}\|_{L^{1}_{t}(\dot{B}^{\frac{d}{2}}_{2,1})}^{\ell}+\|R_{3}\|_{L^1_{t}(\dot{B}^{\frac{d}{2}}_{2,1})}^{\ell}\\
&\quad\lesssim \|v|_{t=0}\|_{\dot{B}^{\frac{d}{2}}_{2,1}}^{\ell}+\var^2\|n\|_{L^1_{t}(\dot{B}^{\frac{d}{2}+3}_{2,1})}^{\ell}+\var\|\varphi\|_{L^1_{t}(\dot{B}^{\frac{d}{2}+1}_{2,1}\cap\dot{B}^{\frac{d}{2}+3}_{2,1})}^{\ell}+\var\|v\|_{L^1_{t}(\dot{B}^{\frac{d}{2}+2}_{2,1})}^{\ell}+\|R_{3}\|_{L^1_{t}(\dot{B}^{\frac{d}{2}}_{2,1})}^{\ell}.
\end{aligned}
\end{equation}
Combining (\ref{222})-(\ref{223}) together and using (\ref{lh}), we have
\begin{equation}\label{224}
\begin{aligned}
&\|(n,\varphi,v)\|_{\widetilde{L}^{\infty}_{t}(\dot{B}^{\frac{d}{2}}_{2,1})}^{\ell}+\var\|n\|_{L^{1}_{t}(\dot{B}^{\frac{d}{2}+2}_{2,1})}^{\ell}+\|\varphi\|_{L^{1}_{t}(\dot{B}^{\frac{d}{2}}_{2,1}\cap \dot{B}^{\frac{d}{2}+2}_{2,1})}^{\ell}+\frac{1}{\var}\|v\|_{L^{1}_{t}(\dot{B}^{\frac{d}{2}}_{2,1})}^{\ell}\\
&\quad\lesssim \|(n,\varphi,v)|_{t=0}\|_{\dot{B}^{\frac{d}{2}}_{2,1}}^{\ell}+\var^2\|n\|_{L^1_{t}(\dot{B}^{\frac{d}{2}+3}_{2,1})}^{\ell}+\var\|\varphi\|_{L^{1}_{t}(\dot{B}^{\frac{d}{2}+1}_{2,1}\cap \dot{B}^{\frac{d}{2}+3}_{2,1})}^{\ell}\\
&\quad\quad+\|v\|_{L^1_{t}(\dot{B}^{\frac{d}{2}+1}_{2,1})}^{\ell}+\var\|v\|_{L^1_{t}(\dot{B}^{\frac{d}{2}+2}_{2,1})}^{\ell}+\|(R_{1},R_{2},R_{3})\|_{L^1_{t}(\dot{B}^{\frac{d}{2}}_{2,1})}^{\ell}\\
&\quad\leq  \|(n,\varphi,v)|_{t=0}\|_{\dot{B}^{\frac{d}{2}}_{2,1}}^{\ell}+\var^2 2^{J_{\var}}\|n\|_{L^{1}_{t}(\dot{B}^{\frac{d}{2}+2}_{2,1})}^{\ell}+\var2^{J_{\var}}\|\varphi\|_{L^{1}_{t}(\dot{B}^{\frac{d}{2}}_{2,1}\cap\dot{B}^{\frac{d}{2}+2}_{2,1})}^{\ell}\\
&\quad\quad+\var 2^{J_{\var}}(1+ \var 2^{J_{\var}})\frac{1}{\var}\|v\|_{L^{1}_{t}(\dot{B}^{\frac{d}{2}}_{2,1})}^{\ell}+\|(R_{1},R_{2},R_{3})\|_{L^1_{t}(\dot{B}^{\frac{d}{2}}_{2,1})}^{\ell}\\
\end{aligned}
\end{equation}
If we set $J_{\var}=[-\log_2{\var}]+k$, then one concludes from (\ref{224}) and $\var 2^{J_{\var}}\sim 2^{k}$ that there is a sufficiently small constant $k$ such that
 \begin{equation}\label{226}
\begin{aligned}
&\|(n,\varphi,v)\|_{\widetilde{L}^{\infty}_{t}(\dot{B}^{\frac{d}{2}}_{2,1})}^{\ell}+\var\|n\|_{L^{1}_{t}(\dot{B}^{\frac{d}{2}+2}_{2,1})}^{\ell}+\|\varphi\|_{L^{1}_{t}(\dot{B}^{\frac{d}{2}}_{2,1}\cap \dot{B}^{\frac{d}{2}+2}_{2,1})}^{\ell}+\frac{1}{\var}\|v\|_{L^{1}_{t}(\dot{B}^{\frac{d}{2}}_{2,1})}^{\ell}\\
&\quad\lesssim \|(n,\varphi,v)|_{t=0}\|_{\dot{B}^{\frac{d}{2}}_{2,1}}^{\ell}+\|(R_{1},R_{2},R_{3})\|_{L^1_{t}(\dot{B}^{\frac{d}{2}}_{2,1})}^{\ell}.
\end{aligned}
\end{equation}
By (\ref{lh}), we obtain for $p,r\in[1,\infty]$, $s\in\mathbb{R}$ and $s'>0$ that
\begin{equation}\label{hl}
\begin{aligned}
&\|f^{\ell}\|_{\dot{B}^{s}_{p,r}}\leq \|f\|_{\dot{B}^{s}_{p,r}}^{\ell}\lesssim\var^{-s'}\|f\|_{\dot{B}^{s-s'}_{p,r}}^{\ell},\quad\quad\|f^{h}\|_{\dot{B}^{s}_{p,r}}\leq \|f\|_{\dot{B}^{s}_{p,r}}^{h}\leq \var^{s'}\|f\|_{\dot{B}^{s+s'}_{p,r}}^{h}.
\end{aligned}
\end{equation} 
 According to (\ref{bessel}), (\ref{hl}) and (\ref{F1}), the first term on the right-hand side of (\ref{226}) satisfies
\begin{equation}\label{inv}
\begin{aligned}
& \|(n,\varphi,v)|_{t=0}\|_{\dot{B}^{\frac{d}{2}}_{2,1}}^{\ell}\lesssim  \|(n_{0},u_{0},\psi_{0})\|_{\dot{B}^{\frac{d}{2}}_{2,1}}^{\ell}+\var\|n_{0}\|_{\dot{B}^{\frac{d}{2}+1}_{2,1}}^{h}.
\end{aligned}
\end{equation}
The nonlinear terms $R_{i}$ ($i=1,2,3$) can be estimated as follows. By (\ref{uv2}) and standard interpolations, one has
\begin{equation}\nonumber
\begin{aligned}
&\|u\cdot\nabla n\|_{L^1_{t}(\dot{B}^{\frac{d}{2}}_{2,1})}^{\ell}\\
&\quad\lesssim \|u\|_{L^2_{t}(\dot{B}^{\frac{d}{2}}_{2,1})}\|n\|_{L^2_{t}(\dot{B}^{\frac{d}{2}+1}_{2,1})}\\
&\quad\lesssim \var^{-\frac{1}{2}} \big(\|u\|_{L^2_{t}(\dot{B}^{\frac{d}{2}}_{2,1})}^{\ell}+\|u\|_{L^2_{t}(\dot{B}^{\frac{d}{2}}_{2,1})}^{h}\big) \Big( \big( \|n\|_{L^{\infty}_{t}(\dot{B}^{\frac{d}{2}}_{2,1})}^{\ell} \big)^{\frac{1}{2}} \big( \var\|n\|_{L^1_{t}(\dot{B}^{\frac{d}{2}+2}_{2,1})}^{\ell} \big)^{\frac{1}{2}}\\
&\quad\quad+ \big( \var\|n\|_{L^{\infty}_{t}(\dot{B}^{\frac{d}{2}+1}_{2,1})}^{h} \big)^{\frac{1}{2}} \big(\|n\|_{L^1_{t}(\dot{B}^{\frac{d}{2}+1}_{2,1})}^{h} \big)^{\frac{1}{2}} \Big)\lesssim \mathcal{X}^2(t).
\end{aligned}
\end{equation}
From (\ref{hl}) and the embedding $\dot{B}^{\frac{d}{2}}_{2,1}\hookrightarrow L^{\infty}$, we also get
\begin{equation}\label{ninfty}
\begin{aligned}
&\|n\|_{L^{\infty}_{t}(L^{\infty})}\lesssim \|n\|_{L^{\infty}_{t}(\dot{B}^{\frac{d}{2}}_{2,1})}\lesssim \|n\|_{L^{\infty}_{t}(\dot{B}^{\frac{d}{2}}_{2,1})}^{\ell}+\var\|n\|_{L^{\infty}_{t}(\dot{B}^{\frac{d}{2}+1}_{2,1})}^{h}\lesssim \mathcal{X}(t),
\end{aligned}
\end{equation}
which together with (\ref{aLinfty}), (\ref{uv2}) and $(\ref{F1})$ yields
\begin{equation}\nonumber
\begin{aligned}
&\|G(n)\div u\|_{L^1_{t}(\dot{B}^{\frac{d}{2}}_{2,1})}^{\ell}\\
&\quad\lesssim \|G(n)\|_{L^{\infty}_{t}(\dot{B}^{\frac{d}{2}}_{2,1})}\|u\|_{L^1_{t}(\dot{B}^{\frac{d}{2}+1}_{2,1})}\\
&\quad\lesssim (         \|n\|_{L^{\infty}_{t}(\dot{B}^{\frac{d}{2}}_{2,1})}^{\ell}+\var \|n\|_{L^{\infty}_{t}(\dot{B}^{\frac{d}{2}+1}_{2,1})}^{h})(\|u\|_{L^1_{t}(\dot{B}^{\frac{d}{2}+1}_{2,1})}^{\ell}+\|u\|_{L^1_{t}(\dot{B}^{\frac{d}{2}+1}_{2,1})}^{h})\lesssim \mathcal{X}^2(t).
\end{aligned}
\end{equation}
In addition, we derive from (\ref{ninfty}) and the composition estimate (\ref{q1}) that
\begin{equation}\label{R12}
\begin{aligned}
&\var\|H(n)\|_{L^1_{t}(\dot{B}^{\frac{d}{2}+2}_{2,1})}^{\ell}\lesssim (         \|n\|_{L^{\infty}_{t}(\dot{B}^{\frac{d}{2}}_{2,1})}^{\ell}+\var \|n\|_{L^{\infty}_{t}(\dot{B}^{\frac{d}{2}+1}_{2,1})}^{h})(\var\|n\|_{L^1_{t}(\dot{B}^{\frac{d}{2}+2}_{2,1})}^{\ell}+\|n\|_{L^1_{t}(\dot{B}^{\frac{d}{2}+1}_{2,1})}^{h})\lesssim \mathcal{X}^2(t).
\end{aligned}
\end{equation}
It thus holds that
\begin{equation}\label{R1}
\begin{aligned}
        &\|R_{1}\|_{L^1_{t}(\dot{B}^{\frac{d}{2}}_{2,1})}^{\ell}\lesssim \mathcal{X}^2(t).
\end{aligned}
\end{equation}
Similarly to (\ref{R1}), it follows from $\eqref{m1n}_{1}$, (\ref{ninfty}), (\ref{uv2}) and (\ref{F1}) that
\begin{equation}\nonumber
\begin{aligned}
&\|\partial_{t}H(n)\|_{L^1_{t}(\dot{B}^{\frac{d}{2}}_{2,1})}^{\ell}\\
&\quad\lesssim \|u\cdot \nabla H(n)\|_{L^1_{t}(\dot{B}^{\frac{d}{2}}_{2,1})}^{\ell}+\|  \big( H(n)+c_{0}H'(n)+G(n)H'(n) \big)\div u\|_{L^1_{t}(\dot{B}^{\frac{d}{2}}_{2,1})}^{\ell}\lesssim \mathcal{X}^2(t).
\end{aligned}
\end{equation}
The above estimate and (\ref{R1}) imply
\begin{equation}\label{R2}
\begin{aligned}
&\|R_{2}\|_{L^1_{t}(\dot{B}^{\frac{d}{2}}_{2,1})}^{\ell}\lesssim \|R_{1}\|_{L^1_{t}(\dot{B}^{\frac{d}{2}}_{2,1})}^{\ell}+\|\partial_{t}H(n)\|_{L^1_{t}(\dot{B}^{\frac{d}{2}}_{2,1})}^{\ell}\lesssim\mathcal{X}^2(t).
\end{aligned}
\end{equation}
Finally, one deduces by (\ref{hl}), (\ref{R1}) and (\ref{uv2}) that
\begin{equation}\label{R3}
\begin{aligned}
&\|R_{3}\|_{L^1_{t}(\dot{B}^{\frac{d}{2}}_{2,1})}^{\ell}\lesssim \|R_{1}\|_{L^1_{t}(\dot{B}^{\frac{d}{2}}_{2,1})}^{\ell}+(\|u\|_{L^{\infty}_{t}(\dot{B}^{\frac{d}{2}}_{2,1})}^{\ell}+\var\|u\|_{L^{\infty}_{t}(\dot{B}^{\frac{d}{2}+1}_{2,1})}^{h}) \|u\|_{L^1_{t}(\dot{B}^{\frac{d}{2}+1}_{2,1})}\lesssim\mathcal{X}^2(t).
\end{aligned}
\end{equation}
Substituting the above estimates (\ref{inv})-(\ref{R3}) into (\ref{226}), we end up with (\ref{XLvw}). The proof of Lemma \ref{lemma22} is complete.
\end{proof}

The estimates of $(n,\varphi,v)$ implies the expected low-frequency estimates on the original unknowns $(n,\psi,u)$. This is illustrated by the following lemma.
\begin{lemma}\label{lemma23}
Let $T>0$ be any given time,  $J_{\var}=[-\log_2{\var}]+k$ for a sufficiently small constant $k$, and $(n,u,\psi)$ for $t\in(0,T)$ be the solution of System \eqref{m1n} with the initial data $(n_{0},u_{0},\phi_{0})$. Then, under the assumptions {\rm(\ref{ABmu})} and {\rm(\ref{aLinfty})}, we have
\begin{equation}\label{XL}
\begin{aligned}
&\|(u,\psi)\|_{\widetilde{L}^{\infty}_{t}(\dot{B}^{\frac{d}{2}}_{2,1})}^{\ell}+\|u\|_{L^{1}_{t}(\dot{B}^{\frac{d}{2}+1}_{2,1})}^{\ell}+\var^{-\frac{1}{2}}\|u\|_{\widetilde{L}^2_{t}(\dot{B}^{\frac{d}{2}}_{2,1})}^{\ell}\\
&\quad\quad+\var\|\psi\|_{L^{1}_{t}(\dot{B}^{\frac{d}{2}+2}_{2,1})}^{\ell}+\|\partial_{t}\psi\|_{L^{1}_{t}(\dot{B}^{\frac{d}{2}}_{2,1})}^{\ell}\\
&\quad\leq C\big( \mathcal{X}_{0}+C\mathcal{X}^2(t) \big),\quad\quad 0<t<T,
\end{aligned}
\end{equation}
where $\mathcal{X}(t)$, $\mathcal{X}_{0}$ and $(\varphi,v)$ are defined by \eqref{Xdef}, \eqref{a1} and \eqref{vw1}, respectively.
\end{lemma}
\begin{proof}
Since $u$ satisfies $(\ref{vw1})_{2}$, applying (\ref{XLvw}), (\ref{bessel}), (\ref{hl}), (\ref{R12}) and (\ref{F1}), we obtain
\begin{equation}\label{231}
\left\{
\begin{aligned}
\|u\|_{\widetilde{L}^{\infty}_{t}(\dot{B}^{\frac{d}{2}}_{2,1})}^{\ell}&\lesssim\|v\|_{\widetilde{L}^{\infty}_{t}(\dot{B}^{\frac{d}{2}}_{2,1})}^{\ell}+\var\|(n,H(n))\|_{\widetilde{L}^{\infty}_{t}(\dot{B}^{\frac{d}{2}+1}_{2,1})}^{\ell}+ \var\|\varphi\|_{\widetilde{L}^{\infty}_{t}(\dot{B}^{\frac{d}{2}+1}_{2,1})}^{\ell}\\
&\lesssim \|(n,v,\varphi)\|_{\widetilde{L}^{\infty}_{t}(\dot{B}^{\frac{d}{2}}_{2,1})}^{\ell}+\var\|n\|_{\widetilde{L}^{\infty}_{t}(\dot{B}^{\frac{d}{2}+1}_{2,1})}^{h}\\
&\lesssim \mathcal{X}_{0}+\mathcal{X}^2(t),\\
\|u\|_{L^{1}_{t}(\dot{B}^{\frac{d}{2}+1}_{2,1})}^{\ell}&\lesssim \|v\|_{L^{1}_{t}(\dot{B}^{\frac{d}{2}+1}_{2,1})}^{\ell}+\var\|(n,H(n))\|_{L^{1}_{t}(\dot{B}^{\frac{d}{2}+2}_{2,1})}^{\ell}+\var\|\varphi\|_{L^{1}_{t}(\dot{B}^{\frac{d}{2}+1}_{2,1})}^{\ell}\\
&\lesssim \frac{1}{\var} \|v\|_{L^{1}_{t}(\dot{B}^{\frac{d}{2}}_{2,1})}^{\ell} +\var\|(n,H(n))\|_{L^{1}_{t}(\dot{B}^{\frac{d}{2}+2}_{2,1})}^{\ell} +\|\varphi\|_{L^{1}_{t}(\dot{B}^{\frac{d}{2}}_{2,1})}^{\ell}\\
&\lesssim \mathcal{X}_{0}+\mathcal{X}^2(t).
\end{aligned}
\right.
\end{equation}
In addition, as $(\ref{good})_{2}$, $(\ref{vw1})_{1}$, (\ref{XLvw}), (\ref{R12}) and (\ref{F1}), it holds that
\begin{equation}\label{232}
\left\{
\begin{aligned}
\|\psi\|_{\widetilde{L}^{\infty}_{t}(\dot{B}^{\frac{d}{2}}_{2,1})}^{\ell}&\lesssim \|\varphi\|_{\widetilde{L}^{\infty}_{t}(\dot{B}^{\frac{d}{2}}_{2,1})}^{\ell}+\|(n,H(n))\|_{\widetilde{L}^{\infty}_{t}(\dot{B}^{\frac{d}{2}}_{2,1})}^{\ell}\\
&\lesssim \mathcal{X}_{0}+\mathcal{X}^2(t),\\
\var\|\psi\|_{L^{1}_{t}(\dot{B}^{\frac{d}{2}+2}_{2,1})}^{\ell}&\lesssim\|\varphi\|_{L^{1}_{t}(\dot{B}^{\frac{d}{2}+2}_{2,1})}^{\ell}+\var\|(n,H(n))\|_{L^{1}_{t}(\dot{B}^{\frac{d}{2}+2}_{2,1})}^{\ell}\\
&\lesssim \mathcal{X}_{0}+\mathcal{X}^2(t),\\
\|\partial_{t}\psi\|_{L^{1}_{t}(\dot{B}^{\frac{d}{2}}_{2,1})}^{\ell}&\lesssim \|\varphi\|_{L^{1}_{t}(\dot{B}^{\frac{d}{2}}_{2,1}\cap \dot{B}^{\frac{d}{2}+2}_{2,1})}^{\ell}\\
&\lesssim \mathcal{X}_{0}+\mathcal{X}^2(t).
\end{aligned}
\right.
\end{equation}
The combination of (\ref{XLvw}), (\ref{232}) and $\var^{-\frac{1}{2}}u=\var^{-\frac{1}{2}}v-\var^{\frac{1}{2}}(\nabla n-\mu\nabla \psi)$ gives rise to
\begin{equation}\label{235}
\begin{aligned}
\var^{-\frac{1}{2}}\|u\|_{\widetilde{L}^2_{t}(\dot{B}^{\frac{d}{2}}_{2,1})}^{\ell}&\lesssim  \var^{-\frac{1}{2}}\|v\|_{\widetilde{L}^{2}_{t}(\dot{B}^{\frac{d}{2}}_{2,1})}^{\ell}+\var^{\frac{1}{2}} \|(n,\psi)\|_{\widetilde{L}^2_{t}(\dot{B}^{\frac{d}{2}+1}_{2,1})}^{\ell}\\
&\lesssim \big ( \|v\|_{\widetilde{L}^{\infty}_{t}(\dot{B}^{\frac{d}{2}}_{2,1})}^{\ell} \big)^{\frac{1}{2}}\big ( \frac{1}{\var}\|v\|_{L^{1}_{t}(\dot{B}^{\frac{d}{2}}_{2,1})}^{\ell} \big)^{\frac{1}{2}}\\
&\quad+\big ( \|(n,\psi)\|_{\widetilde{L}^{\infty}_{t}(\dot{B}^{\frac{d}{2}}_{2,1})}^{\ell} \big)^{\frac{1}{2}}\big ( \var\|(n,\psi)\|_{L^{1}_{t}(\dot{B}^{\frac{d}{2}+2}_{2,1})}^{\ell} \big)^{\frac{1}{2}}\\
&\lesssim  \mathcal{X}_{0}+\mathcal{X}^2(t).
\end{aligned}
\end{equation}
 By (\ref{XLvw}) and (\ref{231})-(\ref{235}), (\ref{XL}) follows.
\end{proof}
This concludes the low-frequency analysis for the global existence. 

\subsection{High-frequency analysis}\label{subsection32}

In this subsection, we establish the corresponding a-priori estimates in high-frequency regime $ |\xi|\geq 2^{J_{\var}-1}$. To do so, we will differentiate in time the energy functional $\mathcal{L}_{j}(t)$ defined through \eqref{Lyapunov}.

First, based on some new observations about hyperbolic-parabolic dissipative structures for System (\ref{m1n}), we localize in frequencies and derive the following nonlinear energy estimates.
\begin{lemma}\label{lemma24}
Let $(n,u,\psi)$ be a solution of System \eqref{m1n}. For any $j\in\mathbb{Z}$ and the constant $\eta_{0}>0$ to be chosen later, it holds that
\begin{equation}\label{E1}
\begin{aligned}
&\frac{d}{dt}\int[ \frac{1}{2}|n_{j}|^2+\frac{1}{2\eta_{0}} 2^{-2j}|H(n)_{j}|^2+\frac{1}{2}w_{j}|u_{j}|^2+\frac{\mu b}{2c_{1}} |\psi_{j}|^2+\frac{\mu}{2c_{1}} |\nabla \psi_{j}|^2\\
&\quad-\mu n_{j} \psi_{j}-H(n)_{j} \psi_{j}]dx+\int  (\frac{1}{\var} w_{j}|u_{j}|^2+|\partial_{t}\psi_{j}|^2)dx\\
&\lesssim (1+\|w_{j}\|_{L^{\infty}})( \|\div u\|_{L^{\infty}} \|(n_{j},u_{j})\|_{L^2}+\|(\mathcal{R}_{1,j},\mathcal{R}_{2,j})\|_{L^2} ) \|(n_{j},u_{j},\psi_{j})\|_{L^2}\\
&\quad\quad+\|\partial_{t}w_{j}\|_{L^{\infty}}\|u_{j}\|_{L^2}^2+\| \nabla w_{j}\|_{L^{\infty}}(1+\|u\|_{L^{\infty}})\|u_{j}\|_{L^2}\|(n_{j},u_{j},\psi_{j})\|_{L^2} \\
&\quad\quad+2^{-j}\|\partial_{t}H(n)_{j}\|_{L^2}(2^{j}\|\psi_{j}\|_{L^2}+\frac{1}{\eta_{0}}2^{-j}\|H(n)_{j}\|_{L^2})\\
&\quad\quad+ \| u\|_{L^{\infty}} 2^{-j}\|\nabla (n_{j},H(n)_{j})\|_{L^2}2^{j}\|\psi_{j}\|_{L^2}  ,
\end{aligned}
\end{equation}
where $w_{j}$ is the the weight function
\begin{equation}\label{wj}
\begin{aligned}
&w_{j}:=c_{0}+\dot{S}_{j-1}G(n),
\end{aligned}
\end{equation}
and $\mathcal{R}_{i,j}$ $(i=1,2)$ are the commutator terms
\begin{equation}\label{R12j}
\left\{
\begin{aligned}
&\mathcal{R}_{1,j}:=\dot{S}_{j-1} u\cdot \nabla n_{j}-\dot{\Delta}_{j} (u\cdot \nabla n)+ \dot{S}_{j-1}G(n) \div u_{j} - \dot{\Delta}_{j} (G(n)\div u),\\
&\mathcal{R}_{2,j}:=\dot{S}_{j-1} u\cdot \nabla u_{j}-\dot{\Delta}_{j} (u\cdot \nabla u).
\end{aligned}
\right.
\end{equation}
\end{lemma}
\begin{proof}
Applying the operator $\dot{\Delta}_{j}$ to the equations (\ref{m1n}), we have
\begin{equation}\label{m1j}
\left\{
\begin{aligned}
&\partial_{t}n_{j} +\dot{S}_{j-1}u\cdot \nabla n_{j} +w_{j} \div u_{j} =\mathcal{R}_{1,j},\\
&\partial_{t}u_{j} +\dot{S}_{j-1}u\cdot \nabla u_{j}+\frac{1}{\var} u_{j} +\nabla n_{j} -\mu\nabla \psi_{j} =\mathcal{R}_{2,j},\\
&\partial_{t}\psi_{j} -\Delta\psi_{j} +b\psi_{j} -c_{1}n_{j} -H(n)_{j}=0.
\end{aligned}
\right.
\end{equation}
We take the $L^2$ inner product of $(\ref{m1j})_{1}$ with $n_{j}$ to obtain
 \begin{equation}\label{aj}
\begin{aligned}
&\frac{d}{dt}\int \frac{1}{2}|n_{j}|^2 dx+\int w_{j}\div u_{j} n_{j}dx=\int (\frac{1}{2} \dot{S}_{j-1}\div u| n_{j}|^2+\mathcal{R}_{1,j} n_{j})dx.
\end{aligned}
\end{equation}
Meanwhile, taking the $L^2$ inner product of $(\ref{m1j})_{2}$ with $w_{j}u_{j}$, we get
\begin{equation}\label{uj}
\begin{aligned}
&\frac{d}{dt}\int \frac{1}{2}w_{j} |u_{j}|^2dx+\int  [ \frac{1}{\var} w_{j}|u_{j}|^2-w_{j}n_{j}\div u_{j}+\mu w_{j}\psi_{j} \div u_{j}]dx\\
&=\int [\frac{1}{2} w_{j}\dot{S}_{j-1}\div u |u_{j}|^2+w_{j}\mathcal{R}_{2,j}\cdot u_{j} +\frac{1}{2}\partial_{t} w_{j} |u_{j}|^2\\
&\quad+\nabla w_{j}\cdot(\frac{1}{2}\dot{S}_{j-1}u |u_{j}|^2+n_{j} u_{j}-\mu\psi_{j}u_{j})]dx. 
\end{aligned}
\end{equation}
In addition, one derives after multiplying $(\ref{m1j})_{3}$ by $\partial_{t}\psi_{j}$ and integrating the resulting equation over $\mathbb{R}^{d}$ that
\begin{equation}\label{psij1111}
\begin{aligned}
&\frac{d}{dt}\int (\frac{b}{2}|\psi_{j}|^2+\frac{1}{2}|\nabla \psi_{j}|^2)dx+\int |\partial_{t}\psi_{j}|^2dx-c_{1}\int n_{j}\partial_{t}\psi_{j} dx\\
&=\int H(n)_{j} \partial_{t}\psi_{j}dx=\frac{d}{dt}\int H(n)_{j} \psi_{j}dx-\int \partial_{t}H(n)_{j}\psi_{j}dx.
\end{aligned}
\end{equation}
Due to $(\ref{m1j})_{1}$, it holds that
\begin{equation}\label{apsij}
\begin{aligned}
-\int  n_{j}\partial_{t}\psi_{j} dx&=-\frac{d}{dt}\int n_{j}\psi_{j}dx-\int w_{j}\div u_{j} \psi_{j} dx\\
&\quad+\int (-\dot{S}_{j-1}u\cdot \nabla n_{j}+\mathcal{R}_{1,j})\psi_{j}dx.
\end{aligned}
\end{equation}
Finally, we have
\begin{equation}\label{apsij112}
\begin{aligned}
&\frac{d}{dt}\int \frac{2^{-2j}}{2} |H(n)_{j}|^2dx\leq 2^{-j}\|\partial_{t}H(n)_{j}\|_{L^2}2^{-j}\|H(n)_{j}\|_{L^2}.
\end{aligned}
\end{equation}
Thus the combination of (\ref{aj})-(\ref{apsij112}) leads to $(\ref{E1})$. The proof of Lemma \ref{lemma24} is complete.
\end{proof}
\begin{remark}
Note that the crucial point in the above proof is the cancellation of the term \newline$\displaystyle\int w_j\div u_j \psi_j dx$ that appears in \eqref{aj} thanks to the same term with the opposite sign and different coefficients appearing in \eqref{apsij}.
\end{remark}

\begin{remark}
The main difficulty is to deal with the nonlinear term $H(n)$. At first glance, we need
\begin{align}\label{thisprop}
\int H(n)_{j} \partial_{t}\psi_{j}dx\leq \frac{1}{4}\|\partial_{t}\psi_{j}\|_{L^2}^2+\|H(n)_{j}\|_{L^2}^2.
\end{align}
However, the property {\rm(\ref{thisprop})} may not be applied to derive the $L^1$-time type estimates since we can only obtain $2^{-2j}\|H(n)_{j}\|_{L^2}^2$ from the dissipation $\mathcal{H}_{j}(t)$, see Lemma \ref{ly}. The main ingredient here is that we insert the nonlinear term $2^{-2j}\|H(n)_{j}\|_{L^2}^2$ into the energy functional to cancel the term $\int H(n)_{j} \partial_{t}\psi_{j}dx$. It should be noted that {\rm(\ref{thisprop})} will be useful to show the uniqueness in Subsection \ref{subsection34}.
\end{remark}

Next, we differentiate in time the terms of lower order in the Lyapunov functional. It gives us the dissipation estimates for the non-directly damped unknowns $n$ and $\psi$.
\begin{lemma}\label{lemma25}
Let $(n,u,\psi)$ be a solution of System \eqref{m1n}. For any $j\in\mathbb{Z}$, we have
\begin{equation}\label{E2}
\begin{aligned}
&\frac{d}{dt} \int (\frac{\mu}{2c_{1}}|\nabla \psi_{j}|^2+u_{j} \cdot\nabla n_{j})dx+\int  [|\nabla n_{j}|^2+\frac{\mu b}{c_{1}}|\nabla \psi_{j}|^2\\
&\quad+\frac{\mu}{c_{1}}|\Delta\psi_{j}|^2-2\mu \nabla n_{j}\cdot \nabla \psi_{j}-w_{j}|\div u_{j}|^2+\frac{1}{\var} u_{j} \cdot\nabla n_{j} ] dx\\
&\lesssim \|\nabla u\|_{L^{\infty}} \|\nabla n_{j}\|_{L^2}\|u_{j}\|_{L^2}+\|(\mathcal{R}_{1,j},\mathcal{R}_{2,j})\|_{L^2}\|\nabla (n_{j},u_{j})\|_{L^2}\\
&\quad+\|H(n)_{j}\|_{L^2}\|\Delta\psi_{j}\|_{L^2}.
\end{aligned}
\end{equation}
\end{lemma}
\begin{proof}
One concludes from $(\ref{m1j})_{1}$-$(\ref{m1j})_{2}$ that
\begin{equation}\label{251}
\begin{aligned}
&\frac{d}{dt} \int u_{j} \cdot\nabla n_{j}dx+\int  [ |\nabla n_{j}|^2-w_{j}|\div u_{j}|^2+\frac{1}{\var} u_{j} \cdot\nabla n_{j}-\mu \nabla n_{j}\cdot \nabla \psi_{j} ] dx\\
&=\int [-(\nabla \dot{S}_{j-1}u\cdot \nabla n_{j})\cdot u_{j}-\dot{S}_{j-1}\div u  u_{j} \cdot \nabla n_{j}-\mathcal{R}_{1,j}\div u_{j}+\mathcal{R}_{2,j}\cdot\nabla n_{j}]dx,
\end{aligned}
\end{equation}
where we has used
\begin{equation}\nonumber
\begin{aligned}
&\int [\nabla (\dot{S}_{j-1}u\cdot \nabla n_{j})\cdot u_{j}+(\dot{S}_{j-1}u\cdot \nabla u_{j})\cdot \nabla n_{j}]dx\\
&\quad=\int [(\nabla \dot{S}_{j-1}u\cdot \nabla n_{j})\cdot u_{j}-\dot{S}_{j-1}\div u  u_{j} \cdot \nabla n_{j}]dx.
\end{aligned}
\end{equation}
 Thence we get after multiplying $(\ref{m1j})_{3}$ by $-\Delta\psi_{j}$ and integrating the resulting equation by parts that
\begin{equation}\label{252}
\begin{aligned}
&\frac{d}{dt}\int \frac{1}{2}|\nabla \psi_{j}|^2dx+\int (b|\nabla \psi_{j}|^2+|\Delta\psi_{j}|^2-c_{1}\nabla n_{j}\cdot \nabla \psi_{j}+ H(n)_{j} \Delta\psi_{j})dx=0. 
\end{aligned}
\end{equation}
By (\ref{251})-(\ref{252}), (\ref{E2}) follows. 
\end{proof}

Then, we have the following high-frequency estimates:
\begin{lemma}\label{lemma26} Let $T>0$ be any given time, $J_{\var}=[-\log_2{\var}]+k$ for a sufficiently small constant $k$, and $(n,u,\psi)$ for $t\in(0,T)$ be the solution of System \eqref{m1n} with the initial data $(n_{0},u_{0},\phi_{0})$. Then, under the assumptions {\rm(\ref{ABmu})} and {\rm(\ref{aLinfty})}, $(n,u,\psi)$ satisfies
\begin{equation}
\begin{aligned}
&\mathcal{X}_{H}(t)\leq C\big( \mathcal{X}_{0}+C\mathcal{X}^2(t) \big),\quad\quad  0<t<T,\label{XH}
\end{aligned}
\end{equation}
where $\mathcal{X}(t)$, $\mathcal{X}_{H}(t)$ and $\mathcal{X}_{0}$ are defined by  are defined by \eqref{Xdef}, \eqref{Xdef} and \eqref{a1}, respectively.

\end{lemma}
\begin{proof}
Let the constants $\alpha$ and $\eta_{0}$ be chosen later. We recall that
\begin{equation}
\left\{
\begin{aligned}
&\mathcal{L}_{j}(t):=\var\int[ \frac{1}{2}|n_{j}|^2+\frac{2^{-2j}}{2\eta_{0}} |H(n)_{j}|^2+\frac{1}{2}w_{j}|u_{j}|^2+\frac{\mu b}{2c_{1}} |\psi_{j}|^2+\frac{\mu}{2c_{1}} |\nabla \psi_{j}|^2\nonumber\\
&\quad\quad\quad\quad\quad-\mu n_{j} \psi_{j}-H(n)_{j} \psi_{j}]dx +\eta_{0}  2^{-2j}\int (\frac{\mu}{2c_{1}}|\nabla \psi_{j}|^2+u_{j} \cdot\nabla n_{j} )dx,\nonumber\\
&\mathcal{H}_{j}(t):= \var\int  (\frac{1}{\var} w_{j}|u_{j}|^2+|\partial_{t}\psi_{j}|^2)dx+\eta_{0}  2^{-2j}\int \big{(}|\nabla n_{j}|^2+\frac{\mu b}{c_{1}}|\nabla \psi_{j}|^2+\frac{\mu}{c_{1}}|\Delta\psi_{j}|^2\nonumber\\
&\quad\quad\quad\quad\quad-2\mu \nabla n_{j}\cdot \nabla \psi_{j}-w_{j}|\div u_{j}|^2+\frac{1}{\var} u_{j} \cdot\nabla n_{j} \big{)}dx.\nonumber
\end{aligned}
\right.
\end{equation}
 By (\ref{aLinfty}) and (\ref{wj}), there is a constant $C_{*}>0$ independent of time and $\var$ such that
\begin{equation}\label{wjc0}
\begin{aligned}
\frac{c_{0}}{2}\leq c_{0}-C_{*}\mathcal{X}_{0}\leq w_{j}\leq c_{0}+C_{*}\mathcal{X}_{0}\leq \frac{3c_{0}}{2},
\end{aligned}
\end{equation}
provided 
\begin{align}
\mathcal{X}_{0} \leq \alpha\leq \frac{C_{*}}{2}.\label{var0}
\end{align}
One gets from (\ref{hl}) that
\begin{equation}\label{uinfty}
\begin{aligned}
&\|u\|_{L^{\infty}_{t}(L^{\infty})}+\var\|\nabla u\|_{L^{\infty}_{t}(L^{\infty})}\\
&\quad\lesssim \|u\|_{L^{\infty}_{t}(\dot{B}^{\frac{d}{2}}_{2,1})}^{\ell}+\var \|u\|_{L^{\infty}_{t}(\dot{B}^{\frac{d}{2}+1}_{2,1})}^{h}\lesssim \mathcal{X}(t).
\end{aligned}
\end{equation}
In accordance with (\ref{E1}), (\ref{E2}), (\ref{wjc0}), (\ref{uinfty}), the Bernstein inequality and the fact $2^{-j}\lesssim \var$ for any $j\geq J_{\var}-1$, we have
\begin{equation}\label{ly1}
    \begin{aligned}
       & \frac{d}{dt}\mathcal{L}_{j}(t)+\mathcal{H}_{j}(t)\\
       &\lesssim \var \big{(} \|\nabla u\|_{L^{\infty}}\|(n_{j},u_{j})\|_{L^2}+\| u\|_{L^{\infty}}\|(n_{j},H(n)_{j})\|_{L^2}\\
       &\quad+\|(\partial_{t} G(n),\nabla G(n))\|_{L^{\infty}}\|u_{j}\|_{L^2}+\|(\mathcal{R}_{1,j},\mathcal{R}_{2,j})\|_{L^2}\\
       &\quad+\|H(n)_{j}\|_{L^2}+\frac{2^{-j}}{\eta_{0}}\|\partial_{t}H(n)_{j}\|_{L^2}\big{)}\|(n_{j},u_{j},\psi_{j},\nabla\psi_{j},2^{-j}H(n)_{j})\|_{L^2}.
    \end{aligned}
\end{equation}

We now show the following lemma.
\begin{lemma} \label{ly} Let the assumptions of Lemma \ref{lemma26} be in force. Then it holds for any $j\geq J_{\var}-1$ that
\begin{equation}\label{267}
\left\{
\begin{aligned}
&\mathcal{L}_{j}(t)\sim \var\|(n_{j},u_{j},\psi_{j},\nabla\psi_{j},2^{-j}H(n)_{j})\|_{L^2}^2,\\
&\mathcal{H}_{j}(t)\gtrsim \frac{1}{\var}\mathcal{L}_{j}(t).
\end{aligned}
\right.
\end{equation} 
\end{lemma}
\begin{proof}
Note that we have
\begin{equation}\nonumber
 |n_{j}|^2+\frac{\mu b}{c_{1}}|\psi_{j}|^2-2\mu n_{j} \psi_{j}=(n_{j},\psi_{j})\mathbb{M}(n_{j},\psi_{j})^{\top},\quad\quad \mathbb{M}:=
 \left(
 \begin{matrix}  
    1 & -\mu  \\
   -\mu & \dfrac{\mu b}{c_{1}} 
  \end{matrix}
  \right).
\end{equation}
It is easy to verify that $\mathbb{M}$ is a positive definite matrix under the condition $\dfrac{\mu c_{1}}{b}<1$, i.e. $\dfrac{\mu a}{b}\bar{\rho}<P'(\bar{\rho})$. Hence, there holds
\begin{equation}\label{2631}
\begin{aligned}
& |n_{j}|^2+\frac{\mu b}{c_{1}}|\psi_{j}|^2-2\mu n_{j} \psi_{j}\gtrsim |n_{j}|^2+|\psi_{j}|^2.
\end{aligned}
\end{equation}
Similarly, one has
\begin{equation}\label{263}
\begin{aligned}
& |\nabla n_{j}|^2+ \frac{\mu b}{c_{1}}|\nabla \psi_{j}|^2-2\mu \nabla n_{j}\cdot \nabla \psi_{j}\gtrsim  |\nabla n_{j}|^2+|\nabla \psi_{j}|^2.
\end{aligned}
\end{equation}
For any $j\geq J_{\var}-1$, it follows by (\ref{wjc0}), $(\ref{2631})$, the Bernstein inequality and $2^{-j}\leq \var$ that
\begin{equation}\nonumber
\left\{
\begin{aligned}
&\mathcal{L}_{j}(t)\leq\var \int  [C (1-\eta_{0})(|n_{j}|^2+|u_{j}|^2+|\psi_{j}|^2+2^{2j}|\psi_{j}|^2)+2^{-2j}(\frac{1}{\eta_{0}}+C)|H(n)_{j}|^2]dx,  \\
&\mathcal{L}_{j}(t)\geq\var \int  [ \frac{1}{C} (1-\eta_{0})(|n_{j}|^2+|u_{j}|^2+|\psi_{j}|^2+2^{2j}|\psi_{j}|^2)+2^{-2j}(\frac{1}{\eta_{0}}-C)|H(n)_{j}|^2]dx.
\end{aligned}
\right.
\end{equation}
Noting that
\begin{equation}\nonumber
\begin{aligned}
|\Delta\psi_{j}|^2\geq \frac{1}{2}|H(n)_{j}+c_{1}n_{j}-b\psi_{j}|^2-|\partial_{t}\psi_{j}|^2,
\end{aligned}
\end{equation}
we deduce from $\eqref{263}$, the Bernstein inequality and $2^{-j}\leq \var$ for any $j\geq J_{\var}-1$ that
\begin{equation}\nonumber
\begin{aligned}
\mathcal{H}_{j}(t)&\geq \frac{1}{C} \int [  (1-\eta_{0}) |u_{j}|^2+\var(1-\eta_{0})|\partial_{t}\psi_{j}|^2+2^{-2j}\eta_{0} (|\nabla n_{j}|^2+ |\nabla \psi_{j}|^2+|\Delta\psi_{j}|^2)dx \\
&\geq\frac{1}{C} \int [  (1-\eta_{0}) |u_{j}|^2+\var(1-\eta_{0})|\partial_{t}\psi_{j}|^2\\
&\quad\quad+\frac{\eta_{0}}{C} (|n_{j}|^2+ |\psi_{j}|^2+2^{2j}|\psi_{j}|^2+2^{-2j} |H(n)_{j}+c_{1}n_{j}-b\psi_{j}|^2)]dx \\
&\geq \frac{1}{C} \int [ (1-\eta_{0}) |u_{j}|^2+\var(1-\eta_{0})|\partial_{t}\psi_{j}|^2+\frac{\eta_{0}}{C} (|n_{j}|^2+ |\psi_{j}|^2+2^{2j}|\psi_{j}|^2+ 2^{-2j}|H(n)_{j}|^2)]dx.
\end{aligned}
\end{equation}
Thus, choosing a suitably small constant $\eta_{0}\in(0,1)$, we derive  (\ref{267}).
\end{proof}

We now resume the proof of Lemma \ref{lemma26}. Combining (\ref{ly1})-(\ref{267}) together yields
\begin{equation}\label{264}
\begin{aligned}
&\frac{d}{dt}\mathcal{L}_{j}(t)+\frac{1}{\var}\mathcal{L}_{j}(t)\\
&\lesssim\Big{(} \|\nabla u\|_{L^{\infty}}\var\|(n_{j},u_{j})\|_{L^2}+\|u\|_{L^{\infty}}\var\|n_{j}\|_{L^2}\\
&\quad+\|(\partial_{t} G(n),\nabla G(n))\|_{L^{\infty}}\var\|u_{j}\|_{L^2}+\|(\mathcal{R}_{1,j},\mathcal{R}_{2,j})\|_{L^2}\\
&\quad+\|H(n)_{j}\|_{L^2}+2^{-j}\|\partial_{t}H(n)_{j}\|_{L^2}\Big{)} \sqrt{\frac{1}{\var}\mathcal{L}_{j}(t)},\quad j\geq J_{\var}-1.
\end{aligned}
\end{equation}
After dividing both sides of (\ref{264}) by $(\dfrac{1}{\var}\mathcal{L}_{j}(t)+\eta^2)^{\frac{1}{2}}$ for $\eta>0$, integrating the resulting inequality over $[0,t]$ and then taking the limit as $\eta\rightarrow0$, we get
\begin{equation}\label{26555}
\begin{aligned}
&\var\|(n_{j},u_{j},\psi_{j},\nabla\psi_{j})\|_{L^2}+ \int_{0}^{t}       \|(n_{j},u_{j},\psi_{j},\nabla\psi_{j})\|_{L^2} d\tau\\
&\quad\leq \var\|(n_{j},u_{j},\psi_{j},\nabla\psi_{j})(0)\|_{L^2}+\var2^{-j}\|H(n)_{j}(0)\|_{L^2}\\
&\quad\quad+\int_{0}^{t} \Big{(} \var\|\nabla u\|_{L^{\infty}}\|(n_{j},u_{j})\|_{L^2}+ \var\| u\|_{L^{\infty}}\|(n_{j},H(n)_{j})\|_{L^2}\\
&\quad\quad+ \var\|\big(\nabla G(n),\partial_{t}G(n)\big)\|_{L^{\infty}} \| u_{j}\|_{L^2}+ \var\|(\mathcal{R}_{1,j},\mathcal{R}_{2,j})\|_{L^2}\\
&\quad\quad+\var\|H(n)_{j}\|_{L^2}+2^{-j}\|\partial_{t}H(n)_{j}\|_{L^2}\Big{)} d\tau.
\end{aligned}
\end{equation}
After multiplying (\ref{26555}) by $2^{j(\frac{d}{2}+1)}$ and summing over $j\geq J_{\var}-1$, we infer that
\begin{equation}\label{26677}
\begin{aligned}
&\var\|(n,u,\psi,\nabla\psi)\|_{\widetilde{L}^{\infty}_{t}(\dot{B}^{\frac{d}{2}+1}_{2,1})}^{h} +\|(n,u,\psi,\nabla\psi)\|_{L^{1}_{t}(\dot{B}^{\frac{d}{2}+1}_{2,1})}^{h}\\
&\quad\lesssim \mathcal{X}_{0}+\|\nabla u\|_{L^1_{t}(L^{\infty})} \var\|(n,u)\|_{L^{\infty}_{t}(\dot{B}^{\frac{d}{2}+1}_{2,1})}^{h}+\|u\|_{L^{2}_{t}(L^{\infty})}\|(n,H(n))\|_{L^{2}_{t}(\dot{B}^{\frac{d}{2}+1}_{2,1})}^{h}\\
&\quad\quad+\var\|\big(\nabla G(n),\partial_{t} G(n)\big)\|_{L^{\infty}_{t}(L^{\infty})} \|u\|_{L^1_{t}(\dot{B}^{\frac{d}{2}+1}_{2,1})}^{h}\\
&\quad\quad+\var\sum_{j\geq J_{\var}-1} 2^{j(\frac{d}{2}+1)} \|(\mathcal{R}_{1,j},\mathcal{R}_{2,j})\|_{L^1_{t}(L^2)}\\
&\quad\quad+\|H(n)\|_{L^1_{t}(\dot{B}^{\frac{d}{2}+1}_{2,1})}^{h}+\|\partial_{t}H(n)\|_{L^1_{t}(\dot{B}^{\frac{d}{2}}_{2,1})}^{h}.
\end{aligned}
\end{equation}
The nonlinear terms on the right-hand side of (\ref{26677}) can be estimated as follows. First, it is easy to see that
\begin{equation}\label{I1}
\begin{aligned}
&\|\nabla u\|_{L^1_{t}(L^{\infty})} \var\|(n,u)\|_{L^{\infty}_{t}(\dot{B}^{\frac{d}{2}+1}_{2,1})}^{h}\\
&\quad\leq \big{(} \|u\|_{L^1_{t}(\dot{B}^{\frac{d}{2}+1}_{2,1})}^{\ell}+\|u\|_{L^1_{t}(\dot{B}^{\frac{d}{2}+1}_{2,1})}^{h}\big{)} \var \|(n,u)\|_{L^{\infty}_{t}(\dot{B}^{\frac{d}{2}+1}_{2,1})}^{h}\lesssim \mathcal{X}^2(t).
\end{aligned}
\end{equation}
Due to (\ref{ninfty}), (\ref{F1}) and standard interpolations, one has
\begin{equation}\label{I2}
\begin{aligned}
&\|u\|_{L^{2}_{t}(L^{\infty})}\|(n,H(n))\|_{L^{2}_{t}(\dot{B}^{\frac{d}{2}+1}_{2,1})}^{h}\\
&\quad\lesssim \var^{-\frac{1}{2}}\big( \|u\|_{L^2_{t}(\dot{B}^{\frac{d}{2}}_{2,1})}^{\ell}+\|u\|_{L^{2}_{t}(\dot{B}^{\frac{d}{2}}_{2,1})}^{h})\var^{\frac{1}{2}} (\|n\|_{L^2_{t}(\dot{B}^{\frac{d}{2}+1}_{2,1})}^{\ell}+\|n\|_{L^2_{t}(\dot{B}^{\frac{d}{2}+1}_{2,1})}^{h})\lesssim \mathcal{X}^2(t).
\end{aligned}
\end{equation}
It also holds by ${\rm{\eqref{m1n}_{1}}}$, (\ref{hl}) and (\ref{ninfty}) that
\begin{equation}\label{I3}
\begin{aligned}
&\var\|\big(\nabla G(n),\partial_{t} G(n)\big)\|_{L^{\infty}_{t}(L^{\infty})} \|u\|_{L^1_{t}(\dot{B}^{\frac{d}{2}+1}_{2,1})}^{h}\\
&\quad\lesssim   \big( \var\|\nabla n\|_{L^{\infty}_{t}(L^{\infty})}(1+\|u\|_{L^{\infty}_{t}(L^{\infty})})\\
&\quad\quad+\var\|\nabla u\|_{L^2_{t}(L^{\infty})}(1+\|n\|_{L^{\infty}_{t}(L^{\infty})}) \big) \|u\|_{L^1_{t}(\dot{B}^{\frac{d}{2}+1}_{2,1})}^{h}\\
 &\quad\lesssim \big( \|(n,u)\|_{L^{\infty}_{t}(\dot{B}^{\frac{d}{2}}_{2,1})}^{\ell} +\var \|(n,u)\|_{L^{\infty}_{t}(\dot{B}^{\frac{d}{2}+1}_{2,1})}^{h} \big)\|u\|_{L^1_{t}(\dot{B}^{\frac{d}{2}+1}_{2,1})}^{h}
\lesssim \mathcal{X}^2(t).
\end{aligned}
\end{equation}
We turn to estimate the commutator terms. It follows from (\ref{hl}) and (\ref{commutator}) that
\begin{equation}\nonumber
\begin{aligned}
&\var\sum_{j\geq J_{\var}-1} 2^{j(\frac{d}{2}+1)} \big( \|\dot{S}_{j-1}u\cdot \nabla n_{j}-\dot{\Delta}_{j}(u\cdot\nabla n)\|_{L^1_{t}(L^2)}+\|\dot{S}_{j-1}u\cdot \nabla u_{j}-\dot{\Delta}_{j}(u\cdot\nabla u)\|_{L^1_{t}(L^2)}\big)\\
&\quad\lesssim \|u\|_{L^1_{t}(\dot{B}^{\frac{d}{2}+1}_{2,1})}\var\|(n,u)\|_{L^{\infty}_{t}(\dot{B}^{\frac{d}{2}+1}_{2,1})}\\
&\quad\lesssim (\|u\|_{L^1_{t}(\dot{B}^{\frac{d}{2}+1}_{2,1})}^{\ell}+\|u\|_{L^1_{t}(\dot{B}^{\frac{d}{2}+1}_{2,1})}^{h}) (\|(n,u)\|_{L^{\infty}_{t}(\dot{B}^{\frac{d}{2}}_{2,1})}^{\ell}+\var\|(n,u)\|_{L^{\infty}_{t}(\dot{B}^{\frac{d}{2}+1}_{2,1})}^{h})\lesssim \mathcal{X}^2(t).
\end{aligned}
\end{equation}
And we obtain by using (\ref{ninfty}), (\ref{F1}) and $G(0)=0$ that
\begin{equation}\nonumber
\begin{aligned}
&\var\|G(n)\|_{L^{\infty}_{t}(\dot{B}^{\frac{d}{2}+1}_{2,1})}\lesssim \|n\|_{L^{\infty}_{t}(\dot{B}^{\frac{d}{2}}_{2,1})}^{\ell}+\var\|n\|_{L^{\infty}_{t}(\dot{B}^{\frac{d}{2}+1}_{2,1})}^{h}\lesssim\mathcal{X}(t),
\end{aligned}
\end{equation}
which together with (\ref{lh}) and (\ref{commutator}) leads to
\begin{equation}\nonumber
\begin{aligned}
&\var\sum_{j\geq J_{\var} -1} 2^{j(\frac{d}{2}+1)} \|\dot{S}_{j-1}G(n) \div u_{j} - \dot{\Delta}_{j} (G(n)\div u))\|_{L^1_{t}(L^2)}\\
&\quad\lesssim \|u\|_{L^1_{t}(\dot{B}^{\frac{d}{2}+1}_{2,1})}\var\|G(n)\|_{L^{\infty}_{t}(\dot{B}^{\frac{d}{2}+1}_{2,1})}\lesssim\mathcal{X}^2(t).
\end{aligned}
\end{equation}
Therefore, the commutator terms can be controlled by
\begin{equation}
    \begin{aligned}
        \var\sum_{j\geq J_{\var}-1} 2^{j(\frac{d}{2}+1)} \|(\mathcal{R}_{1,j},\mathcal{R}_{2,j})\|_{L^1_{t}(L^2)}\lesssim \mathcal{X}^2(t).
    \end{aligned}
\end{equation}
With the help of the composition estimate (\ref{q2}) for the quadratic function $H(n)$, it holds that
\begin{equation}\label{I5}
\begin{aligned}
&\|H(n)\|_{L^1_{t}(\dot{B}^{\frac{d}{2}+1}_{2,1})}^{h}\\
&\quad\lesssim \big{(} \|n\|_{L^{\infty}_{t}(\dot{B}^{\frac{d}{2}}_{2,1})}^{\ell}+\var\|n\|_{L^{\infty}_{t}(\dot{B}^{\frac{d}{2}+1}_{2,1})}^{h}\big{)} \big{(} \var\|n\|_{L^1_{t}(\dot{B}^{\frac{d}{2}+2}_{2,1})}^{\ell}+\|n\|_{L^1_{t}(\dot{B}^{\frac{d}{2}+1}_{2,1})}^{h} \big{)}\\
&\quad\lesssim \mathcal{X}^2(t).
\end{aligned}
\end{equation}
Finally, one deduces by ${\rm{\eqref{m1n}_{1}}}$, (\ref{ninfty}), (\ref{hl}), (\ref{uv2}) and (\ref{F1}) that
\begin{equation}\label{I6}
\begin{aligned}
&\|\partial_{t}H(n)\|_{L^1_{t}(\dot{B}^{\frac{d}{2}}_{2,1})}^{h}\\
&\quad\lesssim\big( \|H'(n)-H'(0)\|_{L^{\infty}_{t}(\dot{B}^{\frac{d}{2}}_{2,1})}+H'(0)\big)\|\partial_{t}n\|_{L^1_{t}(\dot{B}^{\frac{d}{2}}_{2,1})}\\
&\quad\lesssim \|n\|_{L^{\infty}_{t}(\dot{B}^{\frac{d}{2}}_{2,1})}\big( \|u\|_{L^2_{t}(\dot{B}^{\frac{d}{2}}_{2,1})}\|n\|_{L^2_{t}(\dot{B}^{\frac{d}{2}+1}_{2,1})}+(1+\|n\|_{L^{\infty}_{t}(\dot{B}^{\frac{d}{2}}_{2,1})})\|u\|_{L^1_{t}(\dot{B}^{\frac{d}{2}+1}_{2,1})}\big)\\
&\quad\lesssim \mathcal{X}^2(t).
\end{aligned}
\end{equation}
Inserting the above estimates (\ref{I1})-(\ref{I6}) into (\ref{26677}), we derive
\begin{equation}\label{271}
\begin{aligned}
&\var\|(n,u,\psi,\nabla\psi)\|_{\widetilde{L}^{\infty}_{t}(\dot{B}^{\frac{d}{2}+1}_{2,1})}^{h}+\|(n,u,\psi,\nabla\psi)\|_{L^{1}_{t}(\dot{B}^{\frac{d}{2}+1}_{2,1})}^{h}\lesssim \mathcal{X}_{0}+\mathcal{X}^2(t).
\end{aligned}
\end{equation}

Furthermore, since $u$ satisfies
\begin{equation}\nonumber
\begin{aligned}
&\|u\|_{\widetilde{L}^{\infty}_{t}(\dot{B}^{\frac{d}{2}}_{2,1})}^{h}\lesssim \var \|u\|_{\widetilde{L}^{\infty}_{t}(\dot{B}^{\frac{d}{2}+1}_{2,1})}^{h},\quad\quad\frac{1}{\var} \|u\|_{L^{1}_{t}(\dot{B}^{\frac{d}{2}}_{2,1})}^{h}\lesssim \|u\|_{L^{1}_{t}(\dot{B}^{\frac{d}{2}+1}_{2,1})}^{h},
\end{aligned}
\end{equation}
one infers from (\ref{271}) that
\begin{equation}\label{274}
\begin{aligned}
&\var^{-\frac{1}{2}}\|u\|_{\widetilde{L}^2_{t}(\dot{B}^{\frac{d}{2}}_{2,1})}^{h}\lesssim \big( \|u\|_{\widetilde{L}^{\infty}_{t}(\dot{B}^{\frac{d}{2}}_{2,1})}^{h} \big)^{\frac{1}{2}}\big( \frac{1}{\var}\|u\|_{L^{1}_{t}(\dot{B}^{\frac{d}{2}}_{2,1})}^{h} \big)^{\frac{1}{2}}\lesssim \mathcal{X}_{0}+\mathcal{X}^2(t).
\end{aligned}
\end{equation}
According to (\ref{hl}), (\ref{I5})-(\ref{271}) and maximal regularity estimates for the heat equation ${\rm\eqref{m1n}_{3}}$ (e.g., see \cite[Page 157]{bahouri1}), it follows that
\begin{equation}\label{275}
\begin{aligned}
&\|\psi\|_{L^1_{t}(\dot{B}^{\frac{d}{2}+3}_{2,1})}^{h}+\|\partial_{t}\psi\|_{L^1_{t}(\dot{B}^{\frac{d}{2}+1}_{2,1})}^{h}\\
&\quad\lesssim \|\psi_{0}\|_{\dot{B}^{\frac{d}{2}+1}_{2,1}}^{h}+\|-b\psi+c_{1}n+H(n)\|_{L^1_{t}(\dot{B}^{\frac{d}{2}+1}_{2,1})}^{h}\\
&\quad\lesssim \var\|\psi_{0}\|_{\dot{B}^{\frac{d}{2}+2}_{2,1}}^{h}+\|(n,\psi)\|_{L^1_{t}(\dot{B}^{\frac{d}{2}+1}_{2,1})}^{h}+\|H(n)\|_{L^1_{t}(\dot{B}^{\frac{d}{2}+1}_{2,1})}^{h} \lesssim \mathcal{X}_{0}+\mathcal{X}^2(t).
\end{aligned}
\end{equation}
 Collecting (\ref{271})-(\ref{275}), we have (\ref{XH}) and finish the proof of Lemma \ref{lemma26}.
 \end{proof}

 \subsection{Global existence}\label{subsection33}

In this subsection, we construct a local Friedrichs approximation (see, e.g., \cite[Page 440]{bahouri1}) and extend the local approximate sequence to a global one by the a-priori estimates established in Subsections \ref{subsection31}-\ref{subsection32}. Then we show the convergence of the approximate sequence to the expected global solution to the Cauchy problem for System (\ref{m1n}).

 \underline{\it\textbf{Proof of Theorem \ref{theorem11}:~}}Denote the Friedrichs projector
 \begin{equation}\nonumber
\begin{aligned}
&\dot{\mathbb{E}}_{q}f:=\mathcal{F}^{-1}(\mathbf{1}_{\mathcal{C}_{q}}\mathcal{F}f),\quad\forall f\in L^2_{q},
\end{aligned}
\end{equation}
where $L^2_{q}$ is the set of $L^2$ functions spectrally supported in the annulus $\mathcal{C}_{q}:=\{\xi\in\mathbb{R}^{d}~|~\frac{1}{q}\leq|\xi|\leq q\}$ endowed with the standard $L^2$ topology, and $\mathbf{1}_{\mathcal{C}_{q}}$ is the characteristic function on the annulus $\mathcal{C}_{q}$.

We are ready to solve the following approximate problem for $q\geq1$:
\begin{equation}\label{m1nn}
\left\{
\begin{aligned}
&\partial_{t}n^{q}+\dot{\mathbb{E}}_{q}(u^{q}\cdot\nabla n^{q}) +c_{0}\div \dot{\mathbb{E}}_{q} u^{q}+\dot{\mathbb{E}}_{q}(G(n^{q})\div u^{q})=0,\\
&\partial_{t}u^{q}+\dot{\mathbb{E}}_{q}(u^{q}\cdot \nabla u^{q})+\frac{1}{\var} \dot{\mathbb{E}}_{q} u^{q}+\nabla \dot{\mathbb{E}}_{q} n^{q}-\mu\nabla \dot{\mathbb{E}}_{q} \psi^{q}=0,\\
&\partial_{t}\psi^{q}-\Delta \dot{\mathbb{E}}_{q}\psi^{q}+b\dot{\mathbb{E}}_{q} \psi^{q}-c_{1}\dot{\mathbb{E}}_{q} n^{q}-\dot{\mathbb{E}}_{q}H(n^{q})=0,\quad\quad x\in\mathbb{R}^{d},\quad t>0,\\
&(n^{q},u^{q},\psi^{q})(x,0)=(\dot{\mathbb{E}}_{q}n_{0}, \dot{\mathbb{E}}_{q}u_{0},\dot{\mathbb{E}}_{q}\psi_{0})(x), \quad\quad  \quad\quad\quad ~ x\in\mathbb{R}^{d}.
\end{aligned}
\right.
\end{equation}
It is classical to show that $(\dot{\mathbb{E}}_{q}n_{0}, \dot{\mathbb{E}}_{q}u_{0},\dot{\mathbb{E}}_{q}\psi_{0})$ satisfies (\ref{a1}) uniformly with respect to $q\geq 1$ and converges to $(n_{0},u_{0},\psi_{0})$ strongly in the sense  (\ref{a1}). Since all the Sobolev norms are equivalent in (\ref{m1nn}) due to the Bernstein inequality, we can check that (\ref{m1nn}) is a system of ordinary differential equations in $L^2_{q}\times L^2_{q}\times L^2_{q}$ and locally Lipschitz with respect to the variable $(n^{q},u^{q},\psi^{q})$ for every $q\geq 1$. By virtue of the Cauchy-Lipschitz theorem in \cite[Page 124]{bahouri1}, there exists a maximal time $T^{*}_{q}>0$ such that the problem (\ref{m1nn}) admits a unique solution $(n^{q},u^{q},\psi^{q})\in C([0,T^{*}_{q});L^2_{q})$ on $[0,T_{q}^{*})$.

Let $\mathcal{X}(n,u,\psi):=\mathcal{X}(t)$ be given by $(\ref{Xdef})$. Define the maximal time
\begin{equation}\label{T*}
\begin{aligned}
&T_{q}:=\sup\big{\{}t\geq 0~| ~\text{there exists a unique solution $(n^{q},u^{q},\psi^{q})$} \\
 &\quad\quad\quad\quad\quad\quad\qquad\text{to the Cauchy problem (\ref{m1nn}) on $[0,t]$~satisfying}\\
  &\quad\quad\quad\quad\quad\quad\qquad~\mathcal{X}(n^{q},u^{q},\psi^{q})\leq  C_{0}\mathcal{X}_{0}   \big{\}}.
\end{aligned}
\end{equation}
By the time continuity of $(n^{q},u^{q},\psi^{q})$, we have $0<T_{q}\leq T_{q}^{*}$.

We claim $T_{q}=T_{q}^{*}$. To prove it, we assume that $T_{q}<T_{q}^{*}$ and use a contradiction argument. Since $(n^{q},u^{q},\psi^{q})=\dot{\mathbb{E}}_{n}(n^{q},u^{q},\psi^{q})$, the orthogonal projector $\dot{\mathbb{E}}_{n}$ has no effect on the energy estimates established in Lemmas \ref{lemma22}-\ref{lemma23} and \ref{lemma26}. By virtue of (\ref{XLvw}), (\ref{XL}), (\ref{XH}) and (\ref{T*}), as long as $\mathcal{X}_{0}$ satisfies (\ref{var0}), we have
\begin{equation}
\begin{aligned}
\mathcal{X}(n^{q},u^{q},\psi^{q})(t)\leq C\mathcal{X}_{0}+C\big{(} \mathcal{X}(n^{q},u^{q},\psi^{q})\big{)}^{2}(t),\quad 0< t<T_{q}.\label{X2}
\end{aligned}
\end{equation}
 By (\ref{X2}), we first choose $C_{0}=4c$ and then take 
$$
\mathcal{X}_{0}\leq \alpha:=\min\{\frac{1}{16 C^2},\frac{C_{*}}{2}\}
$$
 to derive 
\begin{equation}\label{X22}
\begin{aligned}
\mathcal{X}(n^{q},u^{q},\psi^{q})(t)\leq \frac{1}{2} C_{0} \mathcal{X}_{0},\quad 0< t<T_{q}.
\end{aligned}
\end{equation}
By (\ref{X22}) and the time continuity property, $T_{q}$ is not the maximal time such that $\mathcal{X}(n^{q},u^{q},\psi^{q})(t)\leq C_{0}\mathcal{X}_{0}$ holds. This contradicts the definition of $T_{q}$. 

If $T_{q}^{*}<\infty$, by (\ref{X22}) and $T_{q}=T_{q}^{*}$, we can take $(n^{q},u^{q},\psi^{q})(t)$ for $t$ sufficiently close to $T_{q}^{*}$ as the new initial data and obtain the existence from $t$ to some $t+\eta^{*}>T_{q}^{*}$ with a suitably small constant $\eta^{*}>0$ by the Cauchy-Lipschitz theorem, which contradicts the definition of $T_{q}^{*}$. Therefore, we have $T_{q}^{*}=\infty$, and $(n^{q},u^{q},\psi^{q})$ is indeed a global solution to the Cauchy problem (\ref{m1nn}).

From the uniform estimates $\mathcal{X}_{q}(t)\leq C_{0}\mathcal{X}_{0}$ and the equations (\ref{m1nn}), one can estimate the time derivatives $(\partial_{t}n^{q},\partial_{t}u^{q},\partial_{t}\psi^{q})$ in a suitable sense uniform with respect to $q$. According to these uniform estimates, the Aubin-Lions lemma and the Cantor diagonal process, there is a limit $(n,u,\psi)$ such that as $n\rightarrow\infty$, it holds, up to a subsequence (still denoted by $(n^{q},u^{q},\psi^{q})$), that
\begin{align}
&(\chi n^{q}, \chi u^{q}, \chi\psi^{q})\rightarrow (\chi n,\chi u,\chi\psi) \quad\text{strongly in}~L^{2}(0,T;\dot{B}^{\frac{d}{2}}_{2,1}),\quad \forall T>0,\quad \forall \chi \in\mathcal{D}(\mathbb{R}^{d}\times(0,T)).\nonumber
\end{align}
Thus, it is easy to prove that the limit $(n,u,\psi)$ solves System (\ref{m1n}) with the initial data $(n_{0},u_{0},\psi_{0})$ in the sense of distributions, and thanks to the uniform estimates $\mathcal{X}_{q}(t)\leq C_{0}\mathcal{X}_{0}$, the global solution $(n,u,\psi)$ to System (\ref{m1n}) is indeed a classical one and satisfies the properties (\ref{r1})-(\ref{XX0}). To finish the proof of Theorem \ref{theorem11}, we show that the solution constructed in this section is unique.

\subsection{Uniqueness}\label{subsection34}

  Suppose that (\ref{ABmu})-(\ref{a1}) of Theorem \ref{theorem11} are in force. Without loss of generation, we assume $\var=1$ in this subsection. For any given time $T>0$, let $(n_i,u_i,\psi_i)$ $(i=1,2)$ be two solutions to System (\ref{m1n}) satisfying (\ref{r1})-(\ref{XX0}) with the same initial data $(n_{0},u_{0},\psi_{0})$ on $[0,T]$. Denote the discrepancy 
 $$
(\widetilde{n},\widetilde{u},\widetilde{\psi}):=(n_{1}-n_{2},u_{1}-u_{2},\psi_{1}-\psi_{2}).
 $$
 Then it is easy to verify that $(\widetilde{n}_{j},\widetilde{u}_{j},\widetilde{\psi}_{j})$ satisfies the equations
\begin{equation}
\left\{
\begin{aligned}
&\partial_{t}\widetilde{n}_{j}+u_{1}\cdot \nabla \widetilde{n}_{j} +(c_{0}+G(n_{1}))\div \widetilde{u}_{j}=\widetilde{\mathcal{R}}_{1,j},\\
&\partial_{t}\widetilde{u}_{j}+u_{1}\cdot \nabla \widetilde{u}_{j}+\frac{1}{\varepsilon} \widetilde{u}_{j}+\nabla \widetilde{n}_{j}-\mu \nabla \widetilde{\psi}_{j}=\widetilde{ \mathcal{R}}_{2,j},\\
&\partial_{t}\widetilde{\psi}_{j} -\Delta \widetilde{\psi}_{j} +B\widetilde{\psi}_{j} -c_{1}\widetilde{n}_{j}=H(n_{1})_{j}-H(n_{2})_{j},
\end{aligned}
\right.
\end{equation}
where $w_{1,j}$ is the weight function $w_{1,j}=c_{0}+\dot{S}_{j-1}G(n_{1})$, and $\widetilde{\mathcal{R}}_{i,j}$ $(i=1,2)$ are the nonlinear terms 
\begin{equation}\nonumber
\left\{
\begin{aligned}
&\widetilde{ \mathcal{R}}_{1,j}:=\dot{S}_{j-1} u_{1}\cdot \nabla \widetilde{n}_{j}-\dot{\Delta}_{j} (u_{1}\cdot \nabla\widetilde{n})+\dot{S}_{j-1}G(n_{1}) \div \widetilde{u}_{j} - \dot{\Delta}_{j} (G(n_{1})\div \widetilde{u}_{j})\\
&\quad\quad~\quad-\dot{\Delta}_{j}(\widetilde{u}\cdot\nabla n_{2})-\dot{\Delta}_{j}\big((G(n_{1})-G(n_{2}))\div u_{2}\big),\\
&\widetilde{ \mathcal{R}}_{2,j}:=\dot{S}_{j-1} u_{1}\cdot \nabla \widetilde{u}_{j}-\dot{\Delta}_{j} (u_{1}\cdot \nabla \widetilde{u})-\dot{\Delta}_{j}(\widetilde{u}\cdot \nabla u_{2}).
\end{aligned}
\right.
\end{equation}
From (\ref{r1}), there holds
\begin{equation}\label{rdelta}
\begin{aligned}
&\|(n_{i},u_{i})\|_{L^{\infty}}+\|(\nabla n_{i},\nabla u_{i})\|_{L^{\infty}}+\|(n_{i},u_{i})\|_{\dot{B}^{\frac{d}{2}}_{2,1}\cap\dot{B}^{\frac{d}{2}+1}_{2,1}}\lesssim 1,\quad\quad t\in[0,T],\quad i=1,2.
\end{aligned}
\end{equation}
By similar arguments as used in (\ref{m1j})-(\ref{apsij}), one can obtain
\begin{equation}\label{deltaau}
\begin{aligned}
&\frac{d}{dt}\int[ \frac{1}{2}|\widetilde{n}_{j}|^2+\frac{1}{2} w_{1,j}|\widetilde{u}_{j}|^2+\frac{\mu b}{2c_{1}} |\widetilde{\psi}_{j}|^2+\frac{\mu}{2c_{1}} |\nabla \widetilde{\psi}_{j}|^2-\mu \widetilde{n}_{j} \widetilde{\psi}_{j}]dx\\
&\quad+\int  (\frac{1}{\varepsilon} w_{1,j}|\widetilde{u}_{j}|^2+|\partial_{t}\widetilde{\psi}_{j}|^2)dx\\
&=\int \Big{(} \frac{1}{2} \dot{S}_{j-1}\div u_{1} |\widetilde{n}_{j}|^2+\widetilde{\mathcal{R}}_{1,j} \widetilde{n}_{j}+\frac{1}{2}w_{1,j}\dot{S}_{j-1}\div u_{1} |\widetilde{u}_{j}|^2+w_{1,j} \widetilde{ \mathcal{R}}_{2,j} \cdot \widetilde{u}_{j}\\
&\quad+\frac{1}{2}\partial_{t}w_{1,j} |\widetilde{u}_{j}|^2+\nabla w_{1,j}\cdot ( \frac{1}{2} u_{1} |\widetilde{u}_{j}|^2+\widetilde{n}_{j} \widetilde{u}_{j}-\mu\widetilde{\psi}_{j}\widetilde{u}_{j})\\
&\quad+\partial_{t}\widetilde{\psi}_{j}(H(n_{1})_{j}-H(n_{2})_{j}) -c_{1}(-\dot{S}_{j-1} u_{1} \cdot \nabla \widetilde{n}_{j}+\widetilde{\mathcal{R}}_{1,j})\widetilde{\psi}_{j} \Big{)}dx, \quad j\in\mathbb{Z}.
\end{aligned}
\end{equation}
Here we omit the details for the sake of simplicity. Owing to (\ref{wjc0})-(\ref{var0}), we have
\begin{equation}\label{deltaau2}
\begin{aligned}
&\frac{c_{0}}{2}\leq w_{1,j}(x,t)\leq \frac{3c_{0}}{2},\quad (x,t)\in\mathbb{R}^{d}\times(0,T),\quad j\in\mathbb{Z}.
\end{aligned}
\end{equation}
By (\ref{ABmu}), (\ref{XX0}), (\ref{deltaau})-(\ref{deltaau2}) and the fact
\begin{equation}\nonumber
\begin{aligned}
&\int\partial_{t}\widetilde{\psi}_{j}(H(n_{1})_{j}-H(n_{2})_{j})dx\leq \frac{1}{4}\|\partial_{t}\widetilde{\psi}_{j}\|_{L^2}^2+\|H(n_{1})_{j}-H(n_{2})_{j}\|_{L^2}^2,
\end{aligned}
\end{equation}
we deduce after direct computations that
\begin{equation}\label{delta11}
\begin{aligned}
&\|(\widetilde{n},\widetilde{u},\widetilde{\psi},\nabla\widetilde{\psi})\|_{\dot{B}^{\frac{d}{2}}_{2,1}}^2\\
&\quad\leq C\int_{0}^{t}\big{(} \|(\widetilde{n},\widetilde{u},\widetilde{\psi},\nabla\widetilde{\psi})\|_{\dot{B}^{\frac{d}{2}}_{2,1}}^2 +\|H(n_{1})-H(n_{2})\|_{\dot{B}^{\frac{d}{2}}_{2,1}}^2\\
&\quad\quad+\sum_{j\in\mathbb{Z}} 2^{\frac{d}{2}j} \|(\widetilde{ \mathcal{R}}_{1,j}\widetilde{ \mathcal{R}}_{2,j})\|_{L^2} \|(\widetilde{n},\widetilde{u},\widetilde{\psi},\nabla\widetilde{\psi})\|_{\dot{B}^{\frac{d}{2}}_{2,1}} \big) d\tau.
\end{aligned}
\end{equation}
One gets from (\ref{rdelta}), (\ref{uv2}), (\ref{commutator}) and (\ref{F3}) that
\begin{equation}\label{delta13}
\begin{aligned}
&\sum_{j\in\mathbb{Z}} 2^{j\frac{d}{2}}\|\widetilde{\mathcal{R}}_{1,j}\|_{L^2}\\
&\quad\lesssim \|u_{1}\|_{\dot{B}^{\frac{d}{2}+1}_{2,1}}\|\widetilde{n}\|_{\dot{B}^{\frac{d}{2}}_{2,1}}+\|G(n_{1})\|_{\dot{B}^{\frac{d}{2}+1}_{2,1}}\|\widetilde{u}\|_{\dot{B}^{\frac{d}{2}}_{2,1}}\\
&\quad\quad+\|n_{2}\|_{\dot{B}^{\frac{d}{2}+1}_{2,1}}\|\widetilde{u}\|_{\dot{B}^{\frac{d}{2}}_{2,1}}+\|u_{2}\|_{\dot{B}^{\frac{d}{2}+1}_{2,1}} \|G(n_{1})-G(n_{2})\|_{\dot{B}^{\frac{d}{2}}_{2,1}}\\
&\quad\lesssim \|(\widetilde{n},\widetilde{u})\|_{\dot{B}^{\frac{d}{2}}_{2,1}}.
\end{aligned}
\end{equation}
Similarly, it holds that
\begin{equation}\label{delta14}
\begin{aligned}
&\sum_{j\in\mathbb{Z}} 2^{j\frac{d}{2}}\|\widetilde{ \mathcal{R}}_{2,j}\|_{L^2}\lesssim \|(u_{1},u_{2})\|_{\dot{B}^{\frac{d}{2}+1}_{2,1}} \|\widetilde{u}\|_{\dot{B}^{\frac{d}{2}}_{2,1}}\lesssim \|\widetilde{u}\|_{\dot{B}^{\frac{d}{2}}_{2,1}}.
\end{aligned}
\end{equation}
Finally, using the composition estimate (\ref{F3}), we derive
\begin{equation}\label{delta15}
\begin{aligned}
&\|H(n_{1})-H(n_{2})\|_{\dot{B}^{\frac{d}{2}}_{2,1}}^2\lesssim\|\widetilde n\|_{\dot{B}^{\frac{d}{2}}_{2,1}}^2.
\end{aligned}
\end{equation}
Combining (\ref{delta11})-(\ref{delta15}) together and employing the Gr${\rm{\ddot{o}}}$nwall inequality, we prove $(\widetilde{n},\widetilde{u},\widetilde{\psi})=0$ a.e. in $\mathbb{R}^{d}\times(0,T)$.
 
 \section{Optimal time-decay rates}\label{section5}

\subsection{The evolution of negative Besov norms}
In this section, we establish the time-decay estimates of the global solution given by Theorem \ref{theorem11} without additional smallness condition on the initial data. The following lemma shows that the additional low-frequency regularity assumption \eqref{a3} is propagated in time.
\begin{lemma}\label{lemma51}
Let $\sigma_{0}\in[-\frac{d}{2},\frac{d}{2})$. Then, under the assumptions {\rm(\ref{ABmu})}, {\rm(\ref{a1})} and {\rm(\ref{a3})}, the following inequality holds:
\begin{equation}
\begin{aligned}
&\mathcal{X}_{L}^{\sigma_{0}}(t)\leq C\big{(} \mathcal{X}_{0}+\|(n_{0},  u_{0}, \psi_{0})\|_{\dot{B}^{\sigma_{0}}_{2,\infty}}^{\ell} \big{)},\quad\quad t>0,\label{lowd}
\end{aligned}
\end{equation}
where $\mathcal{X}_{L}^{\sigma_{0}}(t)$ is defined by
\begin{equation}\label{XLsigma}
\begin{aligned}
&\mathcal{X}_{L}^{\sigma_{0}}(t):=\|(n,u,\psi)\|_{\widetilde{L}^{\infty}_{t}(\dot{B}^{\sigma_{0}}_{2,\infty})}^{\ell}+\var\|(n,\psi)\|_{\widetilde{L}^1_{t}(\dot{B}^{\sigma_{0}+2}_{2,\infty})}^{\ell}+\|u\|_{\widetilde{L}^1_{t}(\dot{B}^{\sigma_{0}+1}_{2,\infty})}^{\ell}+\var^{-\frac{1}{2}}\|u\|_{\widetilde{L}^2_{t}(\dot{B}^{\sigma_{0}}_{2,\infty})}^{\ell}.
\end{aligned}
\end{equation}
\end{lemma}
\begin{proof}
Let the effective unknowns $(\varphi,v)$ be denoted by (\ref{vw}). Owing to $(\ref{Ga})_{1}$ and (\ref{diss1}), we have
\begin{equation}\label{sigma1}
\begin{aligned}
&\|n\|_{\widetilde{L}^{\infty}_{t}(\dot{B}^{\sigma_{0}}_{2,\infty})}^{\ell}+\var\|n\|_{\widetilde{L}^{1}_{t}(\dot{B}^{\sigma_{0}+2}_{2,\infty})}^{\ell}\\
&\quad\lesssim \|n_{0}\|_{\dot{B}^{\sigma_{0}}_{2,\infty}}^{\ell}+\var\|\varphi\|_{\widetilde{L}^{1}_{t}(\dot{B}^{\sigma_{0}+2}_{2,\infty})}^{\ell}+\|v\|_{\widetilde{L}^{1}_{t}(\dot{B}^{\sigma_{0}+1}_{2,\infty})}^{\ell}+\|R_{1}\|_{\widetilde{L}^1_{t}(\dot{B}^{\sigma_{0}}_{2,\infty})}^{\ell}\\
&\quad\lesssim \|n_{0}\|_{\dot{B}^{\sigma_{0}}_{2,\infty}}^{\ell}+\var\|\varphi\|_{\widetilde{L}^{1}_{t}(\dot{B}^{\sigma_{0}+1}_{2,\infty}\cap\dot{B}^{\sigma_{0}+3}_{2,\infty})}^{\ell}+\|v\|_{\widetilde{L}^{1}_{t}(\dot{B}^{\sigma_{0}+1}_{2,\infty})}^{\ell}+\|R_{1}\|_{\widetilde{L}^1_{t}(\dot{B}^{\sigma_{0}}_{2,\infty})}^{\ell},
\end{aligned}
\end{equation}
which together with $(\ref{Ga})_{2}$, (\ref{diss1}) and (\ref{bessel}) gives
\begin{equation}\label{sigma13}
\begin{aligned}
&\|\varphi\|_{\widetilde{L}^{\infty}_{t}(\dot{B}^{\sigma_{0}}_{2,\infty})}^{\ell}+\|\varphi\|_{\widetilde{L}^{1}_{t}(\dot{B}^{\sigma_{0}}_{2,\infty}\cap\dot{B}^{\sigma_{0}+2}_{2,\infty})}^{\ell}\\
&\quad\lesssim \|\varphi|_{t=0}\|_{\dot{B}^{\sigma_{0}}_{2,\infty}}^{\ell}+\var\|n\|_{\widetilde{L}^1_{t}(\dot{B}^{\sigma_{0}+2}_{2,\infty})}^{\ell}+\|v\|_{\widetilde{L}^1_{t}(\dot{B}^{\sigma_{0}+1}_{2,\infty})}^{\ell}+\|R_{2}\|_{\widetilde{L}^1_{t}(\dot{B}^{\sigma_{0}}_{2,\infty})}^{\ell}\\
&\quad\lesssim \|(n,\varphi)|_{t=0}\|_{\dot{B}^{\sigma_{0}}_{2,\infty}}^{\ell}+\var\|\varphi\|_{\widetilde{L}^{1}_{t}(\dot{B}^{\sigma_{0}+1}_{2,\infty}\cap\dot{B}^{\sigma_{0}+3}_{2,\infty})}^{\ell}\\
&\quad\quad+\|v\|_{\widetilde{L}^{1}_{t}(\dot{B}^{\sigma_{0}+1}_{2,\infty})}^{\ell}+\|(R_{1},R_{2})\|_{\widetilde{L}^1_{t}(\dot{B}^{\sigma_{0}}_{2,\infty})}^{\ell}.
\end{aligned}
\end{equation}
 Making use of $(\ref{Ga})_{3}$ and $(\ref{sigma13})$, we also get
\begin{equation}\label{sigma14}
\begin{aligned}
&\|v\|_{\widetilde{L}^{\infty}_{t}(\dot{B}^{\sigma_{0}}_{2,\infty})}^{\ell}+\frac{1}{\var}\|v\|_{\widetilde{L}^{1}_{t}(\dot{B}^{\sigma_{0}}_{2,\infty})}^{\ell}\\
&\quad\lesssim \|v|_{t=0}\|_{\dot{B}^{\sigma_{0}}_{2,\infty}}^{\ell}+\var^2\|n\|_{\widetilde{L}^1_{t}(\dot{B}^{\sigma_{0}+3}_{2,\infty})}^{\ell}+\var\|\varphi\|_{\widetilde{L}^1_{t}(\dot{B}^{\sigma_{0}+1}_{2,\infty}\cap\dot{B}^{\sigma_{0}+3}_{2,\infty})}^{\ell}\\
&\quad\quad+\var\|v\|_{\widetilde{L}^1_{t}(\dot{B}^{\sigma_{0}+2}_{2,\infty})}^{\ell}+\|R_{3}\|_{\widetilde{L}^1_{t}(\dot{B}^{\sigma_{0}}_{2,\infty})}^{\ell}.
\end{aligned}
\end{equation}
The combination of (\ref{lh}) and (\ref{sigma1})-(\ref{sigma14}) implies that for $J_{\var}$ given by (\ref{J0}) with a sufficiently small constant $k$, we have
\begin{equation}\label{sigma16}
\begin{aligned}
&\|(n,\varphi,v)\|_{\widetilde{L}^{\infty}_{t}(\dot{B}^{\sigma_{0}}_{2,\infty})}^{\ell}+\var\|n\|_{\widetilde{L}^{1}_{t}(\dot{B}^{\sigma_{0}+2}_{2,\infty})}^{\ell}+\|\varphi\|_{\widetilde{L}^{1}_{t}(\dot{B}^{\sigma_{0}}_{2,\infty}\cap\dot{B}^{\sigma_{0}+2}_{2,\infty})}^{\ell}+\frac{1}{\var}\|v\|_{\widetilde{L}^{1}_{t}(\dot{B}^{\sigma_{0}}_{2,\infty})}^{\ell}\\
&\quad\lesssim \|(n,\varphi,v)|_{t=0}\|_{\dot{B}^{\sigma_{0}}_{2,\infty}}^{\ell}+\|(R_{1},R_{2},R_{3})\|_{\widetilde{L}^1_{t}(\dot{B}^{\sigma_{0}}_{2,\infty})}^{\ell}.
\end{aligned}
\end{equation}
We turn to estimate the nonlinear terms $R_{i}$ $(i=1,2,3)$. One infers from (\ref{ninfty}), (\ref{hl}) and the composition estimate (\ref{q5}) that
\begin{equation}\label{sigma161}
\begin{aligned}
&\var\|H(n)\|_{\widetilde{L}^1_{t}(\dot{B}^{\sigma_{0}+2}_{2,\infty})}^{\ell}\\
&\quad\lesssim (\|n\|_{\widetilde{L}^{\infty}_{t}(\dot{B}^{\frac{d}{2}}_{2,1})}^{\ell}+\var\|n\|_{\widetilde{L}^{\infty}_{t}(\dot{B}^{\frac{d}{2}+1}_{2,1})}^{h} ) (\var\|n\|_{\widetilde{L}^1_{t}(\dot{B}^{\sigma_{0}+2}_{2,\infty})}^{\ell}+\|n\|_{L^1_{t}(\dot{B}^{\frac{d}{2}+1}_{2,1})}^{h})\\
&\quad\lesssim  \mathcal{X}(t)\mathcal{X}^{\sigma_{0}}_{L}(t) +\mathcal{X}^2(t).
\end{aligned}
\end{equation}
And due to (\ref{ninfty}), (\ref{hl}) and (\ref{F1}), we get
\begin{equation}\label{sigma162}
\begin{aligned}
&\|u\cdot \nabla u\|_{\widetilde{L}^1_{t}(\dot{B}^{\sigma_{0}}_{2,\infty})}^{\ell}+\|u\cdot \nabla n\|_{\widetilde{L}^1_{t}(\dot{B}^{\sigma_{0}}_{2,\infty})}^{\ell}+\|u\cdot \nabla G(n)\|_{\widetilde{L}^1_{t}(\dot{B}^{\sigma_{0}}_{2,\infty})}^{\ell}+\|\partial_{t}H(n)\|_{\widetilde{L}^1_{t}(\dot{B}^{\sigma_{0}}_{2,\infty})}^{\ell}\\
&\quad\lesssim \|u\|_{\widetilde{L}^2_{t}(\dot{B}^{\frac{d}{2}}_{2,1})}\|(n,u)\|_{\widetilde{L}^2_{t}(\dot{B}^{\sigma_{0}+1}_{2,\infty})}+\|n\|_{\widetilde{L}^{\infty}_{t}(\dot{B}^{\frac{d}{2}}_{2,1})}\|u\|_{\widetilde{L}^1_{t}(\dot{B}^{\sigma_{0}+1}_{2,\infty})}\\
&\quad\lesssim \var^{-\frac{1}{2}}\|u\|_{\widetilde{L}^2_{t}(\dot{B}^{\frac{d}{2}}_{2,1})}\Big( \big( \|(n,u)\|_{\widetilde{L}^{\infty}_{t}(\dot{B}^{\sigma_{0}}_{2,\infty})}^{\ell}\big)^{\frac{1}{2}}\big( \var\|(n,u)\|_{\widetilde{L}^{1}_{t}(\dot{B}^{\sigma_{0}+2}_{2,\infty})}^{\ell} \big)^{\frac{1}{2}}\\
&\quad\quad+\big( \var\|(n,u)\|_{\widetilde{L}^{\infty}_{t}(\dot{B}^{\frac{d}{2}+1}_{2,1})}^{h} \big)^{\frac{1}{2}} \big( \|(n,u)\|_{L^{1}_{t}(\dot{B}^{\frac{d}{2}+1}_{2,1})}^{h} \big)^{\frac{1}{2}} \Big)\\
&\quad\quad+ \big(\|n\|_{\widetilde{L}^{\infty}_{t}(\dot{B}^{\frac{d}{2}}_{2,1})}^{\ell}+\var\|n\|_{\widetilde{L}^{\infty}_{t}(\dot{B}^{\frac{d}{2}+1}_{2,1})}^{h} \big) \big( \|u\|_{\widetilde{L}^1_{t}(\dot{B}^{\sigma_{0}+1}_{2,\infty})}^{\ell}+\|u\|_{L^1_{t}(\dot{B}^{\frac{d}{2}+1}_{2,1})}^{h}\big)\\
&\quad\lesssim \mathcal{X}(t)\mathcal{X}^{\sigma_{0}}_{L}(t) +\mathcal{X}^2(t).
\end{aligned}
\end{equation}
One deduces by (\ref{R123}) and (\ref{sigma161})-(\ref{sigma162}) that
\begin{equation}\label{sigma15}
\begin{aligned}
&\|(R_{1},R_{2},R_{3})\|_{\widetilde{L}^1_{t}(\dot{B}^{\sigma_{0}}_{2,\infty})}^{\ell}\\
&\quad\lesssim \var\|H(n)\|_{\widetilde{L}^1_{t}(\dot{B}^{\sigma_{0}+2}_{2,\infty})}^{\ell}+\|u\cdot \nabla n\|_{\widetilde{L}^1_{t}(\dot{B}^{\sigma_{0}}_{2,\infty})}^{\ell}+\|u\cdot \nabla u\|_{\widetilde{L}^1_{t}(\dot{B}^{\sigma_{0}}_{2,\infty})}^{\ell}\\
&\quad\quad+\|G(n)\div u\|_{\widetilde{L}^1_{t}(\dot{B}^{\sigma_{0}}_{2,\infty})}^{\ell}+\|\partial_{t}H(n)\|_{\widetilde{L}^1_{t}(\dot{B}^{\sigma_{0}}_{2,\infty})}^{\ell}\\
&\quad\lesssim \mathcal{X}(t)\mathcal{X}^{\sigma_{0}}_{L}(t) +\mathcal{X}^2(t).
\end{aligned}
\end{equation}
Combining (\ref{sigma16}) and (\ref{sigma15}) together, we gain
\begin{equation}\nonumber
\begin{aligned}
&\|(n,\varphi,v)\|_{\widetilde{L}^{\infty}_{t}(\dot{B}^{\sigma_{0}}_{2,\infty})}^{\ell}+\var\|n\|_{\widetilde{L}^{1}_{t}(\dot{B}^{\sigma_{0}+2}_{2,\infty})}^{\ell}+\|\varphi\|_{\widetilde{L}^{1}_{t}(\dot{B}^{\sigma_{0}}_{2,\infty}\cap\dot{B}^{\sigma_{0}+2}_{2,\infty})}^{\ell}+\frac{1}{\var}\|v\|_{\widetilde{L}^{1}_{t}(\dot{B}^{\sigma_{0}}_{2,\infty})}^{\ell}\\
&\quad\lesssim\mathcal{X}_{0}+\|(n_{0},  u_{0}, \psi_{0})\|_{\dot{B}^{\sigma_{0}}_{2,\infty}}^{\ell} +\mathcal{X}(t)\mathcal{X}^{\sigma_{0}}_{L}(t) +\mathcal{X}^2(t).
\end{aligned}
\end{equation}
Therefore, by similar computations as in (\ref{231})-(\ref{235}), one has
\begin{equation}\nonumber
\begin{aligned}
&\mathcal{X}^{\sigma_0}_{L}(t)\lesssim\mathcal{X}_{0}+\|(n_{0},  u_{0}, \psi_{0})\|_{\dot{B}^{\sigma_{0}}_{2,\infty}}^{\ell} +\mathcal{X}(t)\mathcal{X}^{\sigma_{0}}_{L}(t) +\mathcal{X}^2(t).
\end{aligned}
\end{equation}
This together with $\mathcal{X}(t)\lesssim \mathcal{X}_{0}<<1$ leads to (\ref{lowd}). 
\end{proof}

\subsection{New time-weighted estimates}

We apply a modified energy argument to prove the optimal time-decay rates (\ref{decay1})-(\ref{decay3}) of global solutions to the Cauchy problem of System (\ref{m1n}). For a suitably large constant $\theta$, we introduce the time-weighted functional
\begin{equation}\label{decayX}
\begin{aligned}
\mathcal{D}^{\theta}(t):=\mathcal{D}_{L}^{\theta}(t)+\mathcal{D}_{H}^{\theta}(t),
\end{aligned}
\end{equation}
with
\begin{align}
&\mathcal{D}_{L}^{\theta}(t):=\|\tau^{\theta}(n,u,\psi)\|_{\widetilde{L}^{\infty}_{t}(\dot{B}^{\frac{d}{2}}_{2,1})}^{\ell}+\var\|\tau^{\theta}(n,\psi)\|_{L^1_{t}(\dot{B}^{\frac{d}{2}+2}_{2,1})}^{\ell}+\|\tau^{\theta}u\|_{L^1_{t}(\dot{B}^{\frac{d}{2}+1}_{2,1})}^{\ell},\nonumber \\
&\mathcal{D}_{H}^{\theta}(t):=\var\|\tau^{\theta}(n,u,\nabla\psi)\|_{\widetilde{L}^{\infty}_{t}(\dot{B}^{\frac{d}{2}+1}_{2,1})}^{h}+ \|\tau^{\theta}(n,u,\nabla\psi)\|_{L^{1}_{t}(\dot{B}^{\frac{d}{2}+1}_{2,1})}^{h}.\nonumber
\end{align}
Compared with $\mathcal{X}(t)$, there are some time-weighted dissipation estimates missing in $\mathcal{D}^{\theta}(t)$. Here we omit these since $\mathcal{D}^{\theta}(t)$ is sufficient to derive the desired time-decay rates.

 In low frequencies, we have the following estimates.
\begin{lemma}\label{lemma52}
Let $\sigma_{0}\in[-\frac{d}{2},\frac{d}{2})$ and $\theta>1+\frac{1}{2}(\frac{d}{2}-\sigma_{0})$. Then, under the assumptions of {\rm(\ref{ABmu})}, {\rm(\ref{a1})} and {\rm(\ref{a3})}, it holds for any constant $\zeta>0$ that
\begin{equation}\label{decayL}
\begin{aligned}
&\mathcal{D}_{L}^{\theta}(t)\leq \frac{C\big{(} \mathcal{X}(t)+\mathcal{X}_{L}^{\sigma_{0}}(t)\big{)}}{\zeta} t^{\theta}(\var t)^{-\frac{1}{2}(\frac{d}{2}-\sigma_{0})}+C\big{(}\zeta+\mathcal{X}(t)\big{)}\mathcal{D}^{\theta}(t),\quad \quad t>0.
\end{aligned}
\end{equation}
\end{lemma}
\begin{proof}

Let $(\varphi,v)$ be the effective unknowns defined in (\ref{vw}). Multiplying the equations (\ref{Ga}) by $t^{\theta}$, we have
\begin{equation}\nonumber
\left\{
\begin{aligned}
&\partial_{t}(t^{\theta}n)-\tilde{\Delta}_1(t^{\theta}n)=\theta t^{\theta-1}n+ t^{\theta} L_{1}+t^{\theta}R_{1},\\
&\partial_{t}( t^{\theta}\varphi)-\tilde{\Delta}_2(t^{\theta}\varphi)=\theta t^{\theta-1}\varphi+ t^{\theta} L_{2}+t^{\theta} R_{2},\\
&\partial_{t} (t^{\theta}v)+\var^{-1} ( t^{\theta}v)=\theta t^{\theta-1}v+t^{\theta} L_{3}+t^{\theta} R_{3},
\end{aligned}
\right.
\end{equation}
with $\tilde{\Delta}_{i}$ $(i=1,2)$, $L_{i}$ $(i=1,2,3)$ and $R_{i}$ $(i=1,2,3)$ given by (\ref{delta1delta2}), (\ref{L123}) and (\ref{R123}), respectively.
By similar calculations as used in (\ref{221})-(\ref{R3}), we can derive
\begin{equation}\label{321}
\begin{aligned}
&\|\tau^{\theta}(n,v,\psi)\|_{\widetilde{L}^{\infty}_{t}(\dot{B}^{\frac{d}{2}}_{2,1})}^{\ell}+\var\|\tau^{\theta}n\|_{L^{1}_{t}(\dot{B}^{\frac{d}{2}+2}_{2,1})}^{\ell}+\|\tau^{\theta}\varphi\|_{L^{1}_{t}(\dot{B}^{\frac{d}{2}}_{2,1}\cap\dot{B}^{\frac{d}{2}+2}_{2,1})}^{\ell}+\frac{1}{\var}\|\tau^{\theta}v\|_{L^{1}_{t}(\dot{B}^{\frac{d}{2}}_{2,1})}^{\ell}\\
&\quad\lesssim \int_{0}^{t}\tau^{\theta-1}\|(n,v,\varphi)\|_{\dot{B}^{\frac{d}{2}}_{2,1}}^{\ell}d\tau+\mathcal{X}(t)\mathcal{D}^{\theta}(t).
\end{aligned}
\end{equation}
It can be proved as in (\ref{231})-(\ref{235}) that the estimates (\ref{321}) implies
\begin{equation}\label{322}
\begin{aligned}
&\mathcal{D}_{L}^{\theta}(t)\lesssim \int_{0}^{t}\tau^{\theta-1}\big(\|(n,u,\psi)\|_{\dot{B}^{\frac{d}{2}}_{2,1}}^{\ell}+\|H(n)\|_{\dot{B}^{\frac{d}{2}}_{2,1}}^{\ell}\big)d\tau+\mathcal{X}(t)\mathcal{D}^{\theta}(t).
\end{aligned}
\end{equation}
We omit the details of the above estimates and focus on controlling the first term on the right-hand side of (\ref{322}) which is the key point. It follows from the interpolation inequality (\ref{inter}) and the Young inequality that
\begin{equation}
\begin{aligned}
&\int_{0}^{t}\tau^{\theta-1} \|(n^{\ell},u^{\ell},\psi^{\ell})\|_{\dot{B}^{\frac{d}{2}}_{2,1}}d\tau \\
&\quad\lesssim \int_{0}^{t}\tau^{\theta-1}\Big{(} \|(n^{\ell},u^{\ell},\psi^{\ell})\|_{\dot{B}^{\sigma_{0}}_{2,\infty}}\Big{)}^{\frac{2}{\frac{d}{2}+2-\sigma_{0}}}\Big{(} \|(n^{\ell},u^{\ell},\psi^{\ell})\|_{\dot{B}^{\frac{d}{2}+2}_{2,1}}\Big{)}^{\frac{\frac{d}{2}-\sigma_{0}}{\frac{d}{2}+2-\sigma_{0}}}d\tau \\
&\quad\lesssim \Big{(}\mathcal{X}_{L}^{\sigma_{0}}(t)\int_{0}^{t}\tau^{\theta-\frac{1}{2}(\frac{d}{2}-\sigma_{0})-1}d\tau\Big{)}^{\frac{2}{\frac{d}{2}+2-\sigma_{0}}}\Big{(} \frac{1}{\var} \big( \var \|\tau^{\theta}(n,\psi)\|_{L^1_{t}(\dot{B}^{\frac{d}{2}+2}_{2,1})}^{\ell}+\|\tau^{\theta}u\|_{L^1_{t}(\dot{B}^{\frac{d}{2}+1}_{2,1})}^{\ell} \big) \Big{)}^{\frac{\frac{d}{2}-\sigma_{0}}{\frac{d}{2}+2-\sigma_{0}}} \\
&\quad\lesssim \frac{\mathcal{X}_{L}^{\sigma_{0}}(t)}{\zeta} t^{\theta}(\var t)^{-\frac{1}{2}(\frac{d}{2}-\sigma_{0})}+\zeta \mathcal{D}_{L}^{\theta}(t).\label{323}
\end{aligned}
\end{equation}
It also holds by (\ref{hl}) that
\begin{equation}\label{324}
\begin{aligned}
&\int_{0}^{t}\tau^{\theta-1} \|(n^{h},u^{h},\psi^{h})\|_{\dot{B}^{\frac{d}{2}}_{2,1}} d\tau\\
&\quad\lesssim \int_{0}^{t}\tau^{\theta-1}\Big{(} \var\|(n,u,\psi)\|_{\dot{B}^{\frac{d}{2}+1}_{2,1}}^{h} \Big{)}^{\frac{2}{\frac{d}{2}+2-\sigma_{0}}}\Big{(} \|(n,u,\psi)\|_{\dot{B}^{\frac{d}{2}+1}_{2,1}}^{h}\Big{)}^{\frac{\frac{d}{2}-\sigma_{0}}{\frac{d}{2}+2-\sigma_{0}}}d\tau\\
&\quad \lesssim\Big{(} \mathcal{X} (t) \int_{0}^{t}\tau^{\theta-\frac{1}{2}(\frac{d}{2}-\sigma_{0})-1}d\tau\Big{)}^{\frac{2}{\frac{d}{2}+2-\sigma_{0}}} \Big{(} \|\tau^{\theta}(n,u,\psi)\|_{L^1_{t}(\dot{B}^{\frac{d}{2}+1}_{2,1})}^{h}\Big{)}^{\frac{\frac{d}{2}-\sigma_{0}}{\frac{d}{2}+2-\sigma_{0}}}\\
&\quad\lesssim \frac{\mathcal{X}(t)}{\zeta} t^{\theta-\frac{1}{2}(\frac{d}{2}-\sigma_{0})}+\zeta\mathcal{D}_{H}^{\theta}(t).
\end{aligned}
\end{equation}
In addition, we obtain from (\ref{323})-(\ref{324}) and the composition estimate (\ref{F1}) that
\begin{equation}\label{3241}
\begin{aligned}
\int_{0}^{t}\tau^{\theta-1} \|H(n)\|_{\dot{B}^{\frac{d}{2}}_{2,1}}^{\ell}d\tau&\lesssim \int_{0}^{t}\tau^{\theta-1} \big( \|n\|_{\dot{B}^{\frac{d}{2}}_{2,1}}^{\ell}+\|n\|_{\dot{B}^{\frac{d}{2}}_{2,1}}^{h}\big)dx\\
&\lesssim \frac{\mathcal{X}(t)}{\zeta} t^{\theta-\frac{1}{2}(\frac{d}{2}-\sigma_{0})}+\zeta\mathcal{D}_{H}^{\theta}(t).
\end{aligned}
\end{equation}
One combines (\ref{323})-(\ref{3241}) together to obtain
\begin{equation}\label{325}
\begin{aligned}
&\int_{0}^{t}\tau^{\theta-1}\|(n,u,\psi)\|_{\dot{B}^{\frac{d}{2}}_{2,1}}^{\ell} d\tau\\
&\quad\lesssim \int_{0}^{t}\tau^{\theta-1}\big{(} \|(n^{\ell},u^{\ell},\psi^{\ell})\|_{\dot{B}^{\frac{d}{2}}_{2,1}}+\|(n^{h},u^{h},\psi^{h})\|_{\dot{B}^{\frac{d}{2}}_{2,1}}\big{)}d\tau\\
&\quad\lesssim \frac{\mathcal{X}_{L}^{\sigma_{0}}(t)+\mathcal{X}(t)}{\zeta} t^{\theta}(\var t)^{-\frac{1}{2}(\frac{d}{2}-\sigma_{0})}+\zeta\mathcal{D}^{\theta}(t).
\end{aligned}
\end{equation}
Inserting (\ref{325}) into (\ref{322}), we prove (\ref{decayL}).
\end{proof}

Then, we have the time-weighted estimates in high frequencies.
\begin{lemma}\label{lemma53}
Let $\sigma_{0}\in[-\frac{d}{2},\frac{d}{2})$ and $\theta>1+\frac{1}{2}(\frac{d}{2}-\sigma_{0})$. Then, under the assumptions of {\rm(\ref{ABmu})}, {\rm(\ref{a1})} and {\rm(\ref{a3})}, it holds for any constant $\zeta>0$ that
\begin{equation}\label{decayH}
\begin{aligned}
&\mathcal{D}_{H}^{\theta}(t)\leq \frac{C\mathcal{X}(t)}{\zeta } t^{\theta-\frac{1}{2}(\frac{d}{2}-\sigma_{0})}+C\big{(}\zeta+ \mathcal{X}(t) \big{)}\mathcal{D}^{\theta}(t),\quad\quad t>0.
\end{aligned}
\end{equation}
\end{lemma}
\begin{proof}
Multiplying the inequality (\ref{264}) by $t^{\theta}$, we obtain for any $j\geq J_{\var}-1$ that
\begin{equation}\label{331}
\begin{aligned}
&\frac{d}{dt}\big{(} t^{\theta}\mathcal{L}_{j}(t) \big{)}+\frac{1}{\var}t^{\theta}\mathcal{L}_{j}(t)\\
&\lesssim t^{\theta-1} \mathcal{L}_{j}(t)+t^{\theta}\Big{(} \|\nabla u\|_{L^{\infty}}\var\|(n_{j},u_{j})\|_{L^2}+\|u\|_{L^{\infty}}\var\|(n_{j},H(n)_{j})\|_{L^2} \\
&\quad+\|(\partial_{t} G(n),\nabla G(n))\|_{L^{\infty}}\var\|u_{j}\|_{L^2}+\|(\mathcal{R}_{1,j},\mathcal{R}_{2,j})\|_{L^2}\\
&\quad+\|H(n)_{j}\|_{L^2}+2^{-j}\|\partial_{t}H(n)_{j}\|_{L^2}\Big{)}\sqrt{\frac{1}{\var}\mathcal{L}_{j}(t)},
\end{aligned}
\end{equation}
which together with $(\ref{267})_{1}$ and $t^{\theta} \sqrt{\mathcal{L}_{j}(t)}|_{t=0}=0$ gives rise to
\begin{equation}\nonumber
\begin{aligned}
&\var t^{\theta} \|(n_{j},u_{j},\psi_{j},\nabla\psi_{j})\|_{L^2}+ \int_{0}^{t} \tau^{\theta}\|(n_{j},u_{j},\psi_{j},\nabla\psi_{j})\|_{L^2}d\tau\\
&\quad\lesssim \int_{0}^{t} \tau^{\theta-1}\big{(} \|(n_{j},u_{j},\psi_{j},\nabla\psi_{j})\|_{L^2}+2^{-j}\|H(n)_{j}\|_{L^2}\big{)}d\tau\\
&\quad\quad+\int_{0}^{t} \tau^{\theta}\Big{(} \var\|\nabla u\|_{L^{\infty}}\big{(}\|(n_{j},u_{j})\|_{L^2}+2^{-j}\|H(n)_{j}\|_{L^2}\big{)}+ \var\| u\|_{L^{\infty}}\|(n_{j},H(n)_{j})\|_{L^2}\\
&\quad\quad+  \var\|\big(\nabla G(n) \big)\|_{L^{\infty}} \| u_{j}\|_{L^2}+ \var\|(\mathcal{R}_{1,j},\mathcal{R}_{2,j})\|_{L^2}+\|H(n)_{j}\|_{L^2}+2^{-j}\|\partial_{t}H(n)_{j}\|_{L^2}\Big{)} d\tau.
\end{aligned}
\end{equation}
Performing similar arguments as in (\ref{I1})-(\ref{271}), we get
\begin{equation}\label{332}
\begin{aligned}
&\var\|\tau^{\theta}(n,u,\psi,\nabla\psi)\|_{\widetilde{L}^{\infty}_{t}(\dot{B}^{\frac{d}{2}+1}_{2,1})}^{h}+\|\tau^{\theta}(n,u,\psi,\nabla\psi)\|_{L^{1}_{t}(\dot{B}^{\frac{d}{2}+1}_{2,1})}^{h}\\
&\quad\lesssim \var\int_{0}^{t} \tau^{\theta-1}\big{(}  \|(n,u,\psi,\nabla\psi)\|_{\dot{B}^{\frac{d}{2}+1}_{2,1}}^{h}+\|H(n)\|_{\dot{B}^{\frac{d}{2}}_{2,1}}^{h}\big{)}d\tau+\mathcal{X}(t)\mathcal{D}^{\theta}(t).
\end{aligned}
\end{equation}
The details are omitted for brevity. In view of the dissipative structures in high frequencies, the first term on the right-hand side of (\ref{332}) can be estimated by
\begin{equation}\label{333}
\begin{aligned}
&\var\int_{0}^{t} \tau^{\theta-1}  \|(n,u,\psi,\nabla\psi)\|_{\dot{B}^{\frac{d}{2}+1}_{2,1}}^{h}d\tau\\
&\quad\lesssim\int_{0}^{t} \tau^{\frac{2}{\frac{d}{2}+2-\sigma_{0}}\big(\theta-\frac{1}{2}(\frac{d}{2}-\sigma_{0})\big)} \Big{(} \var\|(n,u,\psi,\nabla\psi)\|_{\dot{B}^{\frac{d}{2}+1}_{2,1}}^{h}\Big{)}^{\frac{2}{\frac{d}{2}+2-\sigma_{0}}} \Big{(}\|(n,u,\psi,\nabla\psi)\|_{\dot{B}^{\frac{d}{2}+1}_{2,1}}^{h}\Big{)}^{\frac{\frac{d}{2}-\sigma_{0}}{\frac{d}{2}+2-\sigma_{0}}}d\tau\\
&\quad\lesssim \Big{(} \mathcal{X}(t) t^{\theta-\frac{1}{2}(\frac{d}{2}-\sigma_{0})}\Big{)}^{\frac{2}{\frac{d}{2}+2-\sigma_{0}}} \Big{(} \mathcal{D}_{H}^{\theta}(t) \Big{)}^{\frac{\frac{d}{2}-\sigma_{0}}{\frac{d}{2}+2-\sigma_{0}}}\\
&\quad\lesssim \frac{\mathcal{X}(t)}{\zeta}t^{\theta-\frac{1}{2}(\frac{d}{2}-\sigma_{0})}+\zeta \mathcal{D}_{H}^{\theta}(t).
\end{aligned}
\end{equation}
In accordance with (\ref{hl}) and (\ref{F1}), it follows that
\begin{equation}\label{334}
\begin{aligned}
&\var\int_{0}^{t} \tau^{\theta-1}\|H(n)\|_{\dot{B}^{\frac{d}{2}}_{2,1}}^{h}d\tau\\
&\quad\lesssim \int_{0}^{t} \tau^{\theta-1}\var\big{(} \|n\|_{\dot{B}^{\frac{d}{2}}_{2,1}}^{\ell}+\var\|n\|_{\dot{B}^{\frac{d}{2}+1}_{2,1}}^{h}\big{)}d\tau\\
&\quad\lesssim \Big{(}\mathcal{X}_{L}^{\sigma_{0}}(t)t^{\theta-\frac{1}{2}(\frac{d}{2}-\sigma_{0})}\Big{)}^{\frac{2}{\frac{d}{2}+2-\sigma_{0}}}\Big{(}  \var\|\tau^{\theta}n\|_{L^1_{t}(\dot{B}^{\frac{d}{2}+2}_{2,1})}^{\ell}\Big{)}^{\frac{\frac{d}{2}-\sigma_{0}}{\frac{d}{2}+2-\sigma_{0}}}\\
&\quad\quad+ \Big{(}\mathcal{X}(t)t^{\theta-\frac{1}{2}(\frac{d}{2}-\sigma_{0})}\Big{)}^{\frac{2}{\frac{d}{2}+2-\sigma_{0}}}\Big{(}  \|\tau^{\theta}n\|_{L^1_{t}(\dot{B}^{\frac{d}{2}+1}_{2,1})}^{h}\Big{)}^{\frac{\frac{d}{2}-\sigma_{0}}{\frac{d}{2}+2-\sigma_{0}}}\\
&\quad\lesssim \frac{\mathcal{X}_{L}^{\sigma_{0}}(t)+\mathcal{X}(t)}{\zeta}t^{\theta-\frac{1}{2}(\frac{d}{2}-\sigma_{0})}+\zeta \mathcal{D}^{\theta}(t).
\end{aligned}
\end{equation}
The combination of (\ref{332})-(\ref{334}) gives rise to (\ref{decayL}).
\end{proof}

\vspace{2ex}

\underline{\it\textbf{Proof of Theorem \ref{theorem14}:}}~
Assume that (\ref{ABmu}), (\ref{a1}) and (\ref{a3}) holds. Let $(n,u,\psi)$ be the global solution to System (\ref{m1n}) given by Theorem \ref{theorem11}, and $\mathcal{X}(t)$, $\mathcal{X}_{L}^{\sigma_{0}}(t)$ and $\mathcal{D}^{\theta}(t)$ be defined by (\ref{Xdef}), $(\ref{XLsigma})$ and $(\ref{decayX})$, respectively. For any constants $\theta>1+\frac{1}{2}(\frac{d}{2}-\sigma_{0})$ and $\zeta>0$, it follows from Lemmas \ref{lemma52}-\ref{lemma53} that
\begin{equation}\label{ooo}
\begin{aligned}
&\mathcal{D}^{\theta}(t)\lesssim \frac{\mathcal{X}(t)+\mathcal{X}_{L}^{\sigma_{0}}(t)}{\zeta}t^{\theta}(\var t)^{-\frac{1}{2}(\frac{d}{2}-\sigma_{0})}+  \big{(}\zeta+\mathcal{X}(t) \big{)}\mathcal{D}^{\theta}(t),\quad\quad t>0.
\end{aligned}
\end{equation}
 By (\ref{XX0}), (\ref{lowd}) and (\ref{ooo}), we deduce after choosing a suitably small constant $\zeta>0$ that
\begin{equation}\label{Xtimetime}
\begin{aligned}
&t^{\theta}\big{(}\|(n,u,\psi)\|_{\dot{B}^{\frac{d}{2}}_{2,1}}^{\ell}+\var\|(n,u,\nabla\psi )\|_{\dot{B}^{\frac{d}{2}+1}_{2,1}}^{h}\big{)}\lesssim \mathcal{D}^{\theta}(t)\lesssim t^{\theta} (\var t)^{-\frac{1}{2}(\frac{d}{2}-\sigma_{0})}.
\end{aligned}
\end{equation}
Using (\ref{hl}), (\ref{Xtimetime}) and $\mathcal{X}(t)\lesssim \mathcal{X}_{0}$, we infer that
\begin{equation}\label{d21}
\begin{aligned}
&\|(n,u,\psi)\|_{\dot{B}^{\frac{d}{2}}_{2,1}}\lesssim\|(n,u,\psi)\|_{\dot{B}^{\frac{d}{2}}_{2,1}}^{\ell}+\var\|(n,u,\nabla\psi )\|_{\dot{B}^{\frac{d}{2}+1}_{2,1}}^{h}\lesssim (1+\var t)^{-\frac{1}{2}(\frac{d}{2}-\sigma_{0})}.
\end{aligned}
\end{equation}
Thence, it holds by (\ref{hl}), (\ref{lowd}), (\ref{d21}) and (\ref{inter}) that
\begin{equation}\nonumber
\begin{aligned}
&\|(n,u,\psi)\|_{\dot{B}^{\sigma}_{2,1}}\lesssim \|(n,u,\psi)\|_{\dot{B}^{\sigma_{0}}_{2,\infty}}^{\frac{\frac{d}{2}-\sigma}{\frac{d}{2}-\sigma_{0}}} \|(n,u,\psi)\|_{\dot{B}^{\frac{d}{2}}_{2,1}}^{\frac{\sigma-\sigma_{0}}{\frac{d}{2}-\sigma_{0}}}\\
&\quad\quad\quad\quad\quad\quad\lesssim \big{(}\|(n,u,\psi)\|_{\dot{B}^{\sigma_{0}}_{2,\infty}}^{\ell}+\var\|(n,u,\nabla\psi)\|_{\dot{B}^{\frac{d}{2}+1}_{2,1}}^{h}\big{)}^{\frac{\frac{d}{2}-\sigma}{\frac{d}{2}-1-\sigma_{0}}}(1+\var t)^{-\frac{\sigma-\sigma_{0}}{\frac{d}{2}-\sigma_{0}}\frac{1}{2}(\frac{d}{2}-\sigma_{0})}\\
&\quad\quad\quad\quad\quad\quad\lesssim (1+\var t)^{-\frac{1}{2}(\sigma-\sigma_{0})},\quad\quad \sigma\in (\sigma_{0},\frac{d}{2}).
\end{aligned}
\end{equation}
This together with (\ref{d21}) yields the time-decay estimates $(\ref{decay1})$.

We turn to prove $(\ref{decay3})$ for $d\geq2$ and $\sigma_{0}\in[-\frac{d}{2},\frac{d}{2}-1)$. Note that the equation (\ref{m1n}) can be rewritten by
\begin{equation}\label{u}
\begin{aligned}
u= e^{-\frac{1}{\var} t} u_{0}+\int_{0}^{t} e^{-\frac{1}{\var}(t-\tau)} \big{(} -c_{0}\nabla n+\mu \nabla \psi-u\cdot\nabla u \big{)} d\tau,
\end{aligned}
\end{equation}
which gives
\begin{equation}\label{432}
\begin{aligned}
&\|u\|_{\dot{B}^{\sigma_{0}}_{2,\infty}}^{\ell}\lesssim e^{-\frac{1}{\var} t} \|u_{0}\|_{\dot{B}^{\sigma_{0}}_{2,\infty}}^{\ell}+\int_{0}^{t} e^{-\frac{1}{\var}(t-\tau)} \big{(} \|(n,\psi)\|_{\dot{B}^{\sigma_{0}+1}_{2,\infty}}^{\ell}+\|u\cdot \nabla u\|_{\dot{B}^{\sigma_{0}}_{2,\infty}}^{\ell}\big{)}d\tau.
\end{aligned}
\end{equation}
Owing to $\sigma_{0}\leq \frac{d}{2}-1$ and $(\ref{decay1})_{1}$, we get
\begin{equation}\label{433}
\begin{aligned}
\|(n,\psi)\|_{\dot{B}^{\sigma_{0}+1}_{2,\infty}}^{\ell}\lesssim  (1+\var t)^{-\frac{1}{2}},
\end{aligned}
\end{equation}
It also holds by (\ref{decay1}), (\ref{uv3}) and $\|u\|_{\dot{B}^{\frac{d}{2}}_{2,1}}\lesssim \mathcal{X}(t)\lesssim \mathcal{X}_{0}$ that
\begin{equation}\label{434}
\begin{aligned}
\|u\cdot \nabla u\|_{\dot{B}^{\sigma_{0}}_{2,\infty}}^{\ell}\lesssim \|u\|_{\dot{B}^{\frac{d}{2}}_{2,1}}\| u\|_{\dot{B}^{\sigma_{0}+1}_{2,\infty}}\lesssim (1+\var t)^{-\frac{1}{2}}.
\end{aligned}
\end{equation}
Inserting (\ref{433})-(\ref{434}) into (\ref{432}) shows
\begin{equation}\nonumber
\begin{aligned}
\|u\|_{\dot{B}^{\sigma_{0}}_{2,\infty}}&\lesssim \|u\|_{\dot{B}^{\sigma_{0}}_{2,\infty}}^{\ell}+\|u\|_{\dot{B}^{\frac{d}{2}}_{2,1}}^{h}\lesssim e^{-\frac{1}{\var} t} +\int_{0}^{t}e^{-\frac{1}{\var} (t-\tau)} (1+ \var \tau)^{-\frac{1}{2}}d\tau\lesssim (1+\var t )^{-\frac{1}{2}},
\end{aligned}
\end{equation}
where one has used (\ref{d21}) and the estimate
\begin{equation}\nonumber
\begin{aligned}
    &\int_{0}^{t}e^{-\frac{1}{\var} (t-\tau)} (1+ \var \tau)^{-\frac{1}{2}}d\tau\\
    &\quad\lesssim e^{-\frac{1}{2\var}t}\int_{0}^{\frac{t}{2}} (1+\var\tau)^{-\frac{1}{2}}d\tau+(1+\frac{\var}{2}t)^{-\frac{1}{2}}\int_{\frac{t}{2}}^{t} e^{-\frac{1}{\var}(t-\tau)}d\tau\lesssim \frac{1}{\var}(1+\var t)^{-\frac{1}{2}}.
\end{aligned}
\end{equation}
In addition, it is easy to verify that $\widetilde{\varphi}:=b\psi-c_{1}n$ satisfies
\begin{equation}\label{widephi}
\begin{aligned}
&\partial_{t}\widetilde{\varphi}+b\widetilde{\varphi}=b\Delta\psi+c_{1}c_{0}\div u+b H(n)+c_{1} u\cdot\nabla u+c_{1}G(n)\div u.
\end{aligned}
\end{equation}
Similarly, by virtue of (\ref{decay1}), (\ref{hl}), (\ref{widephi}), (\ref{uv3}) and (\ref{F1}), we get
\begin{equation}\label{widephi11}
\begin{aligned}
&\|\widetilde{\varphi}\|_{\dot{B}^{\sigma_{0}}_{2,\infty}}^{\ell}\lesssim e^{-bt} \|(n_{0},\psi_{0})\|_{\dot{B}^{\sigma_{0}}_{2,\infty}}^{\ell}+\int_{0}^{t}e^{-b(t-\tau)}\big(\frac{1}{\var} \|\psi\|_{\dot{B}^{\sigma_{0}+1}_{2,\infty}}^{\ell}+\|u\|_{\dot{B}^{\sigma_{0}+1}_{2,\infty}}^{\ell}\\
&\quad\quad\quad\quad~~+\|(n,u)\|_{\dot{B}^{\frac{d}{2}}_{2,1}} (\|(n,u)\|_{\dot{B}^{\sigma_{0}+1}_{2,\infty}}^{\ell}+\|(n,u)\|_{\dot{B}^{\frac{d}{2}}_{2,1}}^{h})+\|H(n)\|_{\dot{B}^{\sigma_{0}}_{2,\infty}}^{\ell}\big)dx.
\end{aligned}
\end{equation}
Owing to (\ref{hl}), (\ref{lowd}), (\ref{d21}) and the composition estimates (\ref{F1}) and (\ref{q5}), there holds
\begin{equation}\label{widephi12}
\left\{
\begin{aligned}
&\|H(n)\|_{\dot{B}^{\frac{d}{2}}_{2,1}}\lesssim \|n\|_{\dot{B}^{\frac{d}{2}}_{2,1}}\lesssim (1+\var t)^{-\frac{1}{2}(\frac{d}{2}-\sigma_{0})}\lesssim  (1+\var t)^{-\frac{1}{2}},\\
&\|H(n)\|_{\dot{B}^{\sigma_{0}}_{2,\infty}}^{\ell}\lesssim \|n\|_{\dot{B}_{2,1}^{\frac{d}{2}}}(\|n\|_{\dot{B}^{\sigma_{0}}_{2,\infty}}^{\ell}+\var\|n\|_{\dot{B}^{\frac{d}{2}+1}_{2,1}})\lesssim (1+\var t)^{-\frac{1}{2}(\frac{d}{2}-\sigma_{0})}\lesssim  (1+\var t)^{-\frac{1}{2}}.
\end{aligned}
\right.
\end{equation}
Combining (\ref{hl}), (\ref{d21}) and (\ref{widephi11})-(\ref{widephi12}) together, we derive
\begin{equation}\nonumber
\begin{aligned}
&\|B\psi-c_{1}n-H(n)\|_{\dot{B}^{\sigma_{0}}_{2,\infty}}\\
&\quad\lesssim\|\widetilde{\varphi}\|_{\dot{B}^{\sigma_{0}}_{2,\infty}}^{\ell}+\|H(n)\|_{\dot{B}^{\sigma_{0}}_{2,\infty}}^{\ell}+\|(n,\nabla\psi)\|_{\dot{B}^{\frac{d}{2}}_{2,1}}^{h}+\|H(n)\|_{\dot{B}^{\frac{d}{2}}_{2,1}}^{h}\lesssim  \frac{1}{\var}(1+\var t)^{-\frac{1}{2}}.
\end{aligned}
\end{equation}

Finally, we take the $\dot{B}^{\sigma}_{2,1}$-norm of (\ref{u}) for any $\sigma\in (\sigma_{0},\frac{d}{2}-1]$ to have
\begin{equation}\nonumber
\begin{aligned}
&\|u\|_{\dot{B}^{\sigma}_{2,1}}^{\ell}\lesssim e^{-\frac{1}{\var} t} \|u_{0}\|_{\dot{B}^{\sigma}_{2,1}}^{\ell}+\int_{0}^{t} e^{-\frac{1}{\var}(t-\tau)} \big{(} \|(n,\psi)\|_{\dot{B}^{\sigma+1}_{2,1}}^{\ell}+\|u\|_{\dot{B}^{\frac{d}{2}}_{2,1}}\| u\|_{\dot{B}^{\sigma+1}_{2,1}}\big{)}d\tau\\
&\quad\quad\quad~\lesssim\frac{1}{\var} (1+\var t)^{-\frac{1}{2}(1+\sigma-\sigma_{0})},
\end{aligned}
\end{equation}
where one has used $\|(n,u,\psi)\|_{\dot{B}^{\sigma+1}_{2,1}}^{\ell}\lesssim  (1+\var t)^{-\frac{1}{2}(1+\sigma-\sigma_{0})}$ derived from $(\ref{decay1})_{1}$. Therefore, the following decay rate holds:
\begin{equation}\nonumber
\begin{aligned}
&\|u\|_{\dot{B}^{\sigma}_{2,1}}\lesssim \|u\|_{\dot{B}^{\sigma}_{2,1}}^{\ell}+\|u\|_{\dot{B}^{\frac{d}{2}}_{2,1}}^{h}\lesssim \frac{1}{\var} (1+\var t)^{-\frac{1}{2}(1+\sigma-\sigma_{0})}.
\end{aligned}
\end{equation}
Similarly, 
we can show
\begin{equation}\nonumber
\begin{aligned}
&\|B\psi-c_{1}n-H(n)\|_{\dot{B}^{\sigma}_{2,1}}^{\ell}~\lesssim\frac{1}{\var} (1+\var t)^{-\frac{1}{2}(1+\sigma-\sigma_{0})}.
\end{aligned}
\end{equation}
The details are omitted here. The proof of Theorem \ref{theorem14} is complete.

 \section{Relaxation limit}\label{section4}

 \subsection{Global well-posedness problem for the Keller-Segel equations}\label{subsection35}
 
 In this subsection, we prove Theorem \ref{theorem12} related to the global well-posedness of strong solutions to the Cauchy problem for System (\ref{KS}). Thanks to the Taylor formula, there exists a function $G_{1}$ vanishing at $\bar{\rho}$ such that
\begin{equation}\label{Ptay}
\begin{aligned}
 P(\rho^*)-P(\bar{\rho})= P'(\bar{\rho})\,(\rho^*-\bar{\rho})+ G_1(\rho^*)\,(\rho^*-\bar{\rho}).
 \end{aligned}
 \end{equation}
We also notice that
 \begin{equation}\label{Ptay2}
 \begin{aligned}
 &\Delta\phi^{*}=\mu a\bar{\rho}(b-\Delta)^{-1} \Delta \rho^{*}.
 \end{aligned}
 \end{equation}
 Substituting (\ref{Ptay})-(\ref{Ptay2}) into ${\rm{(\ref{KS})_{1}}}$, we can rewrite System (\ref{KS}) as
\begin{equation}\label{KS1}
\left\{
\begin{aligned}
&\partial_{t}\rho^{*}-\tilde{\Delta}_{*}\rho^*=\Delta(G_1(\rho^*)\,(\rho^*-\bar{\rho}))-\mu\div((\rho^*-\bar{\rho})\nabla \phi^*),\\
&\phi^{*}-\bar{\phi}=a(b-\Delta)^{-1}(\rho^*-\bar{\rho}),
\end{aligned}
\right.
\end{equation}
where $\tilde{\Delta}_{*}$ is the differential operator
\begin{equation}\label{widedelta}
\begin{aligned}
\tilde{\Delta}_{*}:=(P'(\bar{\rho})-\mu a\bar{\rho} (b-\Delta)^{-1} )\Delta.
\end{aligned}
\end{equation}
Consider the linear problem
\begin{equation}\label{linear3}
\left\{
\begin{aligned}
&\partial_{t}f-\tilde{\Delta}_{*}f=g ,\quad\quad x\in\mathbb{R}^{d},\quad t>0,\\
&f(x,0)=f_{0}(x),\quad~~ x\in\mathbb{R}^{d}.
\end{aligned}
\right.
\end{equation}
\begin{lemma}\label{lemmadelta}
Let $P'(\bar{\rho})>\dfrac{\mu a}{b}\bar{\rho}$. For $s\in\mathbb{R}$ and $p\in[1, \infty]$, assume $u_{0}\in\dot{B}^{s}_{p,1}$ and $g\in L^{1}(0,T;\dot{B}^{s}_{p,1})$. If $f$ is a solution to the problem {\rm(\ref{linear3})}, then there exists a universal constant $C>0$ such that
\begin{equation}\label{deltalp}
\begin{aligned}
&\|f\|_{\widetilde{L}^{\infty}_{t}(\dot{B}^{s}_{p,1})}+\|f\|_{L^{1}_{t}(\dot{B}^{s+2}_{p,1})}\leq C(\|f_{0}\|_{\dot{B}^{s}_{p,1}}+\|g\|_{L^{1}_{t}(\dot{B}^{s}_{p,1})}),\quad 0<t<T.
\end{aligned}
\end{equation}
\end{lemma}
\begin{proof}
Let the semi-group operator $e^{\tilde{\Delta}_{*} t}$ be defined as
\begin{equation}\nonumber
\begin{aligned}
e^{\tilde{\Delta}_{*} t}  f= \mathcal{F}^{-1} \big{(} e^{-(P'(\bar{\rho})-\frac{\mu a}{b}\bar{\rho})|\xi|^2t-\frac{\mu a\bar{\rho} |\xi|^4}{b+|\xi|^2}t} \mathcal{F}(f)(\xi)\Big{)}.
\end{aligned}
\end{equation}
By similar calculations as in \cite[Page 54]{bahouri1}, one can show that
\begin{equation}\label{splocal}
\begin{aligned}
&\text{Supp}~\mathcal{F}(f)\subset \lambda \mathbf{C} \Rightarrow   \|e^{\tilde{\Delta}_{*} t} f\|_{L^{p}}\leq C e^{-C(P'(\bar{\rho})-\frac{\mu a}{b}\bar{\rho})\lambda^2 t}\|f\|_{L^{p}}.
\end{aligned}
\end{equation}
for any annulus $\mathbf{C}$. Thus, since $P'(\bar{\rho})>\dfrac{\mu a}{b}\bar{\rho}$, we are able to obtain (\ref{deltalp}) by virtue of usual energy estimates similar to the heat equation. For brevity, the details are omitted.
\end{proof}

 \underline{\it\textbf{Proof of Theorem \ref{theorem12}:~}}
 It follows from (\ref{KS1}) and Lemma \ref{lemmadelta} that
 \begin{equation}\label{KSest}
\begin{aligned}
&\|\rho^*-\bar{\rho}\|_{\widetilde{L}^{\infty}_{t}(\dot{B}^{\frac{d}{p}}_{p,1})}+\|\rho^*-\bar{\rho}\|_{L^1_{t}(\dot{B}^{\frac{d}{p}+2}_{p,1})}\\
&\quad\lesssim \|\rho^{*}_0-\bar{\rho}\|_{\dot{B}^{\frac{d}{p}}_{p,1}}+\|G_1(\rho^*)\,(\rho^*-\bar{\rho})\|_{L^1_{t}(\dot{B}^{\frac{d}{p}+2}_{p,1})}+\|(\rho^*-\bar{\rho})\nabla \phi^*\|_{L^1_{t}(\dot{B}^{\frac{d}{p}+1}_{p,1})}.
\end{aligned}
\end{equation}
Making use of $\eqref{KS1}_{2}$ and the properties (\ref{bessel}) of the Bessel potential $(b-\Delta)^{-1}$ yield
\begin{equation}\label{KSest13}
\left\{
\begin{aligned}
&\|\phi^*-\bar{\phi}\|_{L^1_{t}(\dot{B}^{\frac{d}{p}+2}_{p,1})}\leq \|\rho-\bar{\rho}\|_{L^1_{t}(\dot{B}^{\frac{d}{p}+2}_{p,1})},\\
&\|\phi^*-\bar{\phi}\|_{\widetilde{L}^2_{t}(\dot{B}^{\frac{d}{p}+1}_{p,1})}\lesssim \|\rho-\bar{\rho}\|_{\widetilde{L}^2_{t}(\dot{B}^{\frac{d}{p}+1}_{p,1})}\lesssim \|\rho^*-\bar{\rho}\|_{\widetilde{L}^{\infty}_{t}(\dot{B}^{\frac{d}{p}}_{p,1})}+\|\rho^*-\bar{\rho}\|_{L^1_{t}(\dot{B}^{\frac{d}{p}+2}_{p,1})}.
\end{aligned}
\right.
\end{equation}
We employ (\ref{KSest13}), the product law (\ref{uv1}) and the composition estimate (\ref{F1}) to have
 \begin{equation}\label{510}
\begin{aligned}
&\|G_1(\rho^*)\,(\rho^*-\bar{\rho})\|_{L^1_{t}(\dot{B}^{\frac{d}{p}+2}_{p,1})}+\|(\rho^*-\bar{\rho})\nabla \phi^*\|_{L^1_{t}(\dot{B}^{\frac{d}{p}+1}_{p,1})}\\
&\quad\lesssim \|\rho^*-\bar{\rho}\|_{L^{\infty}_{t}(\dot{B}^{\frac{d}{p}}_{p,1})}\|\rho^*-\bar{\rho}\|_{L^1_{t}(\dot{B}^{\frac{d}{p}+2}_{p,1})}+\|\rho^*-\bar{\rho}\|_{L^{\infty}_{t}(\dot{B}^{\frac{d}{p}}_{p,1})}\|\phi^*-\bar{\phi}\|_{L^1_{t}(\dot{B}^{\frac{d}{p}+2}_{p,1})}\\
&\quad\quad+\|\rho^*-\bar{\rho}\|_{L^{2}_{t}(\dot{B}^{\frac{d}{p}+1}_{p,1})}\|\phi^*-\bar{\phi}\|_{L^2_{t}(\dot{B}^{\frac{d}{p}+1}_{p,1})}\\
&\quad\lesssim \big( \|\rho^*-\bar{\rho}\|_{L^{\infty}_{t}(\dot{B}^{\frac{d}{p}}_{p,1})}+\|\rho^*-\bar{\rho}\|_{L^1_{t}(\dot{B}^{\frac{d}{p}+2}_{p,1})}\big)^2.
\end{aligned}
\end{equation}
 By (\ref{KSest}) and (\ref{510}), one obtains
 \begin{equation}\label{KSest2}
\begin{aligned}
&\|\rho^*-\bar{\rho}\|_{\widetilde{L}^{\infty}_{t}(\dot{B}^{\frac{d}{p}}_{p,1})}+\|\rho^*-\bar{\rho}\|_{L^1_{t}(\dot{B}^{\frac{d}{p}+2}_{p,1})}\\
&\quad\lesssim \|\rho^{*}_0-\bar{\rho}\|_{\dot{B}^{\frac{d}{p}}_{p,1}}+\big( \|\rho^*-\bar{\rho}\|_{\widetilde{L}^{\infty}_{t}(\dot{B}^{\frac{d}{p}}_{p,1})}+\|\rho^*-\bar{\rho}\|_{L^1_{t}(\dot{B}^{\frac{d}{p}+2}_{p,1})}\big)^2.
\end{aligned}
\end{equation}
Hence, in view of the Friedrichs approximation and bootstrap argument, the left-hand side of (\ref{KSest2}) can be bounded
for all $T>0$ provided that \eqref{a2} is 
satisfied with $\alpha^*$ small enough. Then it is easy to deduce the global existence of strong solutions to the Cauchy problem for System \eqref{KS}. And the uniqueness in the space $\widetilde{L}^{\infty}(0,T;\dot{B}^{\frac{d}{p}-1}_{p,1})\cap L^1(0,T;\dot{B}^{\frac{d}{p}+1}_{p,1})$ can be proved in a simple way. We omit the details for brevity.

 \subsection{Quantitative error estimates}\label{subsection42}

 This subsection is devoted to the proof of Theorem \ref{theorem13} pertaining to the relaxation limit from the HPC system to the Keller-Segel model. Recall that $(\rho^{\var},u^{\var},\phi^{\var})$ is defined through the following diffusive scaling (\ref{scaling1}) and it satisfies System (\ref{HPCvar}).

From the bounds obtained with Theorem \ref{theorem11}, we deduce the following lemma concerning the unknowns $(\rho^\var, u^\var,\phi^\var)$ under the diffusive scaling.
\begin{cor} \label{boundsrho}
Let the assumptions {\rm(\ref{ABmu})}, {\rm{(\ref{a1})}} and {\rm{(\ref{a2})}} hold and let $(n, u,\psi)$ be the global solution to System {\rm(\ref{m1n})} given by Theorem \ref{theorem11}. Then $(\rho^\var, u^\var,\phi^\var)$ satisfies
\begin{equation}\label{boundsRelax}
\left\{
\begin{aligned}
&\varepsilon\|u^\varepsilon\|^{\ell}_{\widetilde{L}^\infty_t(\dot{B}^{\frac{d}{2}}_{2,1})}+\varepsilon^2\| u^\varepsilon\|^{h}_{\widetilde{L}^\infty_t(\dot{B}^{\frac{d}{2}+1}_{2,1})}+\|u^\varepsilon\|_{L^1_t(\dot{B}^{\frac{d}{2}+1}_{2,1})}^{\ell}+\|u^\varepsilon\|_{\widetilde{L}^2_t(\dot{B}^{\frac{d}{2}}_{2,1})}\leq C\mathcal{X}_{0},\\
&\|\phi^\var-\bar{\phi}\|_{\widetilde{L}^{\infty}_{t}(\dot{B}^{\frac{d}{2}}_{2,1})}^{\ell}+\var\|\phi^\var-\bar{\phi}\|_{\widetilde{L}^{\infty}_{t}(\dot{B}^{\frac{d}{2}+2}_{2,1})}^{h}+\|\phi^\var-\bar{\phi}\|_{L^1_{t}(\dot{B}^{\frac{d}{2}+2}_{2,1})}^{\ell}+\|\phi^\var-\bar{\phi}\|_{L^{1}_{t}(\dot{B}^{\frac{d}{2}+3}_{2,1})}^{h}\leq C\mathcal{X}_{0},\\
&\|\rho^\varepsilon-\bar{\rho}\|_{\widetilde{L}^\infty_t(\dot{B}^{\frac{d}{2}}_{2,1})}^{\ell}+\var\|\rho^\varepsilon-\bar{\rho}\|_{\widetilde{L}^\infty_t(\dot{B}^{\frac{d}{2}+1}_{2,1})}^{h}+\|\rho^\varepsilon-\bar{\rho}\|_{L^1_t(\dot{B}^{\frac{d}{2}+2}_{2,1})}^{\ell}+\var^{-1}\|\rho^\varepsilon-\bar{\rho}\|_{L^1_t(\dot{B}^{\frac{d}{2}+1}_{2,1})}^{h}\leq C\mathcal{X}_{0},\\
&\|\rho^\varepsilon-\bar{\rho}\|_{\widetilde{L}^\infty_t(\dot{B}^{\frac{d}{2}}_{2,1})}+\|(\rho^\var-\bar{\rho},\phi^\var-\bar{\phi})\|_{\widetilde{L}^2_t(\dot{B}^{\frac{d}{2}+1}_{2,1})}\leq C\mathcal{X}_{0},\\
&\| v^{\var}\|_{L^1_t(\dot{B}^{\frac{d}{2}}_{2,1})}+\var\|\partial_t\phi^\var\|_{L^1_t(\dot{B}^{\frac{d}{2}}_{2,1})}\leq C\mathcal{X}_{0}\var,
\end{aligned}
\right.
\end{equation}
with $v^{\var}= u^\varepsilon+\dfrac{1}{\rho^{\var}}\nabla P(\rho^\varepsilon)-\mu\nabla\phi^\var$.
\end{cor}
\begin{remark}
Note that $(\ref{boundsRelax})_{4}$-$(\ref{boundsRelax})_{5}$ are the main ingredients to show the relaxation limit. 
\end{remark}

\begin{proof}
The inequality $(\ref{boundsRelax})_{1}$-$(\ref{boundsRelax})_{2}$ can be derived by (\ref{XX0}) and the scaling (\ref{scaling1}) directly. Recall that in (\ref{m1n}), we have 
$a(\rho-\bar{\rho})=c_{1}n+H(n).$ Thus, we are able to apply (\ref{scaling1}), (\ref{hl}) and the composition estimates (\ref{q1})-(\ref{q2}) to have
\begin{equation}\nonumber
\begin{aligned}
&\|\rho^{\var}-\bar{\rho}\|_{\widetilde{L}^{\infty}_{t}(\dot{B}^{\frac{d}{2}}_{2,1})}^{\ell}+\var\|\rho^{\var}-\bar{\rho}\|_{\widetilde{L}^{\infty}_{t}(\dot{B}^{\frac{d}{2}+1}_{2,1})}^{h}\\
&\quad=\|\rho-\bar{\rho}\|_{\widetilde{L}^{\infty}_{t}(\dot{B}^{\frac{d}{2}}_{2,1})}^{\ell}+\var\|\rho-\bar{\rho}\|_{\widetilde{L}^{\infty}_{t}(\dot{B}^{\frac{d}{2}+1}_{2,1})}^{h}\\
&\quad\lesssim\|n\|_{\widetilde{L}^{\infty}_{t}(\dot{B}^{\frac{d}{2}}_{2,1})}^{\ell}+\|H(n)\|_{\widetilde{L}^{\infty}_{t}(\dot{B}^{\frac{d}{2}}_{2,1})}^{\ell}+\var\|n\|_{\widetilde{L}^{\infty}_{t}(\dot{B}^{\frac{d}{2}+1}_{2,1})}^{h}+\var\|H(n)\|_{\widetilde{L}^{\infty}_{t}(\dot{B}^{\frac{d}{2}+1}_{2,1})}^{h}\\
&\quad\lesssim \|n\|_{\widetilde{L}^{\infty}_{t}(\dot{B}^{\frac{d}{2}}_{2,1})}^{\ell}+\var\|n\|_{\widetilde{L}^{\infty}_{t}(\dot{B}^{\frac{d}{2}+1}_{2,1})}^{h}\lesssim \mathcal{X}_{0}.
\end{aligned}
\end{equation}
Similarly, one gets
\begin{equation}\nonumber
\begin{aligned}
&\|\rho^{\var}-\bar{\rho}\|_{L^1_{t}(\dot{B}^{\frac{d}{2}+2}_{2,1})}^{\ell}+\frac{1}{\var}\|\rho^{\var}-\bar{\rho}\|_{L^1_{t}(\dot{B}^{\frac{d}{2}+1}_{2,1})}^{h}\\
&\quad=\var\|\rho-\bar{\rho}\|_{L^1_{t}(\dot{B}^{\frac{d}{2}+2}_{2,1})}^{\ell}+\|\rho-\bar{\rho}\|_{L^1_{t}(\dot{B}^{\frac{d}{2}+1}_{2,1})}^{h}\\
&\quad\lesssim\var\|n\|_{L^1_{t}(\dot{B}^{\frac{d}{2}+2}_{2,1})}^{\ell}+\var\|H(n)\|_{L^1_{t}(\dot{B}^{\frac{d}{2}+2}_{2,1})}^{\ell}+\|n\|_{L^1_{t}(\dot{B}^{\frac{d}{2}+1}_{2,1})}^{h}+\var\|H(n)\|_{L^1_{t}(\dot{B}^{\frac{d}{2}+1}_{2,1})}^{h}\\
&\quad\lesssim \var\|n\|_{L^1_{t}(\dot{B}^{\frac{d}{2}+2}_{2,1})}^{\ell}+\|n\|_{L^1_{t}(\dot{B}^{\frac{d}{2}+1}_{2,1})}^{h}\lesssim \mathcal{X}_{0}.
\end{aligned}
\end{equation}
Thence, $(\ref{boundsRelax})_{4}$ follows by $(\ref{boundsRelax})_{2}$-$(\ref{boundsRelax})_{3}$ and standard interpolation inequalities. We also obtain from (\ref{scaling1}), (\ref{XX0}), (\ref{hl}) and $(\ref{boundsRelax})_{2}$ that
\begin{equation}\nonumber
\begin{aligned}
&\| v^{\var}\|_{L^1_t(\dot{B}^{\frac{d}{2}}_{2,1})}=\var\|\frac{1}{\var}u+ \nabla n-\mu \nabla \psi\|_{L^1_t(\dot{B}^{\frac{d}{2}}_{2,1})}\\
&\quad\quad\quad\quad\quad\lesssim \var\|\frac{1}{\var}u+\nabla n- \mu \nabla \psi\|_{L^1_t(\dot{B}^{\frac{d}{2}}_{2,1})}^{\ell}+\var\|(n,u,\nabla\psi)\|_{L^1_{t}(\dot{B}^{\frac{d}{2}+1}_{2,1})}^{h}\lesssim\mathcal{X}_{0} \var.
\end{aligned}
\end{equation}
Finally, it holds from (\ref{XX0}) and $\var\partial_{\tau}\phi^{\var}(x,\tau)=\partial_{t}\psi(x,\var t)$ that
\begin{equation}\nonumber
\begin{aligned}
&\|\partial_{t}\phi^{\var}\|_{L^1_t(\dot{B}^{\frac{d}{2}}_{2,1})}=\|\partial_{t}\psi\|_{L^1_t(\dot{B}^{\frac{d}{2}}_{2,1})}\lesssim \|\partial_{t}\psi\|_{L^1_t(\dot{B}^{\frac{d}{2}}_{2,1})}^{\ell}+\var\|\partial_{t}\psi\|_{L^1_t(\dot{B}^{\frac{d}{2}+1}_{2,1})}^{h}\lesssim\mathcal{X}_{0}.
\end{aligned}
\end{equation}
\end{proof}

\underline{\it\textbf{Proof of Theorem \ref{theorem13}:~}} Let the assumptions $(\ref{ABmu})$-$(\ref{a1})$ hold and let $(\rho^{\var}, u^{\var},\phi^{\var})$ be the global solution to System (\ref{HPCvar}) subject to the initial data $(\rho_{0},u_{0},\phi_{0})$ given by Theorem \ref{theorem11}. According to the equations in (\ref{HPCvar}) and Corollary \ref{boundsrho}, we have the uniform estimates
\begin{equation}\label{m123}
\left\{
    \begin{aligned}
        &\|\rho^{\var} v^{\var}\|_{L^1_{t}(\dot{B}^{\frac{d}{2}}_{2,1})}\lesssim\big( \|\rho^{\var}-\bar{\rho}\|_{\widetilde{L}^{\infty}_{t}(\dot{B}^{\frac{d}{2}}_{2,1})}+\bar{\rho} \big) \|v^{\var}\|_{L^1_t(\dot{B}^{\frac d2}_{2,1})}\lesssim  \mathcal{X}_{0}\var,\\
        &\|-\Delta\phi^{\var}-a\rho^{\var}+b\phi^{\var}\|_{L^1_{t}(\dot{B}^{\frac{d}{2}}_{2,1})}=\var\|\partial_{t}\phi^{\var}\|_{L^1_t(\dot{B}^{\frac{d}{2}}_{2,1})}\lesssim \mathcal{X}_{0}\var,\\
    \end{aligned}
    \right.
\end{equation}
which yields (\ref{uniform1}).

Assume further that \eqref{ar} holds. With these bounds in hand, we can derive expected quantitative error estimates. From the Taylor formula, there exists a smooth function $G_1$ vanishing at $\bar\rho$ such that 
 $$ P(\rho^\varepsilon)-P(\bar{\rho})= P'(\bar{\rho})\,(\rho^\varepsilon-\bar{\rho})+ G_1(\rho^\varepsilon)\,(\rho^\varepsilon-\bar{\rho}).$$ 
 It is easy to verify that  $\rho^\var$ satisfies
  \begin{equation}\label{m1nvar2}
  \left\{
 \begin{aligned}
 &\partial_{t}\rho^{\var}-\tilde{\Delta}_{*} \rho^{\var}=\mathcal{R}^{\var}+\Delta(G_1(\rho^\var)\,(\rho^\var-\bar{\rho}))-\mu\div((\rho^\var-\bar{\rho})\nabla \phi^\var),\\
 &\phi^{\var}=(b-\Delta)^{-1}\big{(} -\var \partial_{t}\phi^{\var}+a \rho^\var \big{)}.
  \end{aligned}
  \right.
 \end{equation}
 with
 $$
 \mathcal{R}^{\var}:=-\div( \rho^{\var} v^{\var} )+\var\mu\bar{\rho}\Delta(b-\Delta)^{-1}\partial_{t}\phi^{\var}.
 $$

Introducing the error variables
 \begin{equation}\nonumber
 \begin{aligned}
 &\delta \rho^{\var}:=\rho^{\var}-\rho^{*},\quad\quad \delta u^{\var}:=u^{\var}-u^{*},\quad\quad \delta \phi^{\var}:=\phi^{\var}-\phi^{*},
 \end{aligned}
 \end{equation}
we deduce from \eqref{KS1} and (\ref{m1nvar2}) that $\delta \rho^{\var}$ and $\delta \phi^{\var}$ satisfies the error equations
  \begin{equation}\label{errorm1}
 \left\{
 \begin{aligned}
 &\partial_{t}\delta \rho^{\var}-\tilde{\Delta}_{*} \delta\rho^\var=\mathcal{R}^{\var}+\Delta((G_1(\rho^\var)-G_1(\rho^*))(\rho^\var-\bar{\rho}))+\Delta(\delta\rho^\var G_1(\rho^*))\\
 &\qquad\qquad\qquad\quad\quad~-\mu\div(\delta\rho^\var\nabla\phi^\var)-\mu\div((\rho^*-\bar{\rho})\nabla\delta\phi^\var),\\
 &\delta\phi^{\var}=(b-\Delta)^{-1}( -\var \partial_{t}\phi^{\var}+a\delta \rho^\var ).
 \end{aligned}
 \right.
 \end{equation}
 with the operator $\tilde{\Delta}_{*}$ defined by (\ref{widedelta}). It follows by Lemma \ref{lemmadelta} and \eqref{ar} that 
\begin{equation}\label{erroreq5}
\begin{aligned}
  &  \|\delta \rho^{\var}\|_{\widetilde{L}^\infty_t(\dot{B}^{\frac{d}{2}-1}_{2,1})}+\|\delta \rho^{\var}\|_{L^1_t(\dot{B}^{\frac{d}{2}+1}_{2,1})}\\
  &\quad  \lesssim \var+\|\mathcal{R}^\var\|_{L^1_t(\dot{B}^{\frac{d}{2}-1}_{2,1})}+\|(G_1(\rho^\var)-G_1(\rho^*))(\rho^\var-\bar{\rho})\|_{L^1_t(\dot{B}^{\frac{d}{2}+1}_{2,1})}\\
  &\quad\quad\quad+\|\delta \rho^{\var}G_1(\rho^*)\|_{L^1_t(\dot{B}^{\frac{d}{2}+1}_{2,1})}+\|\delta \rho^{\var}\nabla\phi^\var\|_{L^1_t(\dot{B}^{\frac{d}{2}}_{2,1})}+\|(\rho^*-\bar{\rho})\nabla\delta\phi^\var\|_{L^1_t(\dot{B}^{\frac{d}{2}}_{2,1})}.
\end{aligned}
\end{equation}
Using (\ref{m123}), we can immediately get that $\mathcal{R}^\var$ is a $\mathcal{O}(\var)$ in $L^1_t(B^{d/2-1}_{2,1})$:
\begin{equation}\nonumber
\begin{aligned}
     \|\mathcal{R}^{\var}\|_{L^1_t(\dot{B}^{\frac d2-1}_{2,1})}\lesssim \|\rho^{\var}v^{\var}\|_{L^1_t(\dot{B}^{\frac d2}_{2,1})}+\var\|\partial_{t}\phi^{\var}\|_{L^1_{t}(\dot{B}^{\frac{d}{2}}_{2,1})}\lesssim\mathcal{X}_{0} \var .
\end{aligned}
\end{equation}
 In accordance with (\ref{boundsRelax}), the product law (\ref{uv1}) and the composition estimate (\ref{F3}) from the Appendix, one has
\begin{equation}\nonumber
\begin{aligned}
&  \|(G_1(\rho^\var)-G_1(\rho^*))(\rho^\var-\bar{\rho})\|_{L^1_t(\dot{B}^{\frac{d}{2}+1}_{2,1})}\\
&\quad\lesssim \|\delta\rho^\var\|_{L^1_t(\dot{B}^{\frac{d}{2}+1}_{2,1})}\|\rho^\var-\bar{\rho}\|_{L^\infty_t(\dot{B}^{\frac{d}{2}}_{2,1})}+\|\delta\rho^\var\|_{L^2_t(\dot{B}^{\frac{d}{2}}_{2,1})}\|\rho^\var-\bar{\rho}\|_{L^2_t(\dot{B}^{\frac{d}{2}+1}_{2,1})}\\
 &\quad\lesssim  \mathcal{X}_{0}\big(\|\delta \rho^{\var}\|_{\widetilde{L}^\infty_t(\dot{B}^{\frac{d}{2}-1}_{2,1})}+\|\delta \rho^{\var}\|_{L^1_t(\dot{B}^{\frac{d}{2}+1}_{2,1})} \big).
\end{aligned}
\end{equation}
 Similarly, it also holds by (\ref{r2}), (\ref{boundsRelax}) and (\ref{uv1}) that
\begin{equation}\nonumber
\begin{aligned}
 & \|\delta \rho^{\var}G_1(\rho^*)\|_{L^1_t(\dot{B}^{\frac{d}{2}+1}_{2,1})}\\
 &\quad\lesssim \|\delta\rho^\var\|_{L^1_t(\dot{B}^{\frac{d}{2}+1}_{2,1})}\|\rho^*-\bar{\rho}\|_{L^\infty_t(\dot{B}^{\frac{d}{2}}_{2,1})}+\|\delta\rho^\var\|_{L^2_t(\dot{B}^{\frac{d}{2}}_{2,1})}\|\rho^*-\bar{\rho}\|_{L^2_T(\dot{B}^{\frac{d}{2}+1}_{2,1})}\\
 &\quad\lesssim  \|\rho_{0}^{*}-\bar{\rho}\|_{\dot{B}^{\frac{d}{2}}_{2,1}} (\|\delta \rho^{\var}\|_{\widetilde{L}^\infty_t(\dot{B}^{\frac{d}{2}-1}_{2,1})}+\|\delta \rho^{\var}\|_{L^1_t(\dot{B}^{\frac{d}{2}+1}_{2,1})}).
\end{aligned}
\end{equation}
One deduces from (\ref{bessel}), (\ref{boundsRelax}) and $(\ref{errorm1})_{2}$ that
\begin{equation}\label{deltaphi}
\begin{aligned}
&\|\delta\phi^\var\|_{L^1_t(\dot{B}^{\frac{d}{2}+1}_{2,1}\cap\dot{B}^{\frac{d}{2}+2}_{2,1})}\\
&\quad\lesssim\var\|(b-\Delta)^{-1}\partial_{t}\phi^{\var}\|_{L^1_t(\dot{B}^{\frac{d}{2}+1}_{2,1}\cap\dot{B}^{\frac{d}{2}+2}_{2,1})}+\|(b-\Delta)^{-1}\delta\rho^{\var}\|_{L^1_t(\dot{B}^{\frac{d}{2}+1}_{2,1}\cap\dot{B}^{\frac{d}{2}+2}_{2,1})}\\
&\quad\lesssim   \var+\|\delta \rho^{\var}\|_{L^1_t(\dot{B}^{\frac{d}{2}+1}_{2,1})},
\end{aligned}
\end{equation}
which together with (\ref{r2}) and (\ref{boundsRelax}) yields
\begin{equation}\nonumber
\begin{aligned}
 & \|\delta \rho^{\var}\nabla\phi^\var\|_{L^1_t(\dot{B}^{\frac{d}{2}}_{2,1})}+\|(\rho^*-\bar{\rho})\nabla\delta\phi^\var\|_{L^1_t(\dot{B}^{\frac{d}{2}}_{2,1})} \\
 &\quad\lesssim \|\delta\rho^\var\|_{L^2_t(\dot{B}^{\frac{d}{2}}_{2,1})}\|\phi^\var-\bar{\phi}\|_{L^2_t(\dot{B}^{\frac{d}{2}+1}_{2,1})}+\|\rho^*-\bar{\rho}\|_{L^\infty_t(\dot{B}^{\frac{d}{2}}_{2,1})}\|\delta\phi^\var\|_{L^1_t(\dot{B}^{\frac{d}{2}+1}_{2,1})}\\
 &\quad\lesssim \mathcal{X}_{0}\big(\|\delta \rho^{\var}\|_{\widetilde{L}^\infty_t(\dot{B}^{\frac{d}{2}-1}_{2,1})}+\|\delta \rho^{\var}\|_{L^1_t(\dot{B}^{\frac{d}{2}+1}_{2,1})} \big)+\var.
\end{aligned}
\end{equation}
Thence, gathering all the above estimates, we obtain
\begin{equation}\nonumber
\begin{aligned}
  &  \|\delta \rho^{\var}\|_{L^\infty_t(\dot{B}^{\frac{d}{2}-1}_{2,1})}+\|(\delta \rho^{\var}, \delta\phi^{\var})\|_{L^{1}_{t}(\dot{B}^{\frac{d}{2}+1}_{2,1})}\\
    &\quad\lesssim \big(\mathcal{X}_{0}+\|\rho_{0}^{*}-\bar{\rho}\|_{\dot{B}^{\frac{d}{2}}_{2,1}} \big)\big(\|\delta \rho^{\var}\|_{\widetilde{L}^\infty_t(\dot{B}^{\frac{d}{2}-1}_{2,1})}+\|\delta \rho^{\var}\|_{L^1_t(\dot{B}^{\frac{d}{2}+1}_{2,1})}\big)+ \var.
\end{aligned}
\end{equation}
Thus, with the help of the smallness of $\mathcal{X}_{0}+\|\rho_{0}^{*}-\bar{\rho}\|_{\dot{B}^{\frac{d}{2}}_{2,1}}$, we conclude that
\begin{equation}\label{delta1l}
\begin{aligned}
  &  \|\delta \rho^{\var}\|_{\widetilde{L}^\infty_t(\dot{B}^{\frac{d}{2}-1}_{2,1})}+\|\delta \rho^{\var}\|_{L^{1}_{t}(\dot{B}^{\frac{d}{2}+1}_{2,1})}+\|\delta\phi^\var\|_{L^1_t(\dot{B}^{\frac{d}{2}+1}_{2,1}\cap\dot{B}^{\frac{d}{2}+2}_{2,1})}\lesssim \var.
\end{aligned}
\end{equation}
Furthermore, using $(\ref{boundsRelax})_{3}$, (\ref{deltaphi})-\eqref{delta1l}, (\ref{F3}) and $\displaystyle\delta u^{\var}= v^{\var}-\nabla \int_{\rho_{*}}^{\rho^{\var}}\frac{P'(s)}{s}ds+\mu\nabla\delta\phi^{\var}$, we get
\begin{equation}\label{delta2l}
\begin{aligned}
&\|\delta u^{\var}\|_{L^{1}_{t}(\dot{B}^{\frac{d}{2}}_{2,1})}\lesssim \|v^{\var}\|_{L^{1}_{t}(\dot{B}^{\frac{d}{2}}_{2,1})}+\|\delta \rho^{\var}\|_{L^{1}_{t}(\dot{B}^{\frac{d}{2}+1}_{2,1})}+\|\delta \phi^{\var}\|_{L^{1}_{t}(\dot{B}^{\frac{d}{2}+1}_{2,1})}\lesssim \var.
\end{aligned}
\end{equation}
By (\ref{delta1l})-(\ref{delta2l}), (\ref{error}) holds. Consequently, as $\var\rightarrow0$, the global solution $(\rho^{\var},u^{\var},\phi^{\var})$ for System (\ref{HPCvar}) converges strongly to the global solution $(\rho^{*},u^{*},\phi^{*})$ for System (\ref{KS}) in the sense (\ref{conver}).

\section{Appendix}\label{section6}

\subsection{Functional spaces}\label{sec:func-spaces}

Before stating our main results, we explain the notations and definitions used throughout this paper. $C>0$ denotes a constant independent of $\var$ and time, $f\lesssim g~(\text{resp.}\;f\gtrsim g)$ means $f\leq Cg~(\text{resp.}\;f\geq Cg)$, and $f\sim g$ stands for $f\lesssim g$ and $f\gtrsim g$. For any Banach space $X$ and the functions $f,g\in X$, let $\|(f,g)\|_{X}:=\|f\|_{X}+\|g\|_{X}$. For any $T>0$ and $1\leq \varrho\leq\infty$, we denote by $L^{\varrho}(0,T;X)$ the set of measurable functions $g:[0,T]\rightarrow X$ such that $t\mapsto \|g(t)\|_{X}$ is in $L^{\varrho}(0,T)$ and write $\|\cdot\|_{L^{\varrho}(0,T;X)}:=\|\cdot\|_{L^{\varrho}_{T}(X)}$.  $\mathcal{F}$ and $\mathcal{F}^{-1}$ stand for the Fourier transform and its inverse. In addition, let $\Lambda^{\sigma}$ be defined by $\Lambda^{\sigma}f:=\mathcal{F}^{-}(|\xi|^{\sigma}\mathcal{F}f)$.

We recall the notations of the Littlewood-Paley decomposition and Besov spaces. The reader can refer to \cite{bahouri1}[Chapter 2] for a complete overview. Choose a smooth radial non-increasing function $\chi(\xi)$ with compact supported in $B(0,\dfrac{4}{3})$ and $\chi(\xi)=1$ in $B(0,\dfrac{3}{4})$ such that
$$
\varphi(\xi):=\chi(\frac{\xi}{2})-\chi(\xi),\quad \sum_{j\in \mathbb{Z}}\varphi(2^{-j}\cdot)=1,\quad \text{{\rm{Supp}}}~ \varphi\subset \{\xi\in \mathbb{R}^{d}~|~\frac{3}{4}\leq |\xi|\leq \frac{8}{3}\}.
$$
For any $j\in \mathbb{Z}$, the homogeneous dyadic blocks $\dot{\Delta}_{j}$ and the low-frequency cut-off operator $\dot{S}_{j}$ are defined by
$$
\dot{\Delta}_{j}u:=\mathcal{F}^{-1}(\varphi(2^{-j}\cdot )\mathcal{F}u),\quad\quad \dot{S}_{j}u:= \mathcal{F}^{-1}( \chi (2^{-j}\cdot) \mathcal{F} u).
$$
 From now on, we use the shorthand notation $\dot{\Delta}_{j}u=u_{j}.$

Let $\mathcal{S}_{h}'$ be the set of tempered distributions on $\mathbb{R}^{d}$ such that every $u\in \mathcal{S}_{h}'$ satisfies $u\in \mathcal{S}'$ and $\lim_{j\rightarrow-\infty}\|\dot{S}_{j}u\|_{L^{\infty}}=0$. Then one has
\begin{equation}\nonumber
\begin{aligned}
&u=\sum_{j\in \mathbb{Z}}u_{j}\quad\text{in}~\mathcal{S}',\quad\quad \dot{S}_{j}u= \sum_{j'\leq j-1}u_{j'},\quad \forall u\in \mathcal{S}_{h}',\\
\end{aligned}
\end{equation}

With the help of these dyadic blocks, the homogeneous Besov space $\dot{B}^{s}_{p,r}$ for any $p,r\in[1,\infty]$ and $s\in \mathbb{R}$ is defined by
$$
\dot{B}^{s}_{p,r}:=\{u\in \mathcal{S}_{h}'~|~\|u\|_{\dot{B}^{s}_{p,r}}:=\|\{2^{js}\|u_{j}\|_{L^{p}}\}_{j\in\mathbb{Z}}\|_{l^{r}}<\infty\}.
$$
We denote the Chemin-Lerner type space $\widetilde{L}^{\varrho}(0,T;\dot{B}^{s}_{p,r})$ for any $\varrho,r, q\in[1,\infty]$, $s\in\mathbb{R}$ and $T>0$:
$$
\widetilde{L}^{\varrho}(0,T;\dot{B}^{s}_{p,r}):=\{u\in L^{\varrho}(0,T;\mathcal{S}'_{h})~|~ \|u\|_{\widetilde{L}^{\varrho}_{T}(\dot{B}^{s}_{p,r})}:=\|\{2^{js}\|u_{j}\|_{L^{\varrho}_{T}(L^{p})}\}_{j\in\mathbb{Z}}\|_{l^{r}}<\infty\}.
$$
By the Minkowski inequality, it holds that
\begin{equation}\nonumber
\begin{aligned}
&\|u\|_{\widetilde{L}^{\varrho}_{T}(\dot{B}^{s}_{p,r})}\leq \|u\|_{L^{\varrho}_{T}(\dot{B}^{s}_{p,r})}\:\:\text{for}\:\: r\geq\varrho \quad \text{and}\quad \|u\|_{\widetilde{L}^{\varrho}_{T}(\dot{B}^{s}_{p,r})}\geq \|u\|_{L^{\varrho}_{T}(\dot{B}^{s}_{p,r})} \:\:\text{for}\:\: r\leq\varrho,
\end{aligned}
\end{equation}
where $\|\cdot\|_{L^{\varrho}_{T}(\dot{B}^{s}_{p,r})}$ is the usual Lebesgue-Besov norm. Moreover, we write
\begin{equation}\nonumber
\begin{aligned}
&\mathcal{C}_{b}(\mathbb{R}_{+};\dot{B}^{s}_{p,r}):=\{u\in\mathcal{C}(\mathbb{R}_{+};\dot{B}^{s}_{p,r})~|~\|f\|_{\widetilde{L}^{\infty}(\mathbb{R}_{+};\dot{B}^{s}_{p,r})}<\infty\}.
\end{aligned}
\end{equation}

In order to restrict our Besov norms to the low and high frequencies regions, we set an integer $J_{\var}$, called threshold, to denote the following notations for $p\in[1,\infty]$ and $s\in\mathbb{R}$:
\begin{equation}\nonumber
\begin{aligned}
&\|u\|_{\dot{B}^{s}_{p,r}}^{\ell,J_{\var}}:=\|\{2^{js}\|u_{j}\|_{L^{p}}\}_{j\leq J_{\var}}\|_{\ell^{r}},\quad\quad\quad\quad~ \|u\|_{\dot{B}^{s}_{p,r}}^{h,J_{\var}}:=\|\{2^{js}\|u_j\|_{L^{p}}\}_{j\geq J_{\var}-1}\|_{\ell^{r}},\\
&\|u\|_{\widetilde{L}^{\varrho}_{T}(\dot{B}^{s}_{p,r})}^{\ell,J_{\var}}:=\|\{2^{js}\|u_j\|_{L^{\varrho}_{T}(L^{p})}\}_{j\leq J_{\var}}\|_{\ell^{r}},\quad \|u\|_{\widetilde{L}^{\varrho}_{T}(\dot{B}^{s}_{p,r})}^{h,J_{\var}}:=\|\{2^{js}\|u_j\|_{L_{T}^{\varrho}(L^{p})}\}_{j\geq J_{\var}-1}\|_{\ell^{r}}.
\end{aligned}
\end{equation}
For any $u\in\mathcal{S}'_{h}$, we also define the low-frequency part $u^{\ell,J_{\var}}$ and the high-frequency part $u^{h,J_{\var}}$ by
$$
u^{\ell,J_{\var}}:=\sum_{j\leq J_{\var}-1}u_{j},\quad\quad u^{h,J_{\var}}:=u-u^{\ell,J_{\var}}=\sum_{j\geq J_{\var}}u_{j}.
$$
It is easy to check for any $s'>0$ that
\begin{equation}\label{lh}
\begin{aligned}
&\|u^{\ell,J_{\var}}\|_{\dot{B}^{s}_{p,r}}\leq \|u\|_{\dot{B}^{s}_{p,r}}^{\ell,J_{\var}}\leq 2^{J_{\var}s'}\|u\|_{\dot{B}^{s-s'}_{p,r}}^{\ell,J_{\var}},\quad\quad\|u^{h,J_{\var}}\|_{\dot{B}^{s}_{p,r}}\leq \|u\|_{\dot{B}^{s}_{p,r}}^{h,J_{\var}}\leq 2^{-(J_{\var}-1)s'}\|u\|_{\dot{B}^{s+s'}_{p,r}}^{h,J_{\var}}.
\end{aligned}
\end{equation}
Whenever the value of $J_{\var}$ is clear from the context (which will also be the case after),  the exponent $J_{\var}$ will be omitted in the above notations for simplicity. Note also that $J_{\var}$ will depend on $\var$ which implies that all the frequency-restricted norms also depend on $\var$. 

\subsection{Technical lemmas}

We recall some basic properties of Besov spaces and product estimates which will be used repeatedly in this paper. The reader can refer to \cite[Chapters 2-3]{bahouri1} for more details. Remark that all the properties remain true for the Chemin--Lerner type spaces, up to the modification of the regularity exponent $s$ according to Hölder's inequality for the time variable.

The first lemma pertains to so-called Bernstein inequalities.
\begin{lemma}[\!\!\cite{bahouri1}]\label{lemma61}
Let $0<r<R$, $1\leq p\leq q\leq \infty$ and $k\in \mathbb{N}$. For any function $u\in L^p$ and $\lambda>0$, it holds
\begin{equation}\nonumber
\left\{
\begin{aligned}
&{\rm{Supp}}~ \mathcal{F}(u) \subset \{\xi\in\mathbb{R}^{d}~| ~|\xi|\leq \lambda R\}\Rightarrow \|D^{k}u\|_{L^q}\lesssim\lambda^{k+d(\frac{1}{p}-\frac{1}{q})}\|u\|_{L^p},\\
&{\rm{Supp}}~ \mathcal{F}(u) \subset \{\xi\in\mathbb{R}^{d}~|~ \lambda r\leq |\xi|\leq \lambda R\}\Rightarrow \|D^{k}u\|_{L^{p}}\sim\lambda^{k}\|u\|_{L^{p}}.
\end{aligned}
\right.
\end{equation}
\end{lemma}

Due to the Bernstein inequalities, the Besov spaces have many useful properties:
\begin{lemma}[\!\!\cite{bahouri1}]\label{lemma62}
The following properties hold{\rm:}
\begin{itemize}
\item{} For any $s\in\mathbb{R}$, $1\leq p_{1}\leq p_{2}\leq \infty$ and $1\leq r_{1}\leq r_{2}\leq \infty$, it holds
\begin{equation}\nonumber
\begin{aligned}
\dot{B}^{s}_{p_{1},r_{1}}\hookrightarrow \dot{B}^{s-d(\frac{1}{p_{1}}-\frac{1}{p_{2}})}_{p_{2},r_{2}}.
\end{aligned}
\end{equation}
\item{} For any $1\leq p\leq q\leq\infty$, we have the following chain of continuous embedding{\rm:}
\begin{equation}\nonumber
\begin{aligned}
\dot{B}^{0}_{p,1}\hookrightarrow L^{p}\hookrightarrow \dot{B}^{0}_{p,\infty}\hookrightarrow \dot{B}^{\sigma}_{q,\infty} \quad \text{for} \quad\sigma=-d(\frac{1}{p}-\frac{1}{q})<0.
\end{aligned}
\end{equation}
\item{} The following real interpolation property is satisfied for all $1\leq p\leq\infty$, $s_{1}<s_{2}$ and $\theta\in(0,1)${\rm:}
\begin{equation}
\begin{aligned}
&\|u\|_{\dot{B}^{\theta s_{1}+(1-\theta)s_{2}}_{p,1}}\lesssim \frac{1}{\theta(1-\theta)(s_{2}-s_{1})}\|u\|_{\dot{B}^{ s_{1}}_{p,\infty}}^{\theta}\|u\|_{\dot{B}^{s_{2}}_{p,\infty}}^{1-\theta}.\label{inter}
\end{aligned}
\end{equation}
\item{}
For any $\sigma\in \mathbb{R}$, the operator $\Lambda^{\sigma}$ is an isomorphism from $\dot{B}^{s}_{p,r}$ to $\dot{B}^{s-\sigma}_{p,r}$.
\item{} Let $1\leq p_{1},p_{2},r_{1},r_{2}\leq \infty$, $s_{1}\in\mathbb{R}$ and $s_{2}\in\mathbb{R}$ satisfy
    \begin{align}
    s_{2}<\frac{d}{p_{2}}\quad\text{\text{or}}\quad s_{2}=\frac{d}{p_{2}}~\text{and}~r_{2}=1.\label{banach}
    \end{align}
    The space $\dot{B}^{s_{1}}_{p_{1},r_{1}}\cap \dot{B}^{s_{2}}_{p_{2},r_{2}}$ endowed with the norm $\|\cdot \|_{\dot{B}^{s_{1}}_{p_{1},r_{1}}}+\|\cdot\|_{\dot{B}^{s_{2}}_{p_{2},r_{2}}}$ is a Banach space and satisfies weak compact and Fatou properties: If $u_{n}$ is a uniformly bounded sequence of $\dot{B}^{s_{1}}_{p_{1},r_{1}}\cap \dot{B}^{s_{2}}_{p_{2},r_{2}}$, then an element $u$ of $\dot{B}^{s_{1}}_{p_{1},r_{1}}\cap \dot{B}^{s_{2}}_{p_{2},r_{2}}$ and a subsequence $u_{n_{k}}$ exist such that
    \begin{equation}\nonumber
    \begin{aligned}
    \lim_{k\rightarrow\infty}u_{n_{k}}=u\quad\text{in}\quad\mathcal{S}'\quad\text{and}\quad\|u\|_{\dot{B}^{s_{1}}_{p_{1},r_{1}}\cap \dot{B}^{s_{2}}_{p_{2},r_{2}}}\lesssim \liminf_{n_{k}\rightarrow \infty} \|u_{n_{k}}\|_{\dot{B}^{s_{1}}_{p_{1},r_{1}}\cap \dot{B}^{s_{2}}_{p_{2},r_{2}}}.
    \end{aligned}
    \end{equation}
\end{itemize}
\end{lemma}

The following Morse-type product estimates in Besov spaces play a fundamental role in our analysis of nonlinear terms:
\begin{lemma}[\!\!\cite{bahouri1}]\label{lemma63}
The following statements hold:
\begin{itemize}
\item{} Let $p,r\in[1,\infty]$ and $s>0$. Then $\dot{B}^{s}_{p,r}\cap L^{\infty}$ is a algebra and
    \begin{equation}\label{uv1}
\begin{aligned}
\|uv\|_{\dot{B}^{s}_{p,r}}\lesssim \|u\|_{L^{\infty}}\|v\|_{\dot{B}^{s}_{p,r}}+ \|v\|_{L^{\infty}}\|u\|_{\dot{B}^{s}_{p,r}}.
\end{aligned}
\end{equation}
\item{}  For $p,r\in[1,\infty]$ and $s\in(-\min\{\frac{d}{p},\frac{d(p-1)}{p}\}, \frac{d}{p}]$, there holds
\begin{equation}\label{uv2}
\begin{aligned}
&\|uv\|_{\dot{B}^{s}_{p,r}}\lesssim \|u\|_{\dot{B}^{\frac{d}{p}}_{p,1}}\|v\|_{\dot{B}^{s}_{p,r}}.
\end{aligned}
\end{equation}
\item{}  If $p\in[1,\infty]$ and $s\in(-\min\{\frac{d}{p},\frac{d(p-1)}{p}\}, \frac{d}{p}]$, then it holds 
\begin{equation}\label{uv3}
\begin{aligned}
&\|uv\|_{\dot{B}^{s}_{p,\infty}}\lesssim \|u\|_{\dot{B}^{\frac{d}{p}}_{p,1}}\|v\|_{\dot{B}^{s}_{p,\infty}}.
\end{aligned}
\end{equation}
\end{itemize}
\end{lemma}

The following commutator estimates are used to control some nonlinearities in high frequencies:

\begin{lemma}[\!\!\cite{bahouri1}]\label{lemmacom}
Let $p\in[1,\infty]$ and $s\in[-\frac{d}{p}-1, \frac{d}{p}+1]$. Then it holds
\begin{align}
&\sum_{j\in\mathbb{Z}}2^{js}\|\dot{S}_{j-1}u\mathcal{A}v_{j}-\dot{\Delta}_{j}(u\mathcal{A}v)\|_{L^{p}}\lesssim\|u\|_{\dot{B}^{\frac{d}{2}+1}_{2,1}}\|v\|_{\dot{B}^{s}_{p,1}},\label{commutator}
\end{align}
where $\mathcal{A}$ is either $\nabla$ or $\div$.
\end{lemma}

We recall the classical estimates about the continuity for composition of functions:
\begin{lemma}[\!\!\cite{bahouri1,danchin2}]\label{lemma64}
Let $p,r\in[1,\infty]$, $s>0$ and $F\in C^{\infty}(\mathbb{R})$. Then for any $f\in\dot{B}^{s}_{p,r}\cap L^{\infty}$, there exists a constant $C_{f}>0$ depending only on $\|f\|_{L^{\infty}}$, $F$, $s$, $p$ and $d$ such that
\begin{equation}
\begin{aligned}
\|F(f)-F(0)\|_{\dot{B}^{s}_{p,r}}\leq C_{f}\|f\|_{\dot{B}^{s}_{p,r}}.\label{F1}
\end{aligned}
\end{equation}
In addition, if $-\frac{d}{p}<s\leq \frac{d}{p}$ and $f_{1}, f_{2}\in\dot{B}^{s}_{p,r}\cap \dot{B}^{\frac{d}{p}}_{p,1}$, then we have
\begin{equation}
\begin{aligned}
\|F(f_{1})-F(f_{2})\|_{\dot{B}^{s}_{p,1}}\leq C_{f_{1},f_{2}}(1+\|(f_{1},f_{2})\|_{\dot{B}^{\frac{d}{2}}_{2,1}})\|f_{1}-f_{2}\|_{ \dot{B}^{s}_{p,1}},\label{F3}
\end{aligned}
\end{equation}
where the constant $C_{f_{1},f_{2}}>0$ depends only on $\|(f_{1},f_{2})\|_{L^{\infty}}$, $F$, $s$, $p$ and $d$.

\end{lemma}

In order to control the nonlinear term $H(n)$, we prove some estimates concerning the continuity of composition by quadratic functions.

\begin{lemma}\label{compositionlp} Let $J_{\var}$ be given by {\rm{(\ref{J0})}}, and $F(f)$ be any smooth function. Then  it holds that
\begin{align}
&\|F(f)-F(0)-F'(0)f\|_{\dot{B}^{s}_{2,1}}^{\ell}\leq C_{f} \|f\|_{L^{\infty}}\big{(} \|f\|_{\dot{B}^{s}_{2,1}}^{\ell}+\var^{\sigma-s}\|f\|_{\dot{B}^{\sigma}_{2,1}}^{h}\big{)},\quad\quad s>0,\quad\sigma\geq0,\label{q1}\\
&\|F(f)-F(0)-F'(0)f\|_{\dot{B}^{s}_{2,1}}^{h}\leq C_{f}\|f\|_{L^{\infty}}\big{(} \var^{\sigma-s}\|f\|_{\dot{B}^{\sigma}_{2,1}}^{\ell}+\|f\|_{\dot{B}^{s}_{2,1}}^{h}\big{)},\quad\quad s>0,\quad \sigma\in\mathbb{R},\label{q2}
\end{align}
where $C_{f}>0$ denotes a constant dependent of $\|f\|_{L^{\infty}}$, $F''$, $s$, $\sigma$ and $d$.

Furthermore, we have
\begin{align}
&\|F(f)-F(0)-F'(0)f\|_{\dot{B}^{s}_{2,\infty}}^{\ell}\leq C_{f}\|f\|_{\dot{B}^{\frac{d}{2}}_{2,1}}\big{(} \|f\|_{\dot{B}^{s}_{2,\infty}}^{\ell}+\var^{\frac{d}{2}+1-s}\|f\|_{\dot{B}^{\frac{d}{2}+1}_{2,1}}^{h}\big{)},\quad\quad s\geq-\frac{d}{2}.\label{q5}
\end{align}
\end{lemma}
\begin{proof}
 First, we adapt a similar proof as in \cite{chen1} to prove (\ref{q1}). Using Taylor formula at order 2 yields
\begin{equation}\label{com1}
\begin{aligned}
&F(f)-F(0)-F'(0)f=\sum_{k'\in\mathbb{Z}} (F(\dot{S}_{k'+1}f)-F(\dot{S}_{k'}f))-F'(0)f\\
&\quad\quad\quad\quad\quad\quad\quad\quad\quad~~=\sum_{k'\in\mathbb{Z}} \big{(} m_{1,k'} \dot{S}_{k'}f\dot{\Delta}_{k'}f+m_{2,k'}(\dot{\Delta}_{k'}f)^2\big{)},
\end{aligned}
\end{equation}
with 
\begin{equation}\nonumber
\begin{aligned}
m_{1,k'}:=\int_{0}^{1}\int_{0}^{1} F''(\theta (\dot{S}_{k'}f+\tau\dot{\Delta}_{k'}f))d\tau d\theta,\quad m_{2,k'}:=\int_{0}^{1}\int_{0}^{1} F''(\theta (\dot{S}_{k'}f+\tau\dot{\Delta}_{k'}f)) \tau d\tau d\theta.
\end{aligned}
\end{equation}
Owing to the Bernstein lemma, we deduce for any $q\geq0$ that
\begin{equation}\label{m1aa}
\begin{aligned}
&\|\dot{\Delta}_{k}(m_{1,k'} \dot{S}_{k'} f\dot{\Delta}_{k'}f)\|_{L^2}\leq C(1+\|f\|_{L^{\infty}})^{q}\|f\|_{L^{\infty}}2^{(k'-k)q}\|\dot{\Delta}_{k'}f\|_{L^{2}}.
\end{aligned}
\end{equation}
It holds by (\ref{m1aa}) that
\begin{equation}\nonumber
\begin{aligned}
&\sum_{k\leq J_{\var},k'\in\mathbb{Z}} 2^{ks} \|\dot{\Delta}_{k}(m_{1,k'} \dot{S}_{k'}f \dot{\Delta}_{k'} f)\|_{L^{2}}\\
&\quad\leq  \Big(\sum_{k'\leq k\leq J_{\var}} +\sum_{k\leq k'\leq J_{\var}}+\sum_{k\leq J_{\var}\leq k'}  \Big) 2^{ks} \|\dot{\Delta}_{k}(m_{1,k'} \dot{S}_{k'}f \dot{\Delta}_{k'} f)\|_{L^{2}}\\
&\quad\leq C_{f} \|f\|_{L^{\infty}}\Big( \sum_{k'\leq J_{\var}} 2^{k's}\|\dot{\Delta}_{k'} f\|_{L^{2}}\sum_{k'\geq k} 2^{(k-k')s}+\sum_{k'\leq J_{\var}} 2^{k's} \|\dot{\Delta}_{k'}f\|_{L^{2}}\sum_{k'\leq k}2^{(k-k')(s-[s]-1)}\\
&\quad\quad+\sum_{k\leq J_{\var}} 2^{ks} \sum_{k'\geq J_{\var}}2^{-k'\sigma}2^{k'\sigma}\|\dot{\Delta}_{k'}f\|_{L^{2}} \Big)\\
&\quad\leq C_{f}\|f\|_{L^{\infty}}(\|f\|_{\dot{B}^{s}_{2,1}}^{\ell}+\var^{\sigma-s}\|f\|_{\dot{B}^{\sigma}_{2,1}}^{h}).
\end{aligned}
\end{equation}
The part of $m_{2,k'}$ can be controlled similarly. Thus, (\ref{q1}) holds.

We turn to the proof of \eqref{q2}. Similarly, we begin with the decomposition \eqref{q5}. One obtains by  (\ref{m1aa}) that
\begin{equation}\label{m1ab1}
\begin{aligned}
& \Big(\sum_{J_{\var}-1\leq k\leq k'}  +\sum_{J_{\var}-1\leq  k'\leq k}  \Big) 2^{ks} \|\dot{\Delta}_{k}(m_{1,k'} \dot{S}_{k'}f \dot{\Delta}_{k'} f)\|_{L^{2}}\\
&\quad\leq C_{f}\|f\|_{L^{\infty}}\Big( \sum_{k'\geq J_{\var}-1} 2^{k's}\|\dot{\Delta}_{k'}f\|_{L^2}\sum_{k'\geq k} 2^{(k-k')s}\\
&\quad\quad+\sum_{k'\geq J_{\var}-1} 2^{k's}\|\dot{\Delta}_{k'}f\|_{L^2}\sum_{k'\leq k} 2^{(k-k')(s-[s]-1)} \Big)\\
&\quad\leq C_{f}\|f\|_{L^{\infty}}\|f\|_{\dot{B}^{s}_{2,1}}^{h}.
\end{aligned}
\end{equation}
Concerning the summation on $\{(k,k')~|~ k'\leq J_{\var}-1\leq k\}$, we derive
\begin{equation}\label{m1ab3}
\begin{aligned}
&\sum_{ k'\leq J_{\var}-1\leq k} 2^{ks} \|\dot{\Delta}_{k}(m_{1,k'} \dot{S}_{k'}f \dot{\Delta}_{k'} f)\|_{L^{2}}\\
&\quad\leq C_{f}\|f\|_{L^{\infty}}\sum_{k'\leq J_{\var}-1}2^{k'(s-\sigma)} 2^{k'\sigma}\|\dot{\Delta}_{k'}f\|_{L^{2}}  \sum_{k'\leq k}2^{(k-k')(s-[s]-1) }\\
&\quad \leq C_{f}\var^{\sigma-s} \|f\|_{L^{\infty}} \|f\|_{\dot{B}^{\sigma}_{2,1}}^{\ell},
\end{aligned}
\end{equation}
provided $s\geq \sigma$. For $0<s<\sigma$, it also holds that
\begin{equation}\label{m1ab4}
\begin{aligned}
&\sum_{ k'\leq J_{\var}-1\leq k} 2^{ks} \|\dot{\Delta}_{k}(m_{1,k'} \dot{S}_{k'}f \dot{\Delta}_{k'} f)\|_{L^{2}}\\
&\quad\leq 2^{(J_{\var}-1)(s-\sigma)} \sum_{ k'\leq J_{\var}-1\leq k} 2^{k\sigma} \|\dot{\Delta}_{k}(m_{1,k'} \dot{S}_{k'}f \dot{\Delta}_{k'} f)\|_{L^{2}}\\
&\quad\leq C_{f}\var^{\sigma-s} \|f\|_{L^{\infty}} \|f\|_{\dot{B}^{\sigma}_{2,1}}^{\ell}.
\end{aligned}
\end{equation}
Due to (\ref{m1ab1})-(\ref{m1ab4}), one has
\begin{equation}\nonumber
\begin{aligned}
\sum_{k\geq J_{\var}-1,k'\in\mathbb{Z}} 2^{ks} \|\dot{\Delta}_{k}(m_{1,k'} \dot{S}_{k'}f \dot{\Delta}_{k'} f)\|_{L^{2}}\leq C_{f}\|f\|_{L^{\infty}}(\var^{\sigma-s}\|f\|_{\dot{B}^{s}_{2,1}}^{\ell}+\|f\|_{\dot{B}^{\sigma}_{2,1}}^{h}).
\end{aligned}
\end{equation}
Since the the second part of (\ref{com1}) can be estimated similarly, (\ref{q2}) follows.

We are left with proving (\ref{q5}) for any $s\geq-\frac{d}{2}$. The case of $s>0$ can be derived by similarly arguments as used in (\ref{q1}) and the embedding $\dot{B}^{\frac{d}{2}}_{2,1}\hookrightarrow L^{\infty}$. For $-\frac{d}{2}\leq s\leq 0$, one concludes from (\ref{hl}), (\ref{uv3}), (\ref{F1}), the Bernstein inequality and the fact $F(f)-F(0)-F'(0)f=\tilde{F}(f)f$ for some smooth function $\tilde{F}(f)$ vanishing at $0$ that
\begin{equation}\nonumber
\begin{aligned}
&\|F(f)-F(0)-F'(0)f\|_{\dot{B}^{s}_{2,\infty}}^{\ell}\leq\|\tilde{F}(f)\|_{\dot{B}^{\frac{d}{2}}_{2,1}}\|f\|_{\dot{B}^{s}_{2,\infty}}\leq C_{f} \|f\|_{\dot{B}^{\frac{d}{2}}_{2,1}}\big{(} \|f\|_{\dot{B}^{s}_{2,\infty}}^{\ell}+\var^{\frac{d}{2}+1-s}\|f\|_{\dot{B}^{\frac{d}{2}+1}_{2,1}}^{h}\big{)}.
\end{aligned}
\end{equation}
The proof of Lemma \ref{compositionlp} is complete.
\end{proof}

\vspace{2ex}

\noindent
\textbf{Acknowledgments} The authors are indebted to the anonymous referees for their valuable comments on the manuscript. The authors would like to thank Prof. Rapha${\rm\ddot{e}}$l Danchin and Prof. Hai-Liang Li for their helpful suggestions. The first author is partially  supported  by the ANR project INFAMIE (ANR-15-CE40-0011) and by the European Research Council (ERC) under the European Union's Horizon 2020 research and innovation programme (grant agreement NO: 694126-DyCon). Q. He and L.-Y. Shou are  partially supported by the National Natural Science Foundation of China (No.11931010, 12226327 and 12226326).

\end{document}